\documentclass[11pt]{amsart}
\usepackage{amsmath,amsthm,amsfonts,amssymb,enumerate,color,url,hyperref}
\usepackage[all]{xy}

\usepackage[top=30mm,right=25mm,bottom=30mm,left=25mm]{geometry}

\newcommand{\C}{\mathbb{C}}
\newcommand{\CB}{\mathbf{C}}
\newcommand{\F}{\mathbb{F}}
\newcommand{\N}{\mathbb{N}}
\newcommand{\Z}{\mathbb{Z}}
\newcommand{\R}{\mathbb{R}}
\newcommand{\Q}{\mathbb{Q}}

\newcommand{\GL}{\mathrm{GL}}
\newcommand{\SL}{\mathrm{SL}}
\newcommand{\SU}{\mathrm{SU}}
\newcommand{\GU}{\mathrm{GU}}

\newcommand{\Sp}{\mathrm{Sp}}

\newcommand{\PSL}{\mathrm{PSL}}
\newcommand{\PSU}{\mathrm{PSU}}
\newcommand{\GO}{\mathrm{O}}
\newcommand{\SO}{\mathrm{SO}}

\newcommand{\Stab}{\mathrm{Stab}}

\newcommand{\Frob}{\mathrm{Frob}}

\newcommand{\height}{\mathrm{ht}}

\newcommand{\Cl}{\mathrm{Cl}}

\newcommand{\cE}{\mathcal{E}}

\newcommand{\cO}{\mathcal{O}}

\newcommand{\cG}{\mathcal{G}}
\newcommand{\cT}{\mathcal{T}}

\newcommand{\St}{\mathsf {St}}
\newcommand{\ZB}{\mathbf{Z}}
\newcommand{\Irr}{\mathrm{Irr}}
\newcommand{\len}{\mathsf{len}}
\newcommand{\lev}{\mathsf{lev}}
\newcommand{\rank}{\mathsf{rank}}
\newcommand{\defect}{\mathsf{def}}
\newcommand{\nint}{\mathsf{nint}}

\newcommand{\e}{\epsilon}

\newcommand{\supp}{\mathsf {supp}}
\newcommand{\diag}{\mathrm {diag}}

\newcommand{\SSS}{\mathsf{S}}
\newcommand{\WWW}{\mathsf{W}}

\newcommand{\sgn}{\mathsf{sgn}}
\newcommand{\tw}[1]{{}^#1\!}

\renewcommand{\emptyset}{\varnothing}
\renewcommand{\setminus}{\smallsetminus}
\renewcommand{\max}{{\sf max}}
\newcommand{\cl}{\mathfrak{l}}

\newcommand{\varep}{\varepsilon}
\newcommand{\al}{\alpha}
\newcommand{\eps}{\epsilon}
\newcommand{\Ind}{\mathrm{Ind}}
\newcommand{\Ker}{\mathrm{Ker}}
\newcommand{\DC}{D^\circ}
\newcommand{\Id}{\mathrm{Id}}

\newcommand{\editnew}[1]{{\color{blue} #1}}

\newtheorem{thm}{Theorem}[section]

\newtheorem{theorema}{Theorem}

\newtheorem{cor}[thm]{Corollary}

\newtheorem{prop}[thm]{Proposition}

\newtheorem{exa}[thm]{Example}

\newtheorem{lem}[thm]{Lemma}

\newtheorem*{claim*}{Claim}

\numberwithin{equation}{section}

\begin{document}
\title[Character estimates and Thompson's conjecture]
{Character estimates for finite classical groups and the Asymptotic Thompson Conjecture}

\author{Michael Larsen}
\address{Department of Mathematics\\
    Indiana University\\
    Bloomington, IN 47405\\
    U.S.A.}
\email{mjlarsen@indiana.edu}

\author{Pham Huu Tiep}
\address{Department of Mathematics\\ Rutgers University\\ Piscataway, NJ 08854\\ U.S.A.}
\email{tiep@math.rutgers.edu}

\thanks{The first named author was partially supported by the NSF 
grants DMS-2001349 and DMS-2401098 and the Simons Foundation. The second named author gratefully acknowledges the support of the NSF (grant DMS-2200850) and the Joshua Barlaz Chair in Mathematics.}

\thanks{We are grateful to Evgeny Khukhro for pointing out the reference \cite{KN} as (presumably) the earliest appearance in print of Thompson's conjecture..} 

\begin{abstract}
If $G$ is a finite classical group, linear or unitary in any characteristic, and orthogonal in odd characteristic, we give an
approximate formula for $\chi(g)$ in which the error term is much smaller than the estimate, when $g\in G$ is an element with large centralizer and $\chi\in \Irr(G)$ is an irreducible character of low degree.  As an application, we prove Thompson's conjecture for all sufficiently large finite simple groups: each such group contains a conjugacy class whose square is the whole group.
\end{abstract}

\maketitle

\tableofcontents

\section{Introduction}
Let $G$ be a finite non-abelian simple group or a closely related group such as $\GL_N(q)$.  Let $g\in G$ be an element, and let $\chi\in \Irr(G)$ be an irreducible character.
There are many upper bounds for $|\chi(g)|$ in the literature, whose quality typically depend on the degree $\chi(1)$ and the size of the centralizer of $g$.
See \cite{La,Li} and the references contained therein for the current state of work on this problem and for applications of such bounds.
For applications, an important role is played by the Frobenius formula, which asserts 
that the number of representations of $g\in G$ as a product of a conjugate of $x$ and a conjugate of $y$ is
\begin{equation}
\label{Frobenius}
\frac{|x^G||y^G|}{|G|}\sum_{\chi\in \Irr(G)} \frac{\chi(x)\chi(y)\bar\chi(g)}{\chi(1)}.
\end{equation}
In many situations of interest, the contribution of the $\chi=1_G$ term dominates the sum, so that the approximation of treating the product of a random conjugate of $x$ and a random conjugate of $y$ as a random element of $G$
is fairly good.  It is often the case that the hardest part in proving this is dealing with the contribution of non-trivial characters of low degree, especially when $g$ has large centralizer.

An examination of the low degree complex irreducible characters of classical groups over finite fields suggests that they naturally fall into families indexed by a (non-negative integral) \emph{level}.
There have been several attempts to formalize this observation \cite{GMST,GH1,GH2,GLT1,GLT2}.
In this paper, we use the concept of \emph{true level} defined in \cite{GLT1} when $G$ is of type $\GL$ or $\GU$, and give a definition of the 
{\it level} $\lev(\chi)$ directly in terms of the Lusztig parameterization of irreducible characters $\chi$ for orthogonal and symplectic groups, see \eqref{lev-gen}.
Roughly speaking, if $G$ is a classical group of dimension $N$ over $\F_q$ and $\chi \in \Irr(G)$ has low degree, then $\lev(\chi) \approx \log_{q^N}\chi(1)$.
Sarnak has suggested to us that one should regard the distinction between low level and high level characters as being akin to the distinction between major and minor arcs in the circle method. This analogy suggests that 
 instead of settling for bounding $\chi(g)$, in low level cases we should aim for an actual estimate, with a small error term.
In this paper we prove such an estimate for classical groups of type A, B, and D over finite fields, in the large centralizer case,
and the low-degree condition means the character has level at most the square root of the largest possible level. 
One can hope that this may be a first step toward developing something like a ``singular series'' in this non-abelian setting.
(Type C can also be treated in a similar way but we do not present details here since this case is not needed for 
our intended application.)

Recall that if $g$ is an $N\times N$-matrix over a field $\F$, we define its \emph{support} $\supp(g)$ to be the codimension of its largest eigenspace
over $\overline{\F}$.  When $\supp(g)<N/2$, there is a unique eigenvalue $\lambda$ of $g$ with an eigenspace of dimension
$N-\supp(g)$, which we call the \emph{primary eigenvalue}. In classical groups, one can directly relate the size of the centralizer of an element $g$ to $\supp(g)$
(see e.g. \cite[\S\S4, 6]{LT3}),
and we formulate our results in terms of the support rather than the order of the centralizer. In the first result, we use the notation $\GL^\eps_N(q)$ to denote 
the general linear group $\GL(N,\F_q)$ when $\eps=+$, and the general unitary group $\GU(N,\F_{q^2})$ when $\eps=-$. For these groups, see \cite[Theorem 3.9]{GLT1} for the determination of the true level of an irreducible character.

\begin{theorema}
\label{mainB}
Let $m$ and $j$ be fixed positive integers, and let $\eps= \pm$.  If 
$$N \geq \max\bigl( 3j^2+j +2m+10,2j^2+4m+5\bigr),$$ 
$\chi$ is an irreducible complex character of $G=\GL^\eps_N(q)$ of true level $j$, and $g\in G$ has support $m$ and primary eigenvalue $\lambda$, 
then
$$\Bigl|\frac{(\eps q)^{mj}\chi(g)}{\chi(1)}\cdot \frac{\overline{\chi(\lambda \cdot \Id_V)}}{\chi(1)}-1\Bigr| < q^{-N/2}.$$
\end{theorema}

\begin{theorema}
\label{mainC}
Let $m,j \in \Z_{\geq 1}$, $q$ any odd prime power, and let $\varep= \pm$. Suppose that
$$N \geq 8j^2+4j +\max(4m,2j)+2$$ 
and $\chi$ is an irreducible complex character of $G=\SO^\varep_N(q)$ of level $j$.
Then there exists some $\beta=\beta(\chi,\lambda) \in \{1,-1\}$ such that when $g\in G$ has support $m$ and primary eigenvalue $\lambda = \pm 1$, we have
$$
\Bigl|\frac{q^{mj}\chi(g)}{\chi(1)}-\beta\Bigr| < q^{-N/4}.$$
Moreover, if $2|N$ or $\lambda=1$, then $\beta = \chi(\lambda \cdot 1_G)/\chi(1)$. 
\end{theorema}

The proofs of Theorems \ref{mainB} and \ref{mainC} are based on Deligne--Lusztig theory \cite{LusztigBook} and Howe duality \cite{Hw}.  
Theorem \ref{mainB} makes use of  \cite[Theorem 1.1]{GLT1}, while
Theorem \ref{mainC} makes use of Pan's work \cite{Pan24,Pan-I,Pan-T}, following work of Aubert, Michel, and Rouquier \cite{AMR96}, explaining the compatibility between the Howe correspondence and Lusztig's classification of irreducible representations of 
orthogonal and symplectic groups. 
As shown in Example \ref{large-level}, 
the constraints bounding $m=\supp(g)$ and $j=\lev(\chi)$ in Theorems~\ref{mainB}, \ref{mainC} cannot be 
completely removed. Thus results like Theorems~\ref{mainB}, \ref{mainC} are purely a phenomenon for characters of low degree and elements of 
not-too-large support.

Before discussing applications of Theorems \ref{mainB} and \ref{mainC}, we make some comments on
how they relate to other recent character bounds for
finite groups of Lie type $G$. The character bound obtained in \cite{LT3} applies uniformly, to all characters and all elements of the group, but because its constant is poor, the bound is non-trivial only when the rank $r$ of the group $G$ in question is very large. The character bounds of \cite{BLST,TT,LT,GLT1,GLT2}
apply to certain elements, or to certain characters of $G$; the bounds in \cite{BLST,TT,LT} work particularly well when the size $q$ of the field over which 
$G$ is defined is large compared to the rank $r$, whereas the bounds in \cite{GLT1,GLT2} typically require the rank $r$ to be large. 
All of these bounds are weakest for elements of small support and characters of low level,
which is exactly where the new results of this paper apply.
In the applicability range
of Theorems \ref{mainB} and \ref{mainC}, they yield not only an upper bound, but an approximate formula for $\chi(g)$ in which the error term is much smaller 
(roughly by a factor of $q^{r/2}$) than the estimate.

\smallskip
As a consequence of Theorems \ref{mainB} and \ref{mainC} and the main result of \cite{LT3}, we can now prove an 
essentially optimal refinement of 
\cite[Theorem 1.2.1]{LST}:

\begin{theorema}
\label{main-supp}
There exists a constant $n_0$ such that the following statement holds for any $N \geq n_0$. If $G=\SL_N(q)$, $\SU_N(q)$, or $G \in\{\SO^\eps_N(q),\Omega^\eps_N(q)\}$ with $2 \nmid q$, $g \in G$ has support $m>0$, and $\chi \in \Irr(G)$ has degree $>1$, then either
\begin{enumerate}[\rm(i)]
\item $|\chi(g)|/\chi(1) < q^{-m}$, or
\item $\chi$ has level $1$; equivalently, $1 < \chi(1) < q^N$. Furthermore, one of the following holds.
\begin{enumerate}[\rm(a)]
\item $G=\SL_N(q)$ and $|\chi(g)|/\chi(1) < q^{1-m}$.
\item $G=\SU_N(q)$ and $|\chi(g)|/\chi(1) < 1.53q^{1-m}$.
\item $G \in \{\SO^\eps_N(q),\Omega^\eps_N(q)\}$ and $|\chi(g)|/\chi(1) < q^{3-m}$.
\end{enumerate}
\end{enumerate}
\end{theorema}

Bounded rank analogues of Theorem \ref{main-supp} will be explored elsewhere, see e.g. \cite{LT4}.

\smallskip
As an application of Theorems~\ref{mainB} and \ref{mainC}, we prove the asymptotic version of Thompson's Conjecture,
which asserts that every finite non-abelian simple group $G$ contains a conjugacy class $C\subset G$ such that $C^2 = G$.
This can be thought of as a strong version of Ore's Conjecture (now a theorem of Liebeck, O'Brien, Shalev, and Tiep \cite{LOST}) that every element in $G$ is 
a commutator.  
According to a private communication by Khukhro, the conjecture first appeared  in Kourovka's Notebook \cite{KN} as Problem 9.24 in 1984. 
It was communicated to Mazurov by Thompson in Oberwolfach in 1982, and although Thompson did not want to claim authorship, he consented to Mazurov 
describing it as ``Thompson's problem''. Within a year \cite{AH} it had achieved its modern name.
In 1998, Ellers and Gordeev \cite{EG} reduced the conjecture to the case of finite simple groups of Lie type over fields of order $\le 8$.  They also handled completely the groups of the form $\PSL_n(q)$.

In \cite{LT3}, we proved Thompson's conjecture for all sufficiently large finite simple groups of symplectic type in any characteristic,
and also for orthogonal groups in characteristic $2$.
In this paper, we treat the cases of  unitary groups and of orthogonal groups over fields of odd order, in sufficiently high rank,
constructing for each such simple group $G$ an {\it explicit} conjugacy class $C\subset G$ such that $G=C^2$,
and thereby completing the proof of the asymptotic version of Thompson's Conjecture:

\begin{theorema}
\label{mainA}
If $G$ is a finite non-abelian simple group of sufficiently high order, then $G$ contains a conjugacy class $C$ satisfying $C^2 = G$.
\end{theorema}

The difficulty in applying \eqref{Frobenius}  to Thompson's Conjecture is that for any choice of $x=y$, there are many different characters $\chi\in \Irr(G)$
for which $\chi(x)=\chi(y)\neq 0$.  (By comparison, it is relatively easy to find distinct $x$ and $y$ in $G$ such that
$\{\chi\in \Irr(G)\mid \chi(x)\chi(y)\neq 0\}$ is very small, often consisting just of the identity character and the Steinberg character,
and this can be used to show that $x^G y^G$ contains every element other than the identity: see \cite{MSW,LST,GM}.)

Nevertheless, in favorable circumstances, the $x=y$ case can be treated by means of \eqref{Frobenius}.
There is a glimmer of hope for this in \cite[Theorem 10.7]{LTT}, where, in joint work with Taylor, we proved that
$C^2\cup \{1_G\}= G$, when $G := \PSL_p(q)$ is sufficiently large,
$p$ is prime, and $C$ is the conjugacy class of an element of $S$ represented by $x\in \SL_p(q)$ of order $\frac{q^p-1}{q-1}$. 
This result is simpler than what we need to do in this paper in several respects:
$\PSL_p(q)$ is easier to deal with than unitary or orthogonal groups; we do not
insist that $1_G\in C^2$, so we are not limited to real classes $C$; and we are only dealing with the case $p$ is prime, so the description of $\{\chi\in\Irr(\SL_p(q))\mid \chi(x)\neq 0\}$, for $x$ of order $\frac{q^p-1}{q-1}$,
is particularly simple: just the cuspidal and unipotent characters.

For Theorem~\ref{mainA}, our strategy is to show that for unitary or orthogonal groups of certain special ranks parametrized by large primes, there exist real, regular semisimple elements $x$
such that for any non-central $g$, the formula \eqref{Frobenius} counting the number of representations of $g$ as a product of two conjugates of $x$ is dominated
by the $\chi=1_G$ term.  This entails showing that $|\chi(g)|$ is substantially smaller than $\chi(1)$ for each non-trivial character $\chi$ (or, at least, all those for which $\chi(x)\neq 0$.)
By  \cite[Theorem 5.5]{LT3}, there is a character ratio bound of the form 

\vskip4pt
\centerline{$\dfrac{|\chi(g)|}{\chi(1)} \le \chi(1)^{-c\cdot\supp(g) /N},$}
\vskip4pt
\noindent
where $N$ is the dimension of the natural representation of $G$, $c$ is a positive absolute constant, and the \emph{support} $\supp(g)$ is the codimension 
of the largest eigenspace of $g$ in the natural representation of $G$.  This bound is useful when either $\log_{q^N}\chi(1)$ or $\supp(g)$  is large.
It implies our result with Taylor \cite[Theorem 1]{LTT} asserting that there exists a universal constant $B$ such that every finite simple group $G$ of classical Lie type has a regular semisimple conjugacy class $C$ such that $C^2$ contains all elements of
support $\ge B$.  The main difficulty in treating elements of small support
is the contribution of characters of low degree.

For Theorem~\ref{mainA}, we need a method to go from groups $G_1$ of special rank, for which we can construct elements $x_1$ with $x_1^{G_1} x_1^{G_1} = G_1$, to the general case.
Suppose, to illustrate the idea, we have a regular semisimple element $x_1$ in $\Omega^+_{2p}(q) = [\SO^+_{2p}(q),\SO^+_{2p}(q)]$ such that every element of $\Omega^+_{2p}(q)$ is a product of two conjugates of $x_1$.  Let $y$ denote a real class in $\Omega^+_{4k}(q)$, where $k$ is small compared to $p$.  Then both $I_{8k}$ and $-I_{8k}$ are
products of two conjugates of $\diag(y,-y)$  in $\Omega_{8k}^+(q)$.  We claim that if $p$ is large enough then every element $g\in \Omega_{2p+8k}^+(q)$ is a product of two conjugates of $x:=\diag(x_1,y,-y)$.  Indeed, unless $g$ has support less than $p$, this follows easily from \eqref{Frobenius}.  If $g$ has small support, its primary eigenvalue
must be $1$ or $-1$, so it is conjugate either to an element of the form  $\diag(g_1,I_{4k},I_{4k})$ or to one of the form $\diag(g_1,-I_{4k},-I_{4k})$. Either way, this is a product of 
two conjugates of $x$ in $\Omega^+_{2p}(q)\times \Omega^+_{8k}(q)$, and therefore of two conjugates of $x$ in $\Omega^+_{2p+8k}(q)$.  

The difficulties in implementing this program for unitary groups and orthogonal groups are very different.  For unitary groups, there are $q+1$ possible primary eigenvalues, so we need special ranks representing all
congruence classes modulo $2q+2$; as a result, the set of special ranks for which the theorem must be proved is relatively large. For orthogonal groups, Lusztig's classification is quite complicated, it is relatively difficult to compute character values, and there are many characters of small level to deal with. Rather than attempting a unified presentation that handles all difficulties simultaneously, we treat the unitary cases  and the orthogonal cases separately.  Sections 2--7 are devoted to the former and sections 8--14 to the latter. 

Although there are similarities in our strategy to attack Thompson's Conjecture between this paper and \cite{LT3}, there are significant differences as well.
Here, the construction of the conjugacy class $C$ is explicit, and the proof that $C^2 = G$ makes use of \eqref{Frobenius} and strong character estimates.  
By comparison, the argument of \cite{LT3} is indirect and depends on Ore's Conjecture \cite{LOST}.
For the groups treated in this paper, \cite{LOST} is not a needed input,  and, indeed, we give a new proof of Ore's Conjecture.
It would be easy to treat the  groups of \cite{LT3} with $q$ odd by the methods of this paper. However, we do not know the relation between Howe duality and Lusztig's parametrization in characteristic $2$, and this prevents us from claiming a complete (independent) proof of asymptotic Ore Conjecture.

Finally, some notational comments. We do not carefully distinguish a representation from its character $\rho$, and write $\rho\sigma$
and $\rho\otimes \sigma$ interchangeably depending on which aspect we want to emphasize. For an $n$-dimensional vector space $V$, we use 
the notation $\Id_V$ (or $\Id_N$) to denote the identity transformation of $V$, whereas $I_n$ denotes the identity $n \times n$-matrix over any field.
For a finite group $G$, $1_G$ denotes the identity element of the group or the principal character, depending on the context.

\section{Some $\SSS_N$-character values}
In this section, we prove that certain elements $g$ of the symmetric group $\SSS_N$ have the property that all irreducible character values at $g$ belong to the set $\{-1,0,1\}$.
This is relevant to us because it allows us to find regular semisimple elements in groups of type A such that every unipotent character value at such an element lies in $\{-1,0,1\}$.

A \emph{partition} $\lambda = \lambda_1\ge \lambda_2\ge\lambda_3\ge\cdots$ will mean a weakly decreasing finitely supported infinite sequence of non-negative integers.  
We let $|\lambda|$ denote the sum $\sum_i \lambda_i$ of the parts of $\lambda$.  We  often use the notation
$\lambda\vdash N$  for $N=|\lambda|$.
For each partition $\lambda\vdash N$, we let $\chi_\lambda$ denote the irreducible
character of $\SSS_N$ associated to $\lambda$. 

We identify a partition $\lambda\vdash N$ with its Ferrers diagram, with row lengths $\lambda_1,\lambda_2,\ldots$.
Each box in the diagram has a \emph{row number} $y\ge 1$ and a \emph{column number} $x\ge 1$ and can be identified by the pair $(x,y)$.
The \emph{layer number} is the smaller of the two, and if the two are equal, we say the box is \emph{diagonal}.
By the \emph{$k^{\mathrm {th}}$ layer}, we mean the set of boxes with layer number $k$; the \emph{size}, that is, the number of boxes, of the $k^{\mathrm {th}}$ layer decreases monotonically with increasing $k$.
The number of diagonal boxes equals the number of (non-empty) layers.  The \emph{distance} between the boxes $(x_1,y_1)$ and $(x_2,y_2)$
is $|x_1-x_2| + |y_1-y_2|$.

A \emph{border strip} in $\lambda$ is a connected skew diagram containing no $2\times 2$ squares.  By connectedness, the size of any border strip containing two boxes must be greater than the distance between the boxes.  The maximum length of a border strip in $\lambda$ is $\lambda_1+\lambda'_1 - 1$, where $\lambda'$ is the transpose of $\lambda$.  If the border strip does not include the last box in the first row of $\lambda$, its size can be no greater than $\lambda_1+\lambda'_1-2$.

A border strip can be thought of as the union of its rows, which must be contained in a consecutive set of rows of $\lambda$, with the property that the first box of row $i$ lies in the same column as the last box of row $i+1$, or as the union of its columns, which must be contained in a consecutive set of columns of $\lambda$ such that the first box of column $i$ lies in the same row as the last box of column $i+1$.

A \emph{border strip tableau of type  $(\alpha_1,\ldots,\alpha_m)$} of  $\lambda$ is a labelling of its boxes by integers in $[1,m]$ such that the 
labels are non-decreasing in each row and column and
such that the boxes labeled by any given $i$ form
a border strip of size $\alpha_i$.  Note that if $x_1<x_2$ and $y_1 < y_2$ then $(x_1,y_1)$ and $(x_2,y_2)$ cannot have the same label $i$, since otherwise, the 
border strip of boxes labeled $i$ would contain the rectangle of which $(x_1,y_1)$ and $(x_2,y_2)$ are opposite corners.  Thus labels are strictly increasing along diagonals of type $(1,1)$.
Regarding the last border strip in a border strip tableau of $\lambda$ as a union of columns, we see that if it contains the last box in the first row of $\lambda$, it is uniquely determined by $\alpha_m$.

The \emph{height} of a tableau $T$ is the sum of the heights of its border strips, where the height of a single border strip is one less than the number of rows
it meets.  By the Murnaghan-Nakayama rule, the value of $\chi_\lambda$ at a cycle $x$ with orbit lengths $\alpha_1,\ldots,\alpha_m$ is $\sum_T (-1)^{\height(T)}$, where the sum
is taken over border strip tableaux of type $(\alpha_1,\cdots,\alpha_m)$.  The order of the $\alpha_i$ in this $m$-tuple may affect the number of terms in the sum but not, of course, the value of the sum.

We say that $\lambda$ is \emph{minimal} for an element $x$ if there is some ordering of the orbit lengths of $x$ so that the number of tableaux of type $(\alpha_1,\ldots,\alpha_m)$ is at most $1$.
This guarantees 
$$\chi_\lambda(x)\in \{-1,0,1\}.$$

\begin{prop}
\label{fast growth}
If $x\in \SSS_N$ has orbits of lengths $\alpha_1\le\cdots \le\alpha_m$ and
$$\alpha_{i+1} \ge 2(\alpha_1+\cdots+\alpha_i)-1$$
for all $1 \leq i<m$, then every $\lambda\vdash N$ is minimal for $x$.  In particular, every irreducible character takes value $0$ or $\pm1$ on $x$.
\end{prop}

\begin{proof}
Since the transpose of any border strip tableau of $\lambda$ is a border strip tableau of the transpose partition $\lambda'$, we may assume
without loss of generality that the number $\lambda_1$ of columns is at least as great as the number $\lambda'_1$ of rows.  

If there are no boxes in the first row which are labeled $m$,
then the last border strip is entirely contained in a rectangle with $\lambda_1$ columns and $\lambda'_1-1$ rows, so
$$\alpha_m \le \lambda_1 + \lambda'_1 - 2\le 2\lambda_1 - 2.$$
Also,
$$\alpha_m \le \lambda_2 + \lambda_3 + \cdots.$$
Adding twice this inequality to the previous inequality, we obtain
$$3\alpha_m \le 2(\lambda_1+\lambda_2+\cdots) - 2 = 2N-2.$$
On the other hand,
$$3\alpha_m = 2\alpha_m + \alpha_m \ge 2\alpha_m + (2\alpha_1+\cdots+2\alpha_{m-1}) - 1 = 2N-1.$$
The contradiction implies that the first row must have some box labeled $m$, and this implies that the last box in the first row is so labeled.
We conclude that the $m$th border strip is uniquely determined by $\alpha_m$, and this determines the partition obtained from $\lambda$ by removing the $m$th border strip.
By induction on $m$, there can be at most one border strip tableau of $\lambda$ of type $(\alpha_1,\ldots,\alpha_m)$.
\end{proof}

Let $a < N/3$ be a positive integer,
and let 
\begin{equation}\label{xy}
  x = (1\,2\,\cdots\,a)\,(a+1\;a+2\,\cdots\,N),\ y = (1)\,(2\,3\,\cdots\,a)\,(a+1\;a+2\,\cdots\,N).
\end{equation}

\begin{prop}\label{value1}
Every partition $\lambda\vdash N$ is minimal for the element $x$ in \eqref{xy}.  The total number of $\lambda$ such that 
$\chi_\lambda(x) \neq 0$ is $a(N-a)$. 
For any fixed $b$, the number of  $\lambda$ for which $\lambda_1 = N-b$ and  $\chi_\lambda(x) \neq 0$ is at most $1+a(a+1)/2$.
If $0 \leq b<a$, the only $\lambda$ with $\lambda_1 = N-b$ and
$\chi_\lambda(x)\neq 0$ is the hook partition $\lambda = (N-b,1,\ldots,1)$. \end{prop}

\begin{proof}
By Proposition~\ref{fast growth}, $\lambda$ is always minimal for $x$.  Therefore, 
$$a(N-a) = |\CB_{\SSS_N}(x)| = \sum_{\chi\in \Irr(\SSS_N)} |\chi(x)|^2 = |\{\chi\in\Irr(\SSS_N)\mid \chi(x)\neq 0\}|.$$

Next, we fix $b = N-\lambda_1$.  Each border strip can contain at most one diagonal box, so $\chi_\lambda(x)=0$ unless $\lambda$ has at most two layers.
Assuming there are $\le 2$ layers, the size $N-k$ of the first layer is at least $N-a$, and, given $b$, this size determines the boxes in the first layer.  For $k=0$, this determines $\lambda$ uniquely.  Otherwise, there are at most $k$ possibilities for the second layer, so the total number of possibilities is at most $1+a(a+1)/2$.

Henceforth, we assume $b<a$.
Suppose first that $\lambda$ has one layer, which implies that $\lambda = (N-b,1,\ldots,1)$.  In any border tableau of type $(a,N-a)$, the unique diagonal box is labeled $1$, and the boxes labeled $2$ form a connected set, which implies that either they all belong to the first row or all to the first column.  However, the first column has $b+1\le a < N/3 < N-a$ boxes, 
so the boxes labeled $2$ must be exactly the last $N-a$ boxes of the first row.
Note that the first row has $N-b>N-a$ boxes, so this is, indeed, possible.
Therefore, $\chi_\lambda(x) = \pm 1$.

If $\lambda$ has two layers, in any border strip tableau of type $(N-a,a)$, the second diagonal box must be labeled $2$.  It is impossible for the last box in the first row to be labeled $2$ as well, since its distance from the second diagonal box is $N-b > a$.  Therefore, all $N-b$ boxes in the first row are labeled $1$, which is
possible only if $b\ge a$; otherwise, there are no tableaux of this shape at all, and $\chi_\lambda(x) = 0$.
\end{proof}

\begin{prop}\label{value2}
Every $\lambda \vdash N$ is minimal for the element $y$ in \eqref{xy}.
The  number of $\lambda$ such that 
$\chi_\lambda(x) \neq 0$ is $(a-1)(N-a)$, unless $a=2$, in which case it is $2(N-2)$.
For fixed $b$, the number of partitions $\lambda$ with $\lambda_1 = N-b$ and $\chi_\lambda(y)\neq 0$ is at most $a^2$.
If $0 \leq b<a-1$, the only partition $\lambda$ with $\lambda_1 = N-b$ and
$\chi_\lambda(y)\neq 0$ is $(N)$ if $b=0$ and $(N-b, 2, 1,\ldots,1)$ if $b\ge 2$. 
\end{prop}

\begin{proof}
By Proposition~\ref{fast growth}, $\lambda$ is always minimal for $y$.  Therefore, assuming $a\ge 3$
$$(a-1)(N-a) = |C_{\SSS_N}(y)| = \sum_{\chi\in \Irr(\SSS_N)} |\chi(y)|^2 = |\{\chi\in\Irr(\SSS_N)\mid \chi(y)\neq 0\}|.$$
If $a=2$, then $|C_{\SSS_N}(y)|=2(N-2)$.

Next, we assume that $b$ is fixed, and $\lambda_1 = N-b$.  Let $d$ be the number of layers of $\lambda$.  If $d > 3$, there are no border strip tableaux consisting of three strips, so we may assume $d\le 3$.
As in the proof of Proposition~\ref{value1}, we fix $k\in [0,a]$ to be the size of the complement of the first layer of $\lambda$, so 
$k$ and $b$ together determine the first layer.  If $d\le 2$, there are $\le k$ possibilities for the second layer.
If $d=3$ and $\lambda$ has a tableau of type $(a-1,N-a,1)$, then the third diagonal box must be labeled $3$, so all boxes in the third layer must be so labeled,
meaning that the third layer must consist only of its diagonal box.  There are $\le k-1$ possibilities for the second layer in this case.  In total, therefore,
for any given $k$, there are at most $2k-1$ possibilities for $\lambda$ with $\chi_\lambda(y)\neq 0$; summing over $k\in [0,a]$, there are
$\le a^2$ possibilities for $\lambda$.

Finally, we assume that $b<a-1$.
The result is trivial for $b=0$, so we assume $b>0$.

\vskip 5pt\noindent
\textbf{Case $d=1$.}  

As $b<a-1$,  there is a unique border strip of size $N-a$, namely the last $N-a$ entries in the first row.  The complement of this border strip
is the Ferrers diagram of $(a-b,\underbrace{1,\ldots,1}_b)$.  This cannot contain a border strip of size $a-1$ because we have already assumed that $b > 0$ and $a-b > 1$.
Therefore, $\lambda$ is minimal for $y$, and, in fact, $\chi_\lambda(y)=0$.

\vskip 5pt\noindent
\textbf{Case $d=2$.} 

We claim that the only way in which a border tableau of type $(1,a-1,N-a)$ can exist is if the second layer has size $1$.  Note that $b$ is the number of boxes which
do not belong to the first row.  Therefore, the condition $b < a-1$ implies that the number of boxes in the first row and labeled by $2$ is greater than the number of boxes not in the first row and labeled by $3$.  Given any box $(1,y)$ labeled by $2$, the box $(2,y+1)$, if it exists, must be labeled by $3$.  Therefore, there exists some box $(1,y)$ labeled by $2$
such that $\lambda_2 \le y$.
This is only possible if the first box in the first row which is labeled $3$ is the only box in its column in $\lambda$, which, by connectedness, means that the third border strip consists of the last $N-a$ boxes in the first row.  The complement of this border strip has a tableau of type $(1,a-1)$.  Since the longest possible border strip in any Ferrers diagram has length equal to the size of the first layer, this is possible only if the second layer consists only of the $(2,2)$ box.

\vskip 5pt\noindent
\textbf{Case $d=3$.} 

Let $T$ be any tableau of type $(N-a,a-1,1)$.  The existence of $T$ implies there is only one box in the third layer.
Removing this box, we obtain a partition $\mu$ with two layers, and $T$ gives a border strip tableau of $\mu$ of type $(N-a,a-1)$.  As $b<a-1$, by 
the last paragraph of the proof of Proposition~\ref{value1}, there is no such tableau.
 \end{proof}

\section{A sharp character estimate for groups of type A}
Fix $\varepsilon = \pm$, $q$, and $N$, and let $G := G_N := \GL_N^\varepsilon(q)$.
Recall the notion of
the {\it true level} $\cl^*(\chi)$ of any $\chi \in \Irr(G)$ introduced in \cite[Definition 1]{GLT1}.
Let $\tau_N$ denote the reducible Weil character of $G$:
$$\tau_N(g) = \varepsilon^N(\varepsilon q)^{\dim \ker (g-1)},$$
(in which the sign $\eps$ is interpreted as an element of $\{\pm1\}$).
Then $\cl^*(\chi)$ is the smallest integer $j \geq 0$ such that $\chi$ is an irreducible constituent of $\tau^j$.

We now prove the following theorem, which implies Theorem~\ref{mainB} by taking $\kappa = 0$ in part (i):

\begin{thm}\label{slu-bound-main}
Let $m,j \in \Z_{\geq 0}$, $0 \leq \kappa \leq 2$ any fixed real number, and let $\eps= \pm$.  
\begin{enumerate}[\rm(i)]
\item If 
$$N \geq \max\bigl(3j^2+j + (2-\kappa)m +10, 2j^2+(4-\kappa) m+5\bigr),$$ 
$\chi$ is an irreducible complex character of $G=\GL^\eps_N(q)$ of true level $j$, and $g\in G$ has support $m$ and primary eigenvalue $\lambda$, 
then
\begin{equation}
\label{bound1}
  \Bigl|\frac{(\eps q)^{mj}\chi(g)}{\chi(1)}\cdot \frac{\overline{\chi(\lambda \cdot \Id_V)}}{\chi(1)}-1\Bigr| < q^{-(N-\kappa m)/2}.
\end{equation}
\item If $N \geq (3j^2+j)/2+5$, $\chi$ is an irreducible complex character of $G=\GL^\eps_N(q)$ of true level $j \geq 1$, and $g\in G$ has support 
$m \geq j+5$, then
\begin{equation}
\label{bound2}
  \Bigl|\frac{\chi(g)}{\chi(1)}\Bigr| < \left\{\begin{array}{rl}1.53q^{1-m}, &\mbox{if } j=1,~\eps=-,\\
 q^{1-m}, & \mbox{if }j=1,~\eps=+,\\
  q^{-m}, & \mbox{if }j \geq 2.\end{array}\right. 
\end{equation}
\end{enumerate}
\end{thm}

\begin{proof}
We prove the statements by induction on $j\geq 0$, with the induction base $j=0$ being obvious.
The statements are also obvious for $m=0$, so we will assume $m \geq 1$. 

Next we verify \eqref{bound2} directly when $j=1$. Recall that the characters of true level $1$ are among the Weil characters of $G$,
and their values are explicitly known (see e.g. \cite[Lemma 4.1]{TZ1} for $\eps=-$). Recall that $n \geq 7$ and $1 \leq m \leq n-1$. If $\eps=+$, then 
$$\Bigl| \frac{\chi(g)}{\chi(1)}\Bigr| \leq \frac{(q-1)(q^{n-m}+1)}{q^n-q}  
\leq \frac{1}{q^{m-1}}\cdot\frac{q-1}{q}\cdot\frac{q^n+q^{n-1}}{q^n-q}= \frac{1}{q^{m-1}}\cdot \frac{q^2-1}{q^2}\cdot \frac{q^n}{q^n-q}<q^{1-m}.$$ 
If $\eps=-$, then 
$$\Bigl| \frac{\chi(g)}{\chi(1)}\Bigr| \leq \frac{(q+1)q^{n-m}}{q^n-q}  
= \frac{1}{q^{m-1}}\cdot\frac{q+1}{q}\cdot\frac{q^n}{q^n-q} <1.53q^{1-m}.$$ 
The analysis of Weil characters shows that, up to the linear factor in front of $q^{1-m}$, this upper bound is optimal.

For the induction step $j \geq 1$ for (i) and $j \geq 2$ for (ii), we will follow in part the proof of \cite[Theorem 1.6]{GLT1}. 
Note that we always have $N \geq m+1$, so $j^2+j+1 \leq N$. Let $\chi \in \Irr(G)$ have $\cl^*(\chi) = j$. 
Let $S := G_j$, and let $\tau:= \tau_{jN}$.  The restriction of $\tau$ from $G_{jN}$ to the central product $G_N\ast G_j$ gives a decomposition
$$\tau|_{G\ast S} = \sum_{\alpha\in \Irr(S)} D_\alpha\boxtimes \alpha,$$
where $D_\alpha$ is given by the well-known formula
\begin{equation}\label{gl-dual1}
  D_\al(g) = \frac{1}{|S|}\sum_{s \in S}\tau(g \otimes s)\bar\al(s).
\end{equation}  
%
By \cite[Theorem 1.1]{GLT1}, there exists a unique irreducible factor $D_\alpha^\circ$ which is of true level $j$, and all other factors have
strictly lower true level.  Moreover, for each $\chi$ of true level $j$, there exists a unique $\alpha\in \Irr(S)$ such that $\chi = D_\alpha^\circ$.
Consider any $g \in G \smallsetminus \ZB(G)$ of support $m$
and primary eigenvalue $\lambda$.

If $\varepsilon=+$ (resp. $\varepsilon = -$), we define $Q := q$ (resp. $Q:=q^2$).  We write $A:=\F_Q^N$, $B := \F_Q^j$, $V := A\otimes_{\F_Q} B$, which is endowed with the tensor product of the Hermitian forms on $A$ and $B$ if $\varepsilon = -$.
We will bound $\DC_\al(1)$ and $|\DC_\al(g)|$ using \eqref{gl-dual1}.
According to Definition 3.2 and Theorem 1.1 of \cite{GLT1}, embedding $G = G\times \{1\}\hookrightarrow G\ast S\hookrightarrow G_{jN}$, we can write 
\begin{equation}\label{gl-dual2}
  \tau|_G = \sum^M_{i=1}a_i\theta_i,~~D'_\al := D_\al-\DC_\al = \sum^{M'}_{i=1}b_i\theta_i,
\end{equation}  
where $\theta_i \in \Irr(G)$ are pairwise distinct, $a_i,b_i \in \Z_{\geq 0}$,
$M \geq M'$, $a_i \geq b_i$ if $i \leq M'$, $\cl^*(\theta_i) \leq j$ for all $i$. In fact, if 
$i \leq M'$, then $\cl^*(\theta_i) \leq j-1$, and so $\cl(\theta_i) = \cl^*(\theta_i) \leq j-1 < N/2$ as $j \leq N/2$, whence 
$\theta_i(1) \leq q^{N(j-1)}$ by \cite[Theorem 1.2]{GLT1}.  

Let $k(X) = |\Irr(X)|$ denote the class number of a finite group $X$. By \cite[Proposition 3.5]{FG}, $k(\GL_n(q)) \leq q^n$.  Note that 
$M'$ cannot exceed the total number of irreducible characters of true level $<j$. Hence, when $\e=+$ by \cite[Theorem 1.1]{GLT1} we have 
$$M' \leq \sum^{j-1}_{i=0}k(\GL_i(q)) \leq \sum^{j-1}_{i=0}q^i < q^j.$$
Also, $\sum^M_{i=1}a_i^2 = [\tau|_G,\tau|_G]_G =[\tau^2|_G,1_G]_G \leq 8q^{j^2}$ by \cite[Lemma 2.4]{GLT1} (since 
$\tau$ is the permutation character of $\GL_{jN}(q)$ on $\F_q^{Nj}$, and therefore $\tau^2\vert_G$ is the permutation character of $G$ acting on $\F_q^{Nj}\oplus \F_q^{Nj}$). It follows that
\begin{equation}\label{for-m10}
  (\sum^{M'}_{i=1}b_i)^2 \leq M'\sum^{M'}_{i=1}b_i^2 \leq q^j\sum^M_{i=1}a_i^2 \leq 8q^{j^2+j},
\end{equation}  
and so
\begin{equation}\label{gl-dual3}
  D'_\al(1) \leq \sum^{M'}_{i=1}b_iq^{N(j-1)} \leq \sqrt{8q^{j^2+j}}q^{N(j-1)} \leq q^{N(j-1)+(j^2+j+3)/2}.
\end{equation}
Assume now that $\e=-$. Then every irreducible constituent of $D'_\al$ has level $\leq j-2$ by \cite[Corollary 4.8]{GLT1} if $j \geq 2$, and  
$k(\GU_n(q)) \leq 8.26q^n$ by \cite[\S3.3]{FG}. So in this case we have
$$M' \leq \sum^{j-2}_{i=0}k(\GU_i(q)) \leq 8.26\sum^{j-2}_{i=0}q^i < 8.26q^{j-1}.$$
If $j=1$, then $D'_\al$ is zero, and so $M'=0$.
Also, $\sum^M_{i=1}a_i^2 = [\tau|_G,\tau|_G]_G = [\tau^2|_G,1_G] \leq 2q^{j^2}$ by \cite[Lemma 2.4]{GLT1} (since 
$\tau^2$ is the permutation character of $\GU_{jN}(q)$ on $\F_{q^2}^{Nj}$). It follows that
\begin{equation}\label{for-m11}
  (\sum^{M'}_{i=1}b_i)^2 \leq M'\sum^{M'}_{i=1}b_i^2 \leq 8.26q^{j-1}\sum^M_{i=1}a_i^2 \leq 16.52q^{j^2+j-1} \leq 8.26q^{j^2+j},
\end{equation}  
and so
$$D'_\al(1) \leq \sum^{M'}_{i=1}b_iq^{N(j-2)} \leq \sqrt{8.26q^{j^2+j}}q^{N(j-2)} < \sqrt{8q^{j^2+j}}q^{N(j-1)},$$
and thus \eqref{gl-dual3} holds in this case as well.

\smallskip
(a) We now complete the induction step for (i).
Recall that $g$ 
has support $m < N/2$, which implies that $g$ has a 
primary eigenvalue $\lambda$ with 
$\lambda^{q-\e}=1$, and the dimension $\delta_A(g)$ of the largest eigenspace of $g$ on $A$, which corresponds to
$\lambda$, satisfies 
$$\delta_A(g) = N-m > \frac N2.$$
Multiplying $g$ by the central element $\lambda^{-1}  \Id_A$, we may assume $\lambda=1$.
Consider any $s \in S = \GL^\e(B) \cong \GL^\e_j(q)$. By \cite[Lemma 8.2]{GLT1} we always have 
\begin{equation}\label{chi10a}
  d_V(g \otimes s):=\dim \mathrm{Ker}({g \otimes s}-\Id_V) \leq j(N-m).
\end{equation}  
More precisely, this bound is attained for $s_0=\Id_B $, whereas for all other elements $s\neq s_0$ in $S$ we have the stronger bound
\begin{equation}\label{chi10}
  d_V(g \otimes s) \leq (N-m)(j-2)+N = N^*+m,\mbox{ where }N^*:=(N-m)(j-1).
\end{equation}  
Note that $\tau(g \otimes s) = \e^{Nj}(\e q)^{d_V(g \otimes s)}$; in particular, 
$$\tau(g \otimes s_0) = \e^{Nj}(\e q)^{(N-m)j} = \e^{mj}q^{(N-m)j}.$$ 

By \eqref{gl-dual1} we can now write
\begin{equation}\label{chi11}
  \chi(g)=\DC_\al(g) = \frac{\al(1)}{|S|}\Bigl(  \e^{mj}q^{(N-m)j} + X \Bigr),
\end{equation}  
where 
$$X := \sum_{s_0 \neq s \in S}\frac{\tau(g \otimes s)\bar\al(s)}{\al(1)}-\frac{|S|}{\al(1)}D'_\al(g).$$
Recall that for $1 \leq i \leq M'$, $\theta_i$ has true level $0 \leq l \leq j-1$ and degree $\leq q^{Nl}$, so by the induction hypothesis we have 
$$|\theta_i(g)| \leq \frac{\theta_i(1)}{q^{ml}}\bigl(1+\frac{1}{q^{(N-\kappa m)/2}}\bigr).$$
Here, as $m < N/2$ and $\kappa \leq 2$, we have $N-\kappa m \geq N-2m \geq 1$, whence 
$$1+q^{-(N-\kappa m)/2} \leq 1+q^{-1/2} \leq 1+\frac1{\sqrt{2}} < q^{0.8}.$$ 
It follows that $|\theta_i(g)| \leq q^{(N-m)l+0.8} \leq q^{N^*+0.8}$. Together with \eqref{for-m10} and \eqref{for-m11}, this implies that
$$|D'_\al(g)| \leq \sqrt{8.26q^{j^2+j}}q^{N^*+0.8} < q^{N^*+(j^2+j)/2+2.4}.$$
Furthermore, 
$$|S| = |\GL^\e_j(q)| \leq (q+1)q^{j^2-1} < q^{j^2+0.6}.$$ 
Using \eqref{gl-dual3} and \eqref{chi10}, we have
$$|X| \leq q^{N^*+m+j^2+0.6}+q^{N^*+(3j^2+j)/2+3}.$$
Similarly, $d_V(\Id_A \otimes s) = Nd_V(s) \leq N(j-1)$ for $s \neq s_0$, and so
\begin{equation}\label{chi12}
  \chi(1) = \DC_\al(1)=\frac{\al(1)}{|S|}\Bigl(  q^{Nj} + Y \Bigr),
\end{equation}  
where 
$$|Y| := \Bigm|\sum_{s_0 \neq s \in S}\frac{\tau(I_A\otimes s)\bar\al(s)}{\al(1)}-\frac{|S|}{\al(1)}D'_\al(1)\Bigm| \leq q^{N(j-1)}\bigl( q^{j^2+0.6} +q^{(3j^2+j)/2+2.1}\bigr).$$
Now, \eqref{chi11} and \eqref{chi12} imply
$$\Bigm| \frac{(\e q)^{mj}\chi(g)}{\chi(1)}-1\Bigm| = \Bigm| \frac{q^{Nj}+X(\e q)^{mj}}{q^{Nj}+Y}-1\Bigm| = \Bigm|\frac{X(\e q)^{mj}-Y}{q^{Nj}+Y} \Bigm|.$$
Setting $R:= q^{m+j^2+0.6}+q^{(3j^2+j)/2+3.1}$, 
we have 
$$|X(\e q)^{mj}+Y| \leq R\bigl(q^{Nj-N+m}+q^{Nj-N-1}\bigr),~~|q^{Nj}+Y| \geq q^{Nj-N}(q^N-R)$$
Hence to prove \eqref{bound1}, it suffices to prove
$$R\bigl((q^m+q^{-1})q^{(N-\kappa m)/2}+1\bigr) < q^{N}.$$
Note that $q^m+q^{-1} \leq q^m(1+1/4) < q^{m+0.4}$, and $1+q^{-(m+0.9)} \leq 1+2^{-1.9} < q^{0.4}$. Hence it suffices to show that 
$$Rq^{m+(N-\kappa m)/2+0.8} < q^N.$$
The condition on $N$ implies that 
$$N-1 > m+(N-\kappa m)/2+0.8+ \max \bigl( m+j^2+0.6,(3j^2+j)/2+3.1\bigr),$$
and so we are done.

\smallskip
(b) We now complete the induction step for (ii); in particular $j \geq 2$.
By \eqref{gl-dual1} we have 
$$\chi(g)=\DC_\al(g)=  \frac{\al(1)Z}{|S|},$$
where 
$$Z := \sum_{s \in S}\frac{\tau(g \otimes s)\bar\al(s)}{\al(1)}-\frac{|S|}{\al(1)}D'_\al(g).$$
By \eqref{chi10a}, the first sum has absolute value at most $|S|q^{(N-m)j}$.
For $1 \leq i \leq M'$, $\theta_i$ has true level $0 \leq l \leq j-1$ and degree $\leq q^{Nl}$. If $l \geq 1$, then the induction hypothesis implies that
$$|\theta_i(g)| \leq q^{2-m}\theta_i(1) \leq q^{2-m+N(j-1)}.$$
Since $j \geq 2$, the same bound also holds for $l=0$.
Together with \eqref{for-m10} and \eqref{for-m11}, this implies that
$$|D'_\al(g)| \leq \sqrt{8.26q^{j^2+j}}q^{N(j-1)+2-m} < q^{N(j-1)-m+(j^2+j)/2+3.6}.$$
It follows that
$$|Z| \leq q^{(N-m)j+j^2+0.6}+q^{N(j-1)-m+(3j^2+j)/2+4.2}.$$
Since \eqref{chi12} also holds in this case, to prove \eqref{bound2}, it suffices to prove
$$q^m|Z| +|Y| < q^{Nj}.$$
The assumption on $m$ implies that $Nj > N_1+m+2$ where $N_1:= (N-m)j+j^2+0.6$, whence
$q^{Nj}-q^{N_1+m} > (1-q^{-2})q^{Nj} \geq 0.75q^{Nj}$. Next, the assumption on $N$ implies that
$Nj \geq N_2+0.8$ for $N_2:= N(j-1)+(3j^2+j)/2+4.2$, whence
$0.75q^{Nj} \geq 0.75q^{0.8}q^{N_2} > 1.3q^{N_2}$. Also, 
$N_2 > N_3+2$ for $N_3:= N(j-1)+(3j^2+j)/2+2.1$ and 
$N_2 > N_4+5.6$ for $N_4:=N(j-1)+j^2+0.6$, whence 
$0.3q^{N_2} > q^{N_3}+q^{N_4}$. Thus
$$q^{Nj} > (q^{N_1+m}+q^{N_2})-(q^{N_3}+q^{N_4}) \geq q^m|Z|+|Y|,$$
and so we are done again.
\end{proof}

We will also need the following quantitative statement, which is well known in the $\GL$-case:

\begin{lem}\label{cent}
Let $q$ be any prime power, $N \in \Z_{\geq 6}$, and $g$ any element in $G=\GU_N(q)$. If $g$ has support $1$, then
$$|\CB_G(g)| \leq q^{N^2-2N+3.2}.$$
If $g$ has support $\geq 2$, then
$$|\CB_G(g)| \leq q^{N^2-3.4N+5}.$$
\end{lem}

\begin{proof}
We will use the estimates 
\begin{equation}\label{ord20}
  q^{N^2} < |\GU_N(q)| < q^{N^2+0.6}
\end{equation}  
for any $N \geq 1$.
Let $m:=\supp(g)$, and consider the case $m \geq 4$. By \cite[Proposition 4.2(a)]{LT3}, $\CB_{\GL_N(\bar{\F}_q)}(g)$ has dimension $\leq N(N-m)$. 
Now we can follow the proof of \cite[Proposition 6.1]{LT3} in passing to the finite case and increase the upper bound
by a factor of $q^{0.6N}$ to account for the upper bound in \eqref{ord20}. This shows
$$|\CB_G(g)| \leq q^{N(N-m+0.6)} \leq q^{N^2-3.4N}.$$ 
In the remaining cases $1 \leq m \leq 3$, let $s$ denote the semisimple part of $g$. Suppose first that $s \in \ZB(G)$. Then we may assume 
that $g$ is unipotent, and use \cite[Theorem 7.1]{LiSe} and \eqref{ord20} to bound $|\CB_{G}(g)|$. Denoting the Jordan block of size $i$ and eigenvalue $1$
by $J_i$, we have the following possibilities for $g$. If $g=J_1^{N-5}\oplus J_2^1 \oplus J_3^1$, then
$$|\CB_G(g)| \leq q^{N^2-6N+14+1.8} \leq q^{N^2-3.4N+5}$$
when $N \geq 5$ (and the term $1.8$ accounts for three $\GU$-factors in $\CB_G(g)$). 

Similarly, if $g=J_1^{N-4}\oplus J_4^1$, then $|\CB_G(g)| \leq q^{N^2-6N+13.2}$.

If $g=J_1^{N-3}\oplus J_3^1$, then $|\CB_G(g)| \leq q^{N^2-4N+7.2}$.

If $g=J_1^{N-6}\oplus J_2^3$, then $|\CB_G(g)| \leq q^{N^2-6N+19.2} < q^{N^2-3.4N+5}$ since $N \geq 6$.

If $g=J_1^{N-4}\oplus J_2^2$, then $|\CB_G(g)| \leq q^{N^2-4N+9.2} \leq q^{N^2-3.4N+5}$ since $N \geq 7$.

If $g=J_1^{N-2}\oplus J_2^1$, then $m=1$, and $|\CB_G(g)| \leq q^{N^2-2N+3.2}$.

Now we may assume $s \notin \ZB(G)$. Then $1 \leq \supp(s) \leq m \leq 3$. If $\supp(s)=3$, then as $N \geq 7$ we have 
$$|\CB_G(g)| \leq |\CB_G(s)| \leq |\GU_3(q) \times \GU_{N-3}(q)| \leq q^{N^2-6N+19.2} < q^{N^2-3.4N+5}.$$
Similarly, if $\supp(s)=2$, then 
$$|\CB_G(g)| \leq |\CB_G(s)| \leq |\GU_2(q) \times \GU_{N-2}(q)| \leq q^{N^2-4N+9.2} \leq q^{N^2-3.4N+5}.$$
In the remaining cases, $\supp(s)=1$ and $\CB_G(s) = \GU_1(q) \times \GU_{N-1}(q)$. If $m=1$, then 
$|\CB_G(g)|=|\CB_G(s)| \leq q^{N^2-2N+3.2}$. If $2 \leq m \leq 3$, then the unipotent part of the $\GU_{N_1}(q)$-component of $g$ has
support $1$ or $2$, respectively, and 
$|\CB_G(g)| \leq q^{N^2-4N+7.8}$, respectively $|\CB_G(g)| \leq q^{N^2-6N+15.8}$. 
\end{proof}

\section{Thompson's conjecture for $\mathrm{PSU}_{2p}(q)$}
Let $p \geq 3$ be any prime, and let $q\le 7$ be a prime power (which may be coprime to $p$).
In this section we show that Thompson's conjecture
holds for the simple group $\mathrm{PSU}_{2p}(q)$ as long as $p$ is sufficiently large. 
Together with the main result of \cite{EG}, this implies Thompson's conjecture for all sufficiently large unitary simple groups 
in the case that the rank is twice a prime less $1$.

We will work inside the general unitary group $G = \GU_{2p}(q)$ and consider any element $x \in G$ which has some  fixed
$\xi \in \F_{q^p} \smallsetminus \F_q$ as an eigenvalue. Note that $\xi$ has degree $p$ over $\F_q$, and the orbit of $\xi$ under the map
$\lambda \mapsto \lambda^{-q}$ has length $2p$ and contains $\xi^{-1}$. It follows that $x$ is regular semisimple and real and 
belongs to $\SU_{2p}(q)$. 

\begin{thm}\label{su2p}
There is an explicit absolute constant $A > 0$ such that the following statement holds when the prime $p$ is at least $A$. If $C$ denotes the 
conjugacy class of the image of the element $x$ described above in the simple group $S = \mathrm{PSU}_{2p}(q)$, then
$S=C^2$. In fact, $(x^{\SU_{2p}(q)})^2$ contains all unipotent elements and all elements of support $\ge 2$ in $\SU_{2p}(q)$.
\end{thm}

The rest of the section is devoted to the proof of Theorem \ref{su2p}.
The chosen element $x$ has 
$$T:=\CB_G(x) \cong C_{q^{2p}-1}.$$ 
In particular, $T \cap \SU_{2p}(q)$ has index $q+1$ in $T$, and hence the conjugacy class $x^G$ is a single $\SU_{2p}(q)$-conjugacy class. As $x$ is real,
it suffices to show that, any non-central element $g \in \SU_{2p}(q)$ 
is a product of two $G$-conjugates of $x$, except possibly in the case $g=zh$ where $1 \neq z \in \ZB(\SU_{2p}(q))$ and $h$ is a transvection.
By the extension \cite[Lemma 5.1]{GT1} of Gow's lemma \cite{Gow}, the statement holds in the case
$g$ is semisimple. So we may assume that $g$ is not semisimple, and furthermore  
$g$ is a transvection if $\supp(g)=1$.

As $x$ is real, we can rewrite \eqref{Frobenius} as follows:
\begin{equation}\label{sum10}
  \sum_{\chi \in \Irr(G)}\frac{|\chi(x)|^2\bar\chi(g)}{\chi(1)} \neq 0.
\end{equation}  
The summation in \eqref{sum10} includes $q+1$ linear characters which all take value $1$ at any element in 
$\SU_{2p}(q)$. Hence it suffices to show
\begin{equation}\label{sum11}
\Bigm|\sum\nolimits^*\frac{|\chi(x)|^2\chi(g)}{\chi(1)}\Bigm| < q+1,
\end{equation} 
where $\sum^*$ indicates the sum over $\chi\in\Irr(G)$ such that $\chi(1)>1$, $\chi(x)\neq 0$, and $\chi(g)\not\in [0,\infty)$.
We denote by $\Sigma^*$ the sum in \eqref{sum11}.

We will identify the dual group $G^*$ with $G$. Note that any non-central semisimple element $s \in G$ has the property that $\CB_G(s)$ is a Levi subgroup
$L$ of $G$. Fix an embedding of $\bar{\F}_q^\times$ into $\C^\times$. Then one can identify $\ZB(\CB_G(s))$ with 
$$\mathrm{Hom}(\CB_G(s)/[\CB_G(s),\CB_G(s)],\C^\times)$$ 
as in \cite[(1.16)]{FS}, and the linear character of $L=\CB_G(s)$ corresponding to
$s$ will be denoted by $\hat{s}$. Now, any irreducible character $\chi$ in the Lusztig series labeled by $s$ is 
\begin{equation}\label{sum12}
  \chi = \pm R^G_L(\hat{s}\psi),
\end{equation} 
where $\psi$ is a unipotent character of $L$; see \cite[p. 116]{FS}. Let $\St_G$ and $\St_L$ denote the Steinberg character of $G$, respectively of
$L$. By \cite[Corollary 10.2.10(ii)]{DM},
\begin{equation}\label{sum13}
  \St_G\cdot \chi = \pm \Ind^G_L(\hat{s}\psi\cdot\St_L).
\end{equation}  

Assume now that $\chi(x) \neq 0$. As $x$ is regular semisimple, 
$$\St_G(x) = \pm 1.$$
Hence \eqref{sum13} implies that, up to conjugacy, $x \in L=\CB_G(s)$. It follows that $s \in \CB_G(x)=T$; in fact,
$s \in T \smallsetminus \ZB(G)$. Now using the primality of $p$, we see that 
the characters $\chi \in \Irr(G)$ that are non-vanishing at $x$ are divided into the following types:

\begin{enumerate}[\rm(I)]
\item Unipotent characters, multiplied possibly by a linear character of $G$.
\item Characters $\chi$ in \eqref{sum12} with $s$ belonging to the unique subgroup $C_{q^2-1}$ of $T$, but not in $\ZB(G) \cong C_{q+1}$
(which is possible only if $q > 2$). For any such $s$, $\CB_G(s) \cong \GL_p(q^2)$, a rational Levi subgroup of $G$.
\item Characters $\chi$ in \eqref{sum12} with $s$ belonging to the unique subgroup $C_{q^p+1}$ of $T$, but not in $C_{q^2-1}$.
For any such $s$, $\CB_G(s) \cong \GU_2(q^p)$.
\item Characters $\chi$ in \eqref{sum12} with $s$ not belonging to the union of the unique subgroups $C_{q^p+1}$ and $C_{q^2-1}$ of $T$.
Any such $s$ is regular semisimple, with $\CB_G(s) =T$.
\end{enumerate} 

Let $m$ denote the support of $g$, which is also $\supp(zg)$ for any $z \in \ZB(G)$. We will use the character bound \cite[Theorem 5.5]{LT3}
\begin{equation}\label{lt-bound}
  \frac{|\chi(g)|}{\chi(1)} \leq \chi(1)^{-\sigma m/N}
\end{equation}
for an explicit absolute constant $\sigma > 0$, and 
$$N:=2p.$$

\begin{lem}\label{type1}
There is an absolute constant $N_1>0$ such that when $N \geq N_1$, the total contribution of characters of type {\rm (I)} to $\Sigma^*$ has absolute value 
$< (q+1)/2$. 
\end{lem}

\begin{proof}
It suffices to prove that the contribution of unipotent characters $\chi$ of type {\rm (I)} to $\Sigma^*$ has absolute value 
$< 1/2$.  Since $x$ is regular semisimple with centralizer $T \cong C_{q^N-1}$, we can have $\chi(x) \neq 0$ only when the unipotent character
$\chi$ is labeled by one of the $N$ hook partitions $\al=(N-j,1^j)$, in which case 
$$|\chi(x)| = 1,$$
cf. the proof of \cite[Proposition 3.1.5]{LST}.
  
First we consider $\chi=\chi_j=\chi^{(N-j,1^j)}$ with $j \geq \sqrt{N/13}$. Choosing $N_1$ sufficiently large, when $N \geq N_1$ we have $\cl^*(\chi) =j$ by \cite[Theorem 3.9]{GLT1},
and so, by \cite[Theorem 1.2]{GLT1}, $\chi(1) \geq q^{c_1N^{3/2}}$ for some absolute constant $c_1 > 0$. So, by \eqref{lt-bound}, the total contribution of 
these unipotent characters to $|\Sigma^*|$ is at most
$$Nq^{-c_1\sigma\sqrt{N}} < 1/12$$
when $N_1$ is large enough.

Now we look at the $\chi=\chi_j$ with $1 \leq j \leq \sqrt{N/13}$. Suppose that 
$$m=\supp(g) \geq \sqrt{N/13}.$$ 
Since $\chi(1) \geq q^{N/2}$, by \eqref{lt-bound}, the total contribution of these unipotent characters to $|\Sigma^*|$ is at most
$$\sqrt{N/13}q^{-\sigma\sqrt{N/52}} < 1/3$$
when $N_1$ is large enough, yielding the claim in this case.

In the remaining case, we have $m \leq \sqrt{N/13}$ and $j \leq \sqrt{N/13}$. By Theorem \ref{mainB},
$$\frac{\chi_j(g)}{\chi_j(1)}=\e_j+(-q)^{-mj}\gamma_\lambda,$$
where $|\e_j| < q^{-N/2}$ and $\gamma_\lambda = \chi(\lambda I_A)/\chi(1)$ for the primary eigenvalue $\lambda$ of $g$. If $m \geq 2$, then
$$\Bigm| \sum_{j \geq 1}(-q)^{-mj}\gamma_\lambda\Bigm|$$
is at most $\sum_{j \geq 1}4^{-j} = 1/3$. If $m =1$, then $g$ is a transvection by our convention, so $\lambda=1$, $\gamma_\lambda=\gamma_1=1$,
and 
$$\Bigm| \sum_{j \geq 1}(-q)^{-mj}\gamma_\lambda\Bigm| \leq 1/(q+1) \leq 1/3.$$ 
Choosing $N_1$ large enough, we therefore have
$$\Bigm| \sum_{1 \leq j \leq \sqrt{N/13}}(-q)^{-mj}\gamma_\lambda\Bigm| \leq \frac{1}{3}+\frac{1}{24}$$
and 
$$\Bigm| \sum_{1 \leq j \leq \sqrt{N/13}}\eps_j\Bigm| \leq 2^{-N/2}\sqrt{N/13} \leq \frac{1}{24}.$$
It follows that
$$\Bigm| \sum_{1 \leq j \leq \sqrt{N/13}}\frac{|\chi_j(x)|^2\chi_j(g)}{\chi_j(1)}\Bigm| \leq \frac{1}{3}+\frac{1}{12},$$
and so we are done in this case as well.
\end{proof}

\begin{lem}\label{type2}
There is an absolute constant $N_2>0$ such that when $N \geq N_2$, the total contribution of characters of type {\rm (II)} to $\Sigma^*$ has absolute value 
$< 1/4$. 
\end{lem}

\begin{proof}
Each of these characters can be written in the form \eqref{sum12}, where $L=\GL_p(q^2)$ is unique up to conjugacy in $G$. 
For the element $x$, we can use the bound
\begin{equation}\label{for-x}
  |\chi(x)| \leq N
\end{equation} 
for all irreducible characters of $G$, which follows from \cite[Corollary 7.6]{LTT}. There are fewer than $q^2 \leq 49$ choices for the element $s$ (up to conjugacy). Furthermore, by
\eqref{sum13}, the condition $\chi(x) \neq 0$ implies that (after conjugating $x$ suitably) $x \in L$ and $\psi(x) \neq 0$. Viewed as an element in
$L$, $x$ is again regular semisimple with centralizer $C_{q^{2p}-1}$. Hence the unipotent character $\psi$ is labeled by one of the $N/2$ hook partitions of
$p=N/2$. Thus there are fewer than $25N$ characters of type (II) with $\chi(x) \neq 0$. Each of them has degree
$$\chi(1) \geq [G:L]_{q'} = (q+1)(q^3+1) \ldots (q^{N-1}+1) > q^{N^2/4}.$$
Hence, by \eqref{lt-bound}, the total contribution of 
these characters to $|\Sigma^*|$ is at most
$$25N^3q^{-\sigma N/4} < 1/4$$
when $N \geq N_2$ and $N_2$ is large enough.
\end{proof}

To complete the proof of \eqref{sum11} and of Theorem \ref{su2p}, it remains to prove:

\begin{prop}\label{type34}
There is an absolute constant $N_3>0$ such that when $N \geq N_3$, the total contribution of characters of types {\rm (III)} and {\rm (IV)} to $\Sigma^*$ has absolute value $< 1/4$. 
\end{prop}

\begin{proof}
(a) First we consider the case $\supp(g) \geq 2$. By Lemma \ref{cent} we have
$$|\chi(g)| \leq \sqrt{|\CB_G(g)|} \leq q^{N^2/2-1.7N +2.5}$$
for any $\chi$; and we still have the bound \eqref{for-x} for $|\chi(x)|$. 

Each $\chi$ of type (III) can be written in the form \eqref{sum12}, 
where $L=\GU_2(q^p)$ is unique up to conjugacy in $G$, and  
$$\chi(1) \geq [G:L]_{q'} = (q+1)(q^2-1) \ldots (q^{p-1}-1) \cdot (q^{p+1}-1)(q^{p+2}+1) \ldots (q^{2p-1}+1) > q^{N^2/2-N-1}$$
by \cite[Lemma 4.1(iii)]{LMT}. 
There are fewer than $q^{N/2}$ choices for $s$, and two choices for the 
unipotent character $\psi$. Thus there are fewer than $2q^{N/2}$ choices for $\chi$ of type (III). Hence the total contribution of 
these characters to $|\Sigma^*|$ is at most
$$2N^2q^{N/2}q^{-0.7N+3.5} \leq N^22^{3.5-0.2N} < \frac18$$
when $N \geq N_3$ and $N_3$ is large enough.

Each $\chi$ of type (IV) can be written in the form \eqref{sum12}, 
where $L=\GL_1(q^{2p})$ is unique up to conjugacy in $G$, and  
$$\chi(1) \geq [G:L]_{q'} = (q+1)(q^2-1) \ldots (q^{2p-1}+1) > q^{(N^2-N)/2}.$$
There are fewer than $q^N$ choices for $s$, and one choice for the 
unipotent character $\psi$. Hence the total contribution of these characters to $|\Sigma^*|$ is at most
$$N^2q^Nq^{-1.2N+2.5}  \leq N^22^{2.5-0.2N}< \frac18$$
when $N \geq N_3$ and $N_3$ is large enough.

\smallskip
(b) Next assume that $\supp(g)=1$. Even though \eqref{sum10} is already proved in the case $g$ is semisimple, for further use we will 
consider this case here. So assume that $g$ is semisimple with $\supp(g)=1$. Then
\begin{equation}\label{for-g1}
  \CB_G(g) = M \cong \GU_1(q) \times \GU_{N-1}(q).
\end{equation}   
Consider a Levi subgroup $L = \GU_2(q^p)$ and assume that $M \cap L$ contains a maximal torus $S$ of $G$. Then $S$ contains the 
central torus $\ZB(L) \cong C_{q^p+1}$ of $L$. In particular, $S$ contains a generator $t$ of $\ZB(L)$, all of whose eigenvalues on 
$\F_{q^2}^N$ are of order $q^p+1$. But then $t$ cannot belong to $M$, since any element in $M$ (namely its 
$\GU_1(q)$-component) has an eigenvalue of order dividing $q+1$, a contradiction. 

The same argument shows that $M \cap L$ cannot contain any maximal torus of $L = \GL_1(q^{2p})$. By the Mackey formula, see 
\cite[Theorem 5.2.1 and p. 140]{DM}, we now have
$$\tw*R^G_M(\chi) = 0$$ 
for any character $\chi$ of type (III) or (IV). Together with \eqref{for-g1}, this implies by \cite[Lemma 13.3]{TT} that $\chi(g)=0$ for all these characters.

\smallskip
(c) In the rest of the proof, we will assume that $g$ is a transvection, and aim to show that 
\begin{equation}\label{for-g2}
  \chi(g) = \frac{\chi(1)}{q^{N-1}+1}
\end{equation} 
for any character $\chi$ of type (III) or (IV). In particular, $\chi(g) > 0$, so it does not contribute to $\Sigma^*$ in \eqref{sum11}.

To do this, we consider $G^+:=\GL_{2p}(q)$ (and identify its dual with $G^+$), and its Levi subgroups 
$$L_1^+:=L^+ \cong \GL_2(q^{p}),~L_2^+:=T^+ \cong \GL_1(q^{2p}),~M^+ \cong \GL_1(q) \times \GL_{N-1}(q).$$ 
Furthermore, we consider any character $\chi^+ \in \Irr(G^+)$ that lies in the Lusztig series labeled
by a semisimple element $s^+$ with $\CB_{G^+}(s^+)=L_i^+$, $i=1,2$; in particular, $\chi^+$ is Lusztig induced from $L_2^+$ by \eqref{sum12}. The 
same argument as in (b) shows that $L_i^+ \cap M^+$ contains no maximal torus of $G^+$, and hence 
\begin{equation}\label{srgl}
  \tw*R^{G^+}_{M^+}(\chi^+) = 0
\end{equation}
by the Mackey formula.

The Levi subgroup $M^+$ is a Levi subgroup of a parabolic subgroup $P=U \rtimes M^+$ of $G^+$ (the stabilizer of a line in $\F_q^N$), where
the unipotent radical $U$ is an elementary abelian subgroup of order $q^{N-1}$. Furthermore, $M^+$ acts transitively on the set of 
nontrivial elements of $U$ and on the set of nontrivial irreducible characters of $U$. Fix a nontrivial element $g^+$ of $U$, which is then a transvection.
Since $M^+$ is rational, $\tw*R^{G^+}_{M^+}(\chi^+)$ is just the Harish-Chandra restriction of $\chi^+$ to $P$. The formula 
\eqref{srgl} now means that the fixed point subspace of $U$ in any module $V$ affording $\chi^+$ is zero. Hence the restriction of $V$ to
$U$ affords all, and only, nontrivial irreducible characters of $U$, each with the same multiplicity 
$$e:=\chi^+(1)/(q^{N-1}-1).$$ 
The sum of these characters is the regular character of $U$ minus $1_U$, hence takes value $-1$ at $g^+$. It follows that 
\begin{equation}\label{for-g3}
  \chi^+(g^+) = -e=-\frac{\chi^+(1)}{q^{N-1}-1}.
\end{equation} 
We will interpret this identity in terms of Green functions of $L^+=\GL_2(q^p)$. According to \cite[p. 191]{DM},
$$R^{L^+}_{T^+_{1,1}}(1_{T^+_{1,1}}) = 1_{L^+}+\St_{L^+},~R^{L^+}_{T^+_{2}}(1_{T^+_{2}}) = 1_{L^+}-\St_{L^+}.$$
Here, $T^+_{1,1}$ is a maximal torus in $L^+$ of order $(q^p-1)^2$ and $T^+_2$ is a maximal torus in $L^+$ of order $q^{2p}-1$. Viewed in $G^+$, they become
maximal tori $T^+_{p,p}$ and $T^+_{2p}$, corresponding to permutations of type $(p,p)$ and $(2p)$ in the Weyl group $\SSS_{2p}$. Applying 
Lusztig induction to $G^+$, we get
\begin{equation}\label{for-g4}
  R^{G^+}_{T^+_{p,p}}(1_{T^+_{p,p}}) = R^{G^+}_{L^+}(1_{L^+})+R^{G^+}_{L^+}(\St_{L^+}),~
  R^{G^+}_{T^+_{2p}}(1_{T^+_{2p}}) = R^{G^+}_{L^+}(1_{L^+})-R^{G^+}_{L^+}(\St_{L^+}).
\end{equation}   
We now evaluate these formulae at the unipotent element $g^+ \in G^+$,
and use the fact that $R^{G^+}_{L'}(\al)$ and 
$R^{G^+}_{L'}(\theta\al)$ agree at the unipotent element $g^+$ for any linear character $\theta$ of a Levi subgroup $L'$ of $G^+$, cf. 
\cite[Proposition 10.1.2]{DM}. 
Also, each of $R^{G^+}_{L^+}(\theta\cdot 1_{L^+})$ and $R^{G^+}_{L^+}(\theta\cdot\St_{L^+})$ 
is, up to sign, an irreducible character $\chi^+$ corresponding to 
$L^+$, for a suitable choice of the linear character $\theta$. Hence \eqref{for-g3} implies the following equalities for Green functions
$$Q^{G^+}_{T^+_{p,p}}(g^+)= -\frac{Q^{G^+}_{T^+_{p,p}}(1)}{q^{N-1}-1},~Q^{G^+}_{T^+_{2p}}(g^+)= -\frac{Q^{G^+}_{T^+_{2p}}(1)}{q^{N-1}-1}.$$
Since we have proved these equalities for all prime powers $q$, they also hold as identities for Green functions considered 
as polynomials in $q$. Applying the Ennola duality $q \mapsto -q$, cf. \cite[\S15.9]{TT}, we now obtain
\begin{equation}\label{for-g5}
  Q^{G}_{T_{p,p}}(g)= \frac{Q^{G}_{T_{p,p}}(1)}{q^{N-1}+1},~Q^{G}_{T_{2p}}(g)= \frac{Q^{G}_{T_{2p}}(1)}{q^{N-1}+1}
\end{equation}  
for Green functions of $G=\GU_{2p}(q)$,
where $T_{p,p}$ is a maximal torus of order $(q^p+1)^2$ and $T_{2p}$ is a maximal torus of order $q^{2p}-1$ of $G$. 

We can take $T_{p,p}$ to be a maximal torus $T_{1,1}$ of $L=\GU_2(q^p)$ and $T_{2p}=T$ to be a maximal torus $T_2$ of $L$. Inside 
$L$ we now have
$$R^{L}_{T_{1,1}}(1_{T_{1,1}}) = 1_{L}-\St_{L},~R^{L}_{T_{2}}(1_{T_{2}}) = 1_{L}+\St_{L}.$$
After Lusztig inducing to $G$, we obtain
$$R^{G}_{T_{p,p}}(1_{T_{p,p}}) = R^{G}_{L}(1_{L})-R^{G}_{L}(\St_{L}),~
   R^{G}_{T_{2p}}(1_{T_{2p}}) = R^{G}_{L}(1_{L})+R^{G}_{L}(\St_{L}),$$
and so
\begin{equation}\label{rgl}  
  R^{G}_{L}(1_{L})= \bigl(R^{G}_{T_{p,p}}(1_{T_{p,p}})+R^{G}_{T_{2p}}(1_{T_{2p}})\bigr)/2,~
  R^{G}_{L}(\St_{L})= \bigl(R^{G}_{T_{p,p}}(1_{T_{p,p}})-R^{G}_{T_{2p}}(1_{T_{2p}})\bigr)/2.
\end{equation}  
Note that each irreducible character $\chi$ of type (III) is, up to sign, $R^{G}_{L}(\theta \cdot 1_{L})$ or $R^{G}_{L}(\theta\cdot\St_{L})$ for 
a suitable linear character $\theta$ of $L$, and the values at $g$ of the latter characters are equal to the values at $g$ of 
$R^{G}_{L}(1_{L})$, respectively $R^{G}_{L}(\St_{L})$.
Hence \eqref{for-g5} and \eqref{rgl} show that \eqref{for-g2} holds for every character of type (III).
Finally, each irreducible character $\chi$ of of type (IV) is, up to sign, $R^{G}_{T_{2p}}(\theta)$ for a suitable linear character $\theta$ of $T_{2p}$. Hence, \eqref{for-g5} shows that \eqref{for-g2} holds for every character of type (IV), and the proof is complete.
\end{proof}

\section{Thompson's conjecture for $\mathrm{PSU}_{2p+1}(q)$}
Let $p \geq 3$ be any prime, and let $q$ be any prime power (which may be coprime to $p$). In this section we show that Thompson's conjecture
holds for the simple group $\mathrm{PSU}_{2p+1}(q)$ as long as $p$ is sufficiently large. 

We will work inside $G := \GU_{2p+1}(q)$ and consider any semisimple element $x \in G$ which has both some  fixed
$\xi \in \F_{q^p} \smallsetminus \F_q$ and $1$ as eigenvalues. As in \S4,
this choice of $\xi$ implies that $x$ is regular semisimple, real, and 
of determinant $1$.

\begin{thm}\label{su2p1}
There is an explicit absolute constant $A > 0$ such that the following statement holds when the prime $p$ is at least $A$. If $C$ denotes the 
conjugacy class of the image of the element $x$ described above in the simple group $S = \mathrm{PSU}_{2p+1}(q)$, then
$S=C^2$. In fact, $(x^{\SU_{2p}(q)})^2$ contains all unipotent elements and all elements of support $\ge 2$.
\end{thm}

The rest of the section is devoted to the proof of Theorem \ref{su2p1}.
By the main result of \cite{EG}, it suffices to prove the result for 
$q \leq 7$.
The chosen element $x$ has 
$$T:=\CB_G(x) = T_1 \times T_2, \mbox{ where }T_1=\GL_1(q^{2p}) \cong C_{q^{2p}-1} \mbox{ and }T_2 = \GU_1(q) \cong C_{q+1}.$$ 
In particular, $T \cap \SU_{2p+1}(q)$ has index $q+1$ in $T$, and hence the conjugacy class $x^G$ is a single $\SU_{2p+1}(q)$-conjugacy class. As $x$ is real,
it suffices to show that, any non-central element $g \in \SU_{2p+1}(q)$ 
is a product of two $G$-conjugates of $x$, except possibly in the case $g=zh$ where $1 \neq z \in \ZB(\SU_{2p+1}(q))$ and $h$ is a transvection.
By \cite[Lemma 5.1]{GT1}, the statement holds in the case
$g$ is semisimple. So we may assume that $g$ is not semisimple, and furthermore  
$g$ is a transvection if $\supp(g)=1$.

Note that the statement also follows in the case
\begin{equation}\label{red10}
  \begin{aligned}g = \diag(g',1) \in \SU_{2p}(q) \times \SU_1(q) \mbox{ with }g' \in \SU_{2p}(q) \mbox{ but}\\
  g' \neq zh, \mbox{ with }1 \neq z \in \ZB(\SU_{2p}(q)) \mbox{ and }h \mbox{ is a 
  transvection}.\end{aligned}
\end{equation}    
Indeed, as $g$ is not semisimple, $g'$ is not central in $\SU_{2p}(q)$. We may put $x$ in $\SU_{2p}(q)$, and then apply Theorem \ref{su2p} to the element $g'$.
Therefore, in what follows we may assume that  
$$m=\supp(g) \geq 2.$$

As in \S4, we use the Frobenius formula in the form \eqref{sum10} and identify the dual group $G^*$ with $G$. Again, any non-central semisimple element $s \in G$ has the property that $\CB_G(s)$ is a Levi subgroup
$L$ of $G$. Fixing an embedding of $\overline{\F_q}^\times$ into $\C^\times$, we can identify $\ZB(\CB_G(s))$ with 
$$\mathrm{Hom}(\CB_G(s)/[\CB_G(s),\CB_G(s)],\C^\times),$$ 
and denote the linear character of $L=\CB_G(s)$ corresponding to
$s$ by $\hat{s}$. Then, any irreducible character $\chi$ in the Lusztig series labeled by $s$ can be written in the form \eqref{sum12},
where $G$ and $L$ are understood to be defined as in this section.
Assume now that $\chi(x) \neq 0$. As $x$ is regular semisimple, 
\eqref{sum13} implies that, up to conjugacy, $x \in L=\CB_G(s)$. It follows that $s \in \CB_G(x)=T$; in fact,
$s \in T \smallsetminus \ZB(G)$. Now using the primality of $p$, we see that 
the characters $\chi \in \Irr(G)$ that are non-vanishing at $x$ are divided into the following types:

\begin{enumerate}[\rm(I)]
\item Characters $\chi$ in \eqref{sum12} with $s \in T_{11} \times T_2 < T$, where $T_{11}$ is the unique subgroup $C_{q+1}$ of $T_1$.
\item Characters $\chi$ in \eqref{sum12} with $s \in T_{12} \times T_2 \setminus (T_{11} \times T_2)$, where $T_{12}$ is the unique subgroup $C_{q^2-1}$ of $T_1$ (in particular, $q > 2$). For any such $s$, $\CB_G(s) \cong \GL_p(q^2) \times \GU_1(q)$.
\item Characters $\chi$ in \eqref{sum12} with $s \in T_{13} \times T_2 \setminus (T_{11} \times T_2)$, where $T_{13}$ is the unique subgroup $C_{q^p+1}$ of $T_1$. For any such $s$, $\CB_G(s) \cong \GU_2(q^p) \times \GU_1(q)$.
\item Characters $\chi$ in \eqref{sum12} with $s \in T$ not belonging to the union of $T_{12} \times T_2$ and $T_{13} \times T_2$. Any such $s$ is regular semisimple, with $\CB_G(s) =T$.
\end{enumerate} 

The Frobenius formula includes $q+1$ linear characters which all take value $1$ at any element in 
$\SU_{2p+1}(q)$. Hence it suffices to show
\begin{equation}\label{sum21}
  \Bigm|\sum\nolimits^*\frac{|\chi(x)|^2\chi(g)}{\chi(1)}\Bigm| < q+1,
\end{equation} 
where, again, we sum only over $\chi$ with $\chi(1)>1$, $\chi(x)\neq 0$, and $\chi(g)\not\in [0,\infty)$.
We again refer to the sum \eqref{sum21} as $\Sigma^*$. We will frequently use the character bound \eqref{lt-bound}, with
$$N:=2p+1.$$

\begin{prop}\label{type1b}
There is an absolute constant $N_1>0$ such that when $N \geq N_1$, the total contribution of characters of type {\rm (I)} to $\Sigma^*$ has absolute value 
$< 2(q+1)/3$. 
\end{prop}

\begin{proof}
(a) The elements $s$ that label the characters $\chi$ of type (I) are 
$$s=\diag(aI_{N-1},b)$$
where $a,b \in \mu_{q+1} \subseteq \F_{q^2}^\times$.

If $a=b$, then $\chi$ is a unipotent character $\chi^\al$ times a linear character. 
Since $x$ is regular semisimple with centralizer $T \cong C_{q^{N-1}-1} \times C_{q+1}$, $\chi(x) \neq 0$ only when the unipotent character
$\chi^\al$ is labeled by one of the $N-1$ partitions $\al=(N)$, $(1^N)$, or $(N-k,2,1^{N-2-k})$ with $2 \leq k \leq N-2$, in which case 
\begin{equation}\label{for-x2}
  |\chi(x)| = 1,
\end{equation}  
cf. \cite[Corollary 3.1.2]{LST}.

Next we show that \eqref{for-x2} also holds for the $\chi$ with $a \neq b$ and $\chi(x) \neq 0$. Indeed, we will use \eqref{sum13} and look at all 
$y \in G$ where $yxy^{-1} \in L = \CB_G(s)=\GU_{N-1}(q) \times \GU_1(q) \ni x$. This implies that $y$ fixes the unique $\F_{q^2}$-line fixed by $x$ and 
hence $y \in L$. For such a $y$, the $\GU_{N-1}(q)$-component of $yxy^{-1}$ is a regular semisimple element of $\GU_{N-1}(q)$ with 
centralizer $C_{q^{N-1}-1}$, and by \cite[Corollary 3.1.2]{LST}, $\chi(x) \neq 0$ implies that the $\GU_{N-1}$-component of $\psi$ is one of 
the $N-1$ hook unipotent characters, each taking value $\pm 1$ at $yxy^{-1}$. As $\St_G(x) = \pm 1$, \eqref{sum13} implies \eqref{for-x2}.
  
\smallskip  
(b) First we consider the $\chi$ with 
$$j:=\cl(\chi) \geq \sqrt{N/13}.$$ 
Choosing $N_1$ sufficiently large, when $N \geq N_1$, by \cite[Theorem 1.2]{GLT1}
we have $\chi(1) \geq q^{c_1N^{3/2}}$ for some absolute constant $c_1 > 0$. So, by \eqref{lt-bound}, the total contribution of 
these characters to $|\Sigma^*|$ is at most
$$2N(q+1)^2q^{-c_1\sigma\sqrt{N}} < \frac{q+1}{12}$$
when $N_1$ is large enough.

Now we look at the $\chi$ with $1 \leq j \leq \sqrt{N/13}$. Suppose that 
$$m=\supp(g) \geq \sqrt{N/13}.$$ 
Since $\chi(1) \geq q^{N/2}$, by \eqref{lt-bound}, the total contribution of these characters to $|\Sigma^*|$ is at most
$$2\sqrt{N/13}(q+1)^2q^{-\sigma\sqrt{N/52}} < \frac{q+1}3$$
when $N_1$ is large enough, yielding the claim in this case.

\smallskip
(c) In the remaining case, we have $2 \leq m \leq \sqrt{N/13}$ and $j \leq \sqrt{N/13}$. Still, for each $j$ we have at most $(q+1)^2$ 
characters $\chi$ to consider, one for each pair $(a,b)$. Let $\delta_a$ denote the linear character $\hat{t}$ for 
$t=aI_N$. Then we can write $\chi=\delta_a\chi'$, where $\chi'$ now has $\cl^*(\chi')=j$ by \cite[Theorem 3.9]{GLT1}, and $\delta_a(g)=1$ as 
$g \in \SU_N(q)$. By Theorem \ref{mainB},
$$\frac{\chi(g)}{\chi(1)}=\frac{\chi'(g)}{\chi'(1)}=\e_{\chi}+(-q)^{-mj}\gamma_\lambda,$$
where $|\e_{\chi}| < q^{-N/2}$ and $\gamma_\lambda = \chi'(\lambda I_A)/\chi'(1)$ for the primary eigenvalue $\lambda$ of $g$. 

If $m \geq 3$, then
$$(q+1)^2\Bigm| \sum_{j \geq 1}(-q)^{-mj}\gamma_{\lambda}\Bigm| \leq (q+1)^2\sum^\infty_{j \geq 1}q^{-3j} = (q+1)^2/(q^3-1) \leq \frac{3(q+1)}7.$$ 
If $m=2$ but $q\geq 3$, then
$$(q+1)^2\Bigm| \sum_{j \geq 1}(-q)^{-mj}\gamma_{\lambda}\Bigm| \leq (q+1)^2\sum^\infty_{j \geq 1}q^{-2j} = (q+1)^2/(q^2-1) \leq \frac{q+1}{2}.$$ 

Suppose $m=2=q$. Recall that $g$ is a non-semisimple element of support $2$ by our convention. 
If $\lambda=1$, then $g$ has the shape \eqref{red10}, and so \eqref{sum10} is already established. 
%
So assume $\lambda \neq 1$.
Since $q=2$, we have $\lambda=\omega$, a fixed cubic root of unity in $\F_4$.  
For the three choices $(a,b) = (1,1)$, $(\omega,\omega)$, $(\omega^2,\omega^2)$, $\chi'$ is the same unipotent character, for which 
$\gamma_\lambda=1$. For the three choices $(a,b) = (1,\omega)$, $(\omega,\omega^2)$, $(\omega^2,1)$, $\chi'$ is the same character corresponding to
$s=\diag(I_{N-1},\omega) \notin \SU_N(q)$, for which $\gamma_\lambda=\zeta$,  a fixed cubic root of unity in $\C$, see \cite[Proposition 4.5]{NT}.
For the three choices $(a,b) = (\omega,1)$, $(\omega^2,\omega)$, $(1,\omega^2)$, $\chi'$ is the complex conjugate of the preceding character (and corresponding to $s=\diag(I_{N-1},\omega^2)$), for which $\gamma_\lambda=\bar\zeta$. Hence the sum of the terms 
$(-q)^{mj}\gamma_\lambda$ over these nine characters for any fixed $j$ is $0$.

Hence, choosing $N_1$ large enough, for all possible $(m,q)$, we have
$$(q+1)^2\Bigm| \sum_{1 \leq j \leq \sqrt{N/13}}(-q)^{-mj}\gamma_\lambda\Bigm| \leq (q+1)\Bigl(\frac{1}{2}+\frac{1}{24}\Bigr)$$
and 
$$(q+1)^2\Bigm| \sum_{1 \leq j \leq \sqrt{N/13}}\eps_j\Bigm| \leq (q+1)^2q^{-N/2}\sqrt{N/13} \leq \frac{q+1}{24}.$$
It follows that
$$\Bigm| \sum_{1 \leq j=\cl(\chi) \leq \sqrt{N/13}}\frac{|\chi(x)|^2\chi_j(g)}{\chi_j(1)}\Bigm| \leq (q+1)\Bigl(\frac{1}{2}+\frac{1}{12}\Bigr) = \frac{7(q+1)}{12},$$
and so we are done in this case as well.
\end{proof}

\begin{lem}\label{type2b}
There is an absolute constant $N_2>0$ such that when $N \geq N_2$, the total contribution of characters of type {\rm (II)} to $\Sigma^*$ has absolute value 
$< 1/6$. 
\end{lem}

\begin{proof}
Each of these characters can be written in the form \eqref{sum12}, where $L=\GL_p(q^2) \times \GU_1(q)$ is unique up to conjugacy in $G$. 
For the element $x$, we can again use the general bound \eqref{for-x}.
There are fewer than $(q+1)q^2 < 400$ choices for the element $s$ (up to conjugacy). Furthermore, by
\eqref{sum13}, the condition $\chi(x) \neq 0$ implies that (after conjugating $x$ suitably) $x \in L$ and $\psi(x) \neq 0$. Viewed as an element in 
$L$, the $\GL_p(q^2)$-component of $x$ is again regular semisimple with centralizer $C_{q^{2p}-1}$. Hence the $\GL_p(q^2)$-component of the unipotent character $\psi$ is labeled by one of the $p$ hook partitions of
$p=(N-1)/2$. Thus there are fewer than $200N$ characters of type (II) with $\chi(x) \neq 0$. Each of them has degree
$$\chi(1) \geq [G:L]_{q'} = (q^3+1)(q^5+1) \ldots (q^{N-2}+1)(q^N+1) > q^{(N+1)^2/4-1}.$$
Hence, by \eqref{lt-bound}, the total contribution of these characters to $|\Sigma^*|$ is at most
$$200N^3q^{-\sigma(N/4+1/2-3/4N)} < \frac16$$
when $N \geq N_2$ and $N_2$ is large enough.
\end{proof}

To complete the proof of \eqref{sum21} and of Theorem \ref{su2p1}, it remains to prove:

\begin{lem}\label{type34b}
There is an absolute constant $N_3>0$ such that when $N \geq N_3$, the total contribution of characters of types {\rm (III)} and {\rm (IV)} to $\Sigma^*$ has absolute value $< 1/6$. 
\end{lem}

\begin{proof}
Since $\supp(g) \geq 2$, by Lemma \ref{cent} we have
$$|\chi(g)| \leq \sqrt{|\CB_G(g)|} \leq q^{N^2/2-1.7N +2.5}$$
for any $\chi$; and \cite[Corollary~7.6]{LTT} still gives the bound \eqref{for-x} for $|\chi(x)|$. 

Each $\chi$ of type (III) can be written in the form \eqref{sum12}, 
where $L=\GU_2(q^p) \times \GU_1(q)$ is unique up to conjugacy in $G$, and  
$$\begin{aligned}\chi(1) & \geq [G:L]_{q'}\\ 
   & = (q^2-1)(q^3+1)\ldots (q^{p-1}-1) \cdot (q^{p+1}-1)(q^{p+2}+1) \ldots (q^{2p-1}+1)(q^{2p+1}+1)\\
   &  > q^{N^2/2-N-1/2}\end{aligned}$$
by \cite[Lemma 4.1(iii)]{LMT}. 
There are fewer than $(q+1)q^{(N-1)/2}$ choices for $s$, and two choices for the 
unipotent character $\psi$. Thus there are fewer than $2(q+1)q^{(N-1)/2}<q^{(N+5)/2}$ choices for $\chi$ of type (III). Hence the total contribution of 
these characters to $|\Sigma^*|$ is at most
$$N^2q^{(N+5)/2}q^{-0.7N+3} \leq N^22^{5.5-0.2N} < \frac1{12}$$
when $N \geq N_3$ and $N_3$ is large enough.

Each $\chi$ of type (IV) can be written in the form \eqref{sum12}, 
where $L=\GL_1(q^{2p}) \times \GU_1(q)$ is unique up to conjugacy in $G$, and  
$$\chi(1) \geq [G:L]_{q'} = (q^2-1)(q^3+1) \ldots (q^{2p-1}+1)(q^{2p+1}+1) > q^{(N^2-N)/2-1}.$$
There are fewer than $(q+1)q^{N-1} < q^{N+1}$ choices for $s$, and one choice for the 
unipotent character $\psi$. Hence the total contribution of these characters to $|\Sigma^*|$ is at most
$$N^2q^{N+1}q^{-1.2N+3.5}  \leq N^22^{4.5-0.2N}< \frac1{12}$$
when $N \geq N_3$ and $N_3$ is large enough.
\end{proof}

\section{Thompson's conjecture for $\PSU_{2p+2r+\kappa}(q)$}
Let $p,r \geq 3$ be primes, $q$ be any prime power, and $\kappa \in \{0,1\}$. In this section we show that Thompson's conjecture
holds for the simple group $\mathrm{PSU}_{2p+2r+\kappa}(q)$ as long as $p$ is sufficiently large, and $r$ is suitably chosen. 

We work inside the general unitary group $G = \GU_N(q)$ where 
$$N:=2p+2r+\kappa,$$ 
and consider any semisimple element 
$$x =\diag(x_{2p},x_{2r},x_\kappa) \in G,$$ 
such that $x_{2p} \in \GU_{2p}(q)$ has some  fixed
$\xi \in \F_{q^p} \smallsetminus \F_q$ as an eigenvalue,  $x_{2r} \in \GU_{2r}(q)$ has some  fixed
$\xi' \in \F_{q^r} \smallsetminus \F_q$ as an eigenvalue. Furthermore, $x_\kappa$ occurs only when $\kappa=1$, in which case $x_\kappa = I_1$. This ensures that $x$ is regular semisimple, real, and 
of determinant $1$.

\begin{thm}\label{su2pr}
There exist explicit absolute constants $A_3 > A_2 > A_1> 0$ such that the following statements hold 
for any prime $p \geq A_3$, any prime $r$  in the interval $[A_1,A_2]$, any prime power $q \leq 7$, and any $\kappa \in \{0,1\}$.
\begin{enumerate}[\rm(i)]
\item The interval $[A_1,A_2]$ contains prime integers that are congruent to any fixed congruence class 
$\ell$ modulo $120$ with $\gcd(\ell,120)=1$.
\item If $C$ denotes the conjugacy class of the image of the element $x$ described above in the simple group 
$S := \mathrm{PSU}_{2p+2r+\kappa}(q)$, then $S=C^2$. In fact, $(x^{\SU_{2p+2r+\kappa}(q)})^2$ contains all unipotent elements and all
elements of support $\ge 2$.
\end{enumerate}
\end{thm}

Certainly, condition \ref{su2pr}(i) holds when both $A_1$ and $A_2-A_1$ are large enough.
The rest of the section is devoted to the proof of Theorem \ref{su2pr}(ii).

The chosen element $x$ has 
$$T:=\CB_G(x) = T_1 \times T_2 \times T_3, \mbox{ where }T_1=\GL_1(q^{2p}) \cong C_{q^{2p}-1} \mbox{ and }T_2 = \GL_1(q^{2r}) \cong C_{q^{2r}-1};$$
furthermore, $T_3$ occurs only when $\kappa=1$, in which case $T_3 = \GU_1(q)$. 
In particular, $T \cap \SU_N(q)$ has index $q+1$ in $T$, and hence the conjugacy class $x^G$ is a single $\SU_N(q)$-conjugacy class. 
As $x$ is real, it suffices to show that, any non-central element $g \in \SU_N(q)$ 
is a product of two $G$-conjugates of $x$, except possibly in the case $g=zh$ where $1 \neq z \in \ZB(\SU_N(q))$ and $h$ is a transvection. Using \cite[Lemma 5.1]{GT1} as before, we may assume that $g$ is not semisimple, and furthermore  $g$ is a transvection if $\supp(g)=1$.

Note that the statement also follows in the case
\begin{equation}\label{red30}
  \begin{aligned}g = \diag(g',I_{2r+\kappa}) \in \SU_{2p}(q) \times \SU_{2r+\kappa}(q) \mbox{ with }g' \in \SU_{2p}(q)\\ \mbox{but }
  g' \neq zh, \mbox{ with }1 \neq z \in \ZB(\SU_{2p}(q)) \mbox{ and }h \mbox{ is a 
  transvection}.\end{aligned}
\end{equation}    
Indeed, as $g$ is not semisimple, $g'$ is not central in $\SU_{2p}(q)$. We may put $x$ in $\SU_{2p}(q)$, and then apply Theorem \ref{su2p} to the element $g'$.
Therefore, in what follows we may assume that  
$$m=\supp(g) \geq 2.$$

As before, we need to show \eqref{sum10}.
We identify the dual group $G^*$ with $G$ and express any irreducible character $\chi$ in the Lusztig series labeled by a semisimple $s \in G$ such that $\CB_G(s)=L$  in the form \eqref{sum12}.
Assume now that $\chi(x) \neq 0$. As $x$ is regular semisimple, 
\eqref{sum13} implies that, up to conjugacy, $x \in L=\CB_G(s)$. It follows that $s \in \CB_G(x)=T$; 
write
$$s=\diag(s_{2p},s_{2r},s_\kappa),$$
where $s_\kappa$ occurs only when $\kappa=1$.
Now using the primality of $p$ and of $r$, we see that 
the characters $\chi \in \Irr(G)$ that are non-vanishing at $x$ are divided into the following types:
\begin{enumerate}[\rm(I)]
\item Characters $\chi$ in \eqref{sum12} with $s_{2p} \in T_{11}$, where $T_{11}$ is the unique subgroup $C_{q+1}$ of $T_1$.
\item Characters $\chi$ in \eqref{sum12} with $s_{2p} \in T_{12} \setminus T_{11}$, where $T_{12}$ is the unique subgroup $C_{q^2-1}$ of $T_1$ (in particular, $q > 2$). 
Either
\begin{equation}\label{cent2a}
  \CB_G(s) = \GL_p(q^2) \times \CB_{\GU_{2r+\kappa}(q)}(s)
\end{equation}  
or
\begin{equation}\label{cent2b}
  \CB_G(s) = \GL_{p+r}(q^2) \times T_3,
\end{equation}  
depending on whether the set of eigenvalues of $s_{2r}$ is different from or the same as the set of eigenvalues of $s_{2p}$.
\item Characters $\chi$ in \eqref{sum12} with $s_{2p} \in T_{13} \setminus T_{11}$, where $T_{13}$ is the unique subgroup $C_{q^p+1}$ of $T_1$. For any such $s$, 
\begin{equation}\label{cent3}
  \CB_G(s) = \GU_2(q^p) \times \CB_{\GU_{2r+\kappa}(q)}(s).
\end{equation}
\item Characters $\chi$ in \eqref{sum12} with $s_{2p} \in T_1 \setminus (T_{12} \cup T_{13})$. For any such $s$,
\begin{equation}\label{cent4}
  \CB_G(s) = T_1 \times \CB_{\GU_{2r+\kappa}(q)}(s).
\end{equation}  
\end{enumerate} 
Here $\CB_{\GU_{2r+\kappa}(q)}(s)$ is short for $\CB_{\GU_{2r+\kappa}(q)}(\diag(s_{2r},s_\kappa))$.

We note that there is an explicit function $f(A_2)$ (which can be taken to be the class number of 
$\SSS_{2A_2+1}$) such that the Levi subgroup $\CB_{\GU_{2r+\kappa}(q)}(s)$ of $\GU_{2r+\kappa}(q)$ 
has at most $f(A_2)$ unipotent characters.  We denote by $T_{21}$  the unique subgroup $C_{q+1}$ of $T_2$.

\smallskip
The Frobenius sum includes $q+1$ linear characters which all take value $1$ at any element in 
$\SU_N(q)$. Hence it suffices to show
\begin{equation}\label{sum31}
  \Bigm|\sum\nolimits^*\frac{|\chi(x)|^2\chi(g)}{\chi(1)}\Bigm| < q+1,
\end{equation} 
where we again refer to the summation in \eqref{sum31},
which is restricted to $\chi\in\Irr(G)$ such that $\chi(1)>1$, $\chi(x)\neq 0$, and $\chi(g)\not\in [0,\infty)$,
 as $\Sigma^*$ and frequently use the character bound \eqref{lt-bound}. For the element $x$, we have the bound
\begin{equation}\label{for-x3}
  |\chi(x)| \leq 4pr
\end{equation} 
which follows from \cite[Corollary 7.6]{LTT}.

\begin{prop}\label{type1c}
The following statement holds when $A_1$ is large enough.  
For each $A_2 > A_1$, there is an absolute constant $N_1>0$ (depending on
$A_2$) such that when $A_1 \leq r \leq A_2$ and $N \geq N_1$, the total contribution of characters of type {\rm (I)} to $\Sigma^*$ has absolute value 
$< 7(q+1)/8$. 
\end{prop}

\begin{proof}
(a) We will choose $N \geq N_1 >8A_2+1$ so that $p > 3r$. 
The elements $s$ that label the characters $\chi$ of type (I) are 
$$s=\diag(aI_{2p},s_{2r},cI_\kappa)$$
where $a$ and, if $\kappa=1$, also $c$, belong to  $\mu_{q+1} \subseteq \F_{q^2}^\times$, and $s_{2r} \in \GL_1(q^{2r})$. 

First we show that 
\begin{equation}\label{for-x4}
  |\chi(x)| = 1
\end{equation}  
if $\chi(x) \neq 0$ and $s_{2r} \in T_{21}$, that is, $s_{2r}=bI_{2r}$ with $b \in \mu_{q+1}$.

Suppose that $a=b$ and furthermore $a=c$ if $\kappa=1$. Then $\chi$ is a unipotent character $\chi^\al$ of $G$ times a linear character. 
Recall that $x$ is regular semisimple with centralizer $T \cong C_{q^{2p}-1} \times C_{q^{2r}-1} \times C_{q+1}$ 
if $\kappa=1$ and $T \cong C_{q^{2p}-1} \times C_{q^{2r}-1}$ 
if $\kappa=0$. Hence $\chi(x) \neq 0$ only when the unipotent character
$\chi^\al$ is labeled by one of the partitions prescribed by Proposition \ref{value1} or Proposition 
\ref{value2} for $\kappa=0$ or $\kappa=1$ respectively, and in either case,
\eqref{for-x4} holds.

Next suppose that $\kappa=1$ and $a=b \neq c$. We now use \eqref{sum13} and look at all 
$y \in G$ where $yxy^{-1} \in L = \CB_G(s)=\GU_{N-1}(q) \times \GU_1(q) \ni x$. This implies that $y$ fixes the unique $\F_{q^2}$-line fixed by $x$ in $\F_{q^2}^N$, and 
hence $y \in L$. For such an $y$, the $\GU_{N-1}(q)$-component of $yxy^{-1}$ is a regular semisimple element of $L$ with 
centralizer $C_{q^{2p}-1} \times C_{q^{2r}-1}$. Now $\chi(x) \neq 0$ implies that the $\GU_{N-1}$-component of 
$\psi$ is labeled by one of the partitions prescribed by Proposition \ref{value1} and takes value $\pm 1$ at 
$yxy^{-1}$. As $\St_G(x) = \pm 1$, \eqref{sum13} implies \eqref{for-x4}.

Next suppose that $\kappa=0$ and $a \neq b$. We again use \eqref{sum13} and look at all 
$y \in G$ where $yxy^{-1} \in L = \CB_G(s)=\GU_{2p}(q) \times \GU_{2r}(q) \ni x$. Since $x$ fixes a unique 
non-degenerate subspace $M$ of dimension $2r$ in $\F_{q^2}^N$, this implies that $y$ fixes $M$ and 
hence $y \in L$. For such a $y$, the $\GU_{2p}(q)$-component of $yxy^{-1}$ is a regular semisimple element of $\GU_{2p}(q)$ with centralizer $C_{q^{2p}-1}$, and the $\GU_{2r}(q)$-component of $yxy^{-1}$ is a regular semisimple element of $\GU_{2r}(q)$ with centralizer $C_{q^{2r}-1}$. Now $\chi(x) \neq 0$ implies that the $\GU_{2p}(q)$-component of 
$\psi$ is labeled by one of the $2p$ hook partitions of $2p$, and the $\GU_{2r}(q)$-component of 
$\psi$ is labeled by one of the $2r$ hook partitions of $2r$, each taking value $\pm 1$ at the corresponding 
component of  $yxy^{-1}$. As $\St_G(x) = \pm 1$, \eqref{sum13} implies \eqref{for-x4}.

The arguments in the preceding paragraph also apply to each of the remaining possibilities when $a \neq b$ and $\kappa=1$: 

$\bullet$ $c=a \neq b$;

$\bullet$ $c=b \neq a$; and 

$\bullet$ $c \notin \{a,b\}$.\\
In each of the above possibilities with $s_{2r} \in T_{21}$, our analysis also shows that 
the number of $\chi$ with $\chi(x) \neq 0$ is at most $(q+1)^3(2p)(2r) \leq 2^{10}NA_2$.
  
If $s_{2r} \notin T_{21}$, then $|\chi(x)| \leq 4pr < 2NA_2$ by \eqref{for-x4}, and the number of $\chi$ with 
$\chi(x) \neq 0$ is at most $(q+1)^2q^{2r}(2p+\kappa)f(A_2) < 2^{6}Nq^{2A_2}f(A_2)$ (here, $2p+\kappa$ accounts 
for the hook unipotent characters of $\GU_{2p}(q)$ or $\GU_{2p+\kappa}(q)$).

Thus the total number of $\chi$ of type (I) with $\chi(x) \neq 0$ is at most $65Nq^{2A_2}f(A_2)$ when $A_2$ is large enough.  
  
\smallskip  
(b) First we consider the $\chi$ with 
$$j:=\cl(\chi) \geq \sqrt{N/13}.$$ 
Choosing $N_1$ sufficiently large, when $N \geq N_1$, by \cite[Theorem 1.2]{GLT1}
we have $\chi(1) \geq q^{c_1N^{3/2}}$ for some absolute constant $c_1 > 0$. So, by \eqref{lt-bound}, the total contribution of 
these characters to $|\Sigma^*|$ is at most
$$260N^3A_2^2f(A_2)q^{2A_2}q^{-c_1\sigma\sqrt{N}} < \frac{q+1}{12}$$
when $N_1$ is large enough compared to $A_2$.

Now we look at the $\chi$ with $1 \leq j \leq \sqrt{N/13}$. Suppose that 
$$m=\supp(g) \geq \sqrt{N/13}.$$ 
Since $\chi(1) \geq q^{N/2}$, by \eqref{lt-bound}, the total contribution of these characters to $|\Sigma^*|$ is at most
$$260N^3A_2^2f(A_2)q^{2A_2}q^{-\sigma\sqrt{N/52}} < \frac{q+1}{2}$$
when $N_1$ is large enough compared to $A_2$, yielding the claim in this case.

\smallskip
(c) In the remaining case, we have $2 \leq m \leq \sqrt{N/13}$ and $j \leq \sqrt{N/13}$. 
Since $p \geq r+1$ and $s_{2p}=aI_{2p}$, $a$ is the primary eigenvalue of $s$.
Let $\delta_a$ denote the linear character $\hat{t}$ of $G$ for 
$t=aI_N$. Then we can write 
$$\chi=\delta_a\chi',$$ 
where $\chi'$ now has $\cl^*(\chi')=j$ by \cite[Theorem 3.9]{GLT1}, and $\delta_a(g)=1$ as 
$g \in \SU_N(q)$. By Theorem \ref{mainB},
\begin{equation}\label{341}
  \frac{\chi(g)}{\chi(1)}=\frac{\chi'(g)}{\chi'(1)}=\e_{\chi}+(-q)^{-mj}\gamma_\lambda,
\end{equation}  
where $|\e_{\chi}| < q^{-N/2}$ and $\gamma_\lambda = \chi'(\lambda  I_A)/\chi'(1)$ for the primary eigenvalue $\lambda$ of $g$. 

\smallskip
Note that the total contribution of the error terms $\e_\chi$ in \eqref{341} is at most
\begin{equation}\label{342}
  260N^3A_2^2f(A_2)q^{2A_2}q^{-N/2} < \frac{q+1}{12}
\end{equation}  
when $N_1$ is large enough compared to $A_2$. So we need to bound the total contribution of the 
main term $(-q)^{-mj}\gamma_\lambda$ in \eqref{341}; call it the total {\it major} contribution.

\smallskip
First we bound the major contribution coming from those $\chi$ where either $s_{2r} \notin T_{21}$ or $s_{2r}=bI_{2r}$ with
$a \neq b \in \mu_{q+1}$. Then $a$ is not an eigenvalue of $s_{2r}$, and this already implies that
$$j \geq 2r.$$
By \cite[Theorem 3.9]{GLT1}, 
$j$ is determined by the partition $\al$, of $2p$ if $\kappa=0$ or if $\kappa=1$ but $c \neq a$, and of 
$2p+1$ if $\kappa=1$ and $a=c$, where $\al$ labels the corresponding component $\psi^\al$ of $\psi$ 
in the subgroup $\GU_{2p}(q)$, respectively $\GU_{2p+1}(q)$, of $L=\CB_G(s)$, 
cf. the expression \eqref{sum12} for $\chi$. In the former case, the unipotent character $\psi^\al$ is nonzero
at the regular semisimple $x_{2p}$, whence
$\al$ is a hook partition of $2p$ and uniquely determined by $j$. In the latter case, $\psi^\al$ is 
nonzero at $\diag(x_{2p},x_\kappa)$ of type $(2p,1)$. In this case
$\al$ is one of $2p$ partitions of $2p+1$, see \cite[Corollary 3.1.2]{LST}, and is again uniquely determined by $j$.
We have therefore shown that, for each choice of $s$ with $s_{2r} \notin T_{21}$, there is a unique $\chi$ of 
level $j$. As in (a), $\psi^\al$ takes value $\pm 1$ at the corresponding component of $x$.
Inducing the $\CB_{\GU_{2r}(q)}(s_{2r})$-component of $\psi$ to $\GU_{2r}(q)$ and using the fact 
that $x_{2r}$ is regular semisimple with centralizer $C_{q^{2r}-1}$ in $\GU_{2r}(q)$, by \cite[Corollary 7.6]{TT} 
we have $|\chi(x)| \leq 2r$ using \eqref{sum13}. Also, there are at most $(q+1)^2q^{2r}$ choices for $s$.
Hence, the total major contribution coming from these $\chi$ is at most
$$4r^2(q+1)^2q^{2r}\Bigm| \sum_{j \geq 2r}(-q)^{-mj}\gamma_{\lambda}\Bigm| \leq 4r^2(q+1)^2q^{2r}\sum^\infty_{j \geq 2r}q^{-2j} = \frac{4r^2(q+1)^2}{q^{2r-2}(q^2-1)} \leq (q+1)/12$$ 
when $r \geq A_1$ and $A_1$ is large enough.

\smallskip
(d) It remains to bound the major contribution coming from those $\chi$ with 
\begin{equation}\label{343}
  s=\diag(aI_{2p+2r},cI_\kappa)
\end{equation}  
where $a$ and, if $\kappa=1$, also $c$, belong to $\mu_{q+1}$. Here, $\CB_G(s) = \GU_{2p+2r}(q) \times T_3$ if $\kappa=0$ or if $\kappa=1$ but 
$a \neq c$, and 
$\CB_G(s)=\GU_{2p+2r+1}(q)$ if $\kappa=1$ and $a = c$. By \eqref{for-x4} we have $|\chi(x)|=1$.
By \cite[Theorem 3.9]{GLT1},  $j$ is determined by the partition $\al$ of $2p+2r$, respectively $2p+2r+1$, where $\al$ labels the corresponding component $\psi^\al$ of $\psi$ 
in the subgroup $\GU_{2p+2r}(q)$, respectively $\GU_{2p+2r+1}(q)$, of $L=\CB_G(s)$. 
In the former case, the unipotent character $\psi^\al$ is nonzero
at the regular semisimple element $\diag(x_{2p},x_{2r})$ of type $(2p,2r)$. By Proposition \ref{value1}, for a given $j$ the number of
$\al$ is $\leq 1$ if $1 \leq j \leq 2r-1$ and at most $2r^2+r+1$ in general.
In the latter case, $\psi^\al$ is 
nonzero at the regular semisimple $\diag(x_{2p},x_{2r},x_\kappa)$ of type $(2p,2r,1)$. By Proposition \ref{value2}, for a given $j$ the number of $\al$ is $\leq 1$ if $1 \leq j \leq 2r-1$ and at most $(2r+1)^2$ in general.

Now, the total major contribution coming from those $\chi$ subject to \eqref{343} and of level 
$2r \leq j \leq \sqrt{N/13}$ is at most  
$$\begin{aligned}(2r+1)^2(q+1)^2\Bigm| \sum_{j \geq 2r}(-q)^{-mj}\gamma_{\lambda}\Bigm| & \leq    
   (2r+1)^2(q+1)^2\sum^\infty_{j \geq 2r}q^{-2j}\\ 
   & =   \frac{(2r+1)^2(q+1)^2}{q^{2r-2}(q^2-1)}\\ 
   & \leq \frac{q+1}{8}\end{aligned}$$
when $r \geq A_1$ and $A_1$ is large enough.
    
The total major contribution coming from those $\chi$ subject to \eqref{343} and $1 \leq j \leq 2r-1$ is at most  
$$(q+1)^{1+\kappa}\Bigm| \sum_{1 \leq j < 2r}(-q)^{-mj}\gamma_{\lambda}\Bigm|  \leq    
   (q+1)^{1+\kappa}\sum^\infty_{j \geq 1}q^{-mj} \leq \frac{(q+1)^{1+\kappa}}{q^{m-2}(q^2-1)} \leq \frac{q+1}{2}$$
if $\kappa=0$, or if $\kappa=1$ but $m \geq 3$, or if $(\kappa,m)=(1,2)$ but $q \geq 3$.

\smallskip
Suppose $\kappa=1$ and $m=2=q$. Recall that $g$ is a non-semisimple element of support $2$ with primary eigenvalue 
$\lambda$ by our convention. 
If $\lambda=1$, then $g$ has the shape \eqref{red30} and so \eqref{sum10} is already established. 
So assume that $\lambda \neq 1$.
Since $q=2$, we have $\lambda=\omega$, a fixed cubic root of unity in $\F_4$.  
For the three choices $(a,c) = (1,1)$, $(\omega,\omega)$, $(\omega^2,\omega^2)$, $\chi'$ is the same unipotent character, for which 
$\gamma_\lambda=1$. For the three choices $(a,c) = (1,\omega)$, $(\omega,\omega^2)$, $(\omega^2,1)$, $\chi'$ is the same character corresponding to
$s=\diag(I_{N-1},\omega) \notin \SU_N(q)$, for which $\gamma_\lambda=\zeta$,  a fixed cubic root of unity in $\C$, see \cite[Proposition 4.5]{NT}.
For the three choices $(a,b) = (\omega,1)$, $(\omega^2,\omega)$, $(1,\omega^2)$, $\chi'$ is the complex conjugate of the preceding character (and corresponding to $s=\diag(I_{N-1},\omega^2)$), for which $\gamma_\lambda=\bar\zeta$. Hence the sum of the terms 
$(-q)^{mj}\gamma_\lambda$ over these nine characters for any fixed $j$ is $0$.

\smallskip
Hence, in all cases, choosing $A_1$ large enough, and $N_1$ large enough compared to $A_2$, from 
the characters $\chi$ subject to \eqref{343} of level $\leq \sqrt{N/13}$ 
the major terms contribute at most $5(q+1)/8$. It follows that
$$\Bigm| \sum_{1 \leq j=\cl(\chi)}\frac{|\chi(x)|^2\chi_j(g)}{\chi_j(1)}\Bigm| \leq (q+1)
   \Bigl(\frac{5}{8}+\frac{1}{12}+\frac{1}{12}+\frac{1}{12}\Bigr) = \frac{7(q+1)}8,$$
and so we are done.
\end{proof}

\begin{lem}\label{type2c}
For any constant $A_2 \geq 3$,  there is an absolute constant $N_2>0$ (depending on $A_2$) such that when 
$r \leq A_2$ and $N \geq N_2$, the total contribution of characters of type {\rm (II)} to $\Sigma^*$ has absolute value 
$< 1/6 < (q+1)/16$. 
\end{lem}

\begin{proof}
We will choose $N_2 > 4A_2+1$ so that $p > r$.
There are fewer than $q^2$ choices for $s_{2p}$, at most 
$q^{2r}$ choices for $s_{2r}$ (up to conjugacy), 
and at most $q+1$ choices for the element $s_\kappa$.  By
\eqref{sum13}, the condition $\chi(x) \neq 0$ implies that (after conjugating $x$ suitably) $x \in L:=\CB_G(s)$ and 
$\psi(x) \neq 0$. Viewed as an element in  $L$, the $\GL_p(q^2)$-component of $x$ in the case of \eqref{cent2a} is again regular semisimple with centralizer $C_{q^{2p}-1}$. Hence the $\GL_p(q^2)$-component of the unipotent character $\psi$ is labeled by one of the $p$ hook partitions of
$p<N/2$, and the number of possibilities for the other component of $\psi$ is bounded by the aforementioned function $f(A_2)$. In the case of \eqref{cent2b}, the $\GL_{p+r}(q^2)$-component of $x$ is regular semisimple with centralizer $C_{q^{2p+2r}-1}$, whence the $\GL_{p+r}(q^2)$-component of the unipotent character $\psi$ is labeled by one of the $p+r$ hook partitions of $p+r \leq N/2$; also, $\psi$ is uniquely determined by
this component.

Thus there are fewer than $(q+1)q^{2r+2}f(A_2)N/2$ characters of type (II) with $\chi(x) \neq 0$. Each of them has degree
$$\chi(1) \geq [G:L]_{q'} > q^{cN^2}$$
for some absolute constant $c > 0$. Also, $|\chi(x)| \leq 4pr < 2NA_2$ by \eqref{for-x3}. 
Hence, by \eqref{lt-bound}, the total contribution of these characters to $|\Sigma^*|$ is at most
$$2(q+1)q^{2r+2}f(A_2)N^3A_2^2q^{-\sigma cN} < \frac16$$
when $N \geq N_2$ and $N_2$ is large enough compared to $A_2$.
\end{proof}

To complete the proof of \eqref{sum31} and of Theorem \ref{su2pr}, it remains to prove:

\begin{lem}\label{type34c}
For any constant $A_2 \geq 3$,  there is an absolute constant $N_3>0$ (depending on $A_2$) such that when 
$r \leq A_2$ and $N \geq N_3$, the total contribution of characters of types {\rm (III)} and {\rm (IV)} to $\Sigma^*$ has absolute value $< 1/6 < (q+1)/16$. 
\end{lem}

\begin{proof}
We will choose $N_3 > 4A_2 +1$ so that $p > r$.
Since $\supp(g) \geq 2$, by Lemma \ref{cent} we have
$$|\chi(g)| \leq \sqrt{|\CB_G(g)|} \leq q^{N^2/2-1.7N +2.5}$$
for any $\chi$; and we still have the bound \eqref{for-x3} for $|\chi(x)|$. 

Each $\chi$ of type (III) can be written in the form \eqref{sum12}, 
where $L=\CB_G(s)$ is given in \eqref{cent3}, and  
$$\chi(1)  \geq [G:L]_{q'} > q^{N^2/2-N-c_1}$$
for some absolute constant $c_1>0$ (that depends on $A_2$).
There are fewer than $(q+1)q^{p+2r} < q^{N/2+2A_2}$ choices for $s$. Furthermore, there are two choices for the
$\GU_2(q^p)$-component of the unipotent character $\psi$, and at most $f(A_2)$ choices for its other component
(as in the proof of Lemma \ref{type34c}). Thus there are fewer than $2q^{N/2+2A_2}f(A_2)$ choices for $\chi$ of type (III). Hence the total contribution of 
these characters to $|\Sigma^*|$ is at most
$$8N^2A_2^2q^{N/2+2A_2}f(A_2)q^{-0.7N+c_1+2.5} \leq 8N^2A_2^2f(A_2)q^{-0.2N+c_1+2.5+2A_2} < \frac{1}{12}$$
when $N \geq N_3$ and $N_3$ is large enough compared to $A_2$.

Each $\chi$ of type (IV) can be written in the form \eqref{sum12}, 
where $L=\CB_G(s)$ is given in \eqref{cent4}, and  
$$\chi(1) \geq [G:L]_{q'} > q^{(N^2-N)/2-c_2}$$
for some absolute constant $c_2>0$ (that depends on $A_2$).
There are fewer than $(q+1)q^{N-1} < q^{N+1}$ choices for $s$. Furthermore, there is one choice for the
$T_1$-component of the unipotent character $\psi$, and at most $f(A_2)$ choices for its other component. Hence the total contribution of these characters to $|\Sigma^*|$ is at most
$$4N^2A_2^2f(A_2)q^{N+1}q^{-1.2N+2.5+c_2} \leq 4N^2A_2^2f(A_2)q^{-0.2N+3.5+c_2}< \frac{1}{12}$$
when $N \geq N_3$ and $N_3$ is large enough compared to $A_2$.
\end{proof}

\section{Thompson's conjecture for unitary groups}

In this section, we prove the unitary part of Theorem~\ref{mainA}.  The idea is to write $N = M + L$, where $L$ is divisible by $2q+2$, and to find a block diagonal regular semisimple element $x\in \SU_N(q)$ of the form
\begin{equation}
\label{xdef}
x:=\diag(x_1,x_2).
\end{equation}
We choose $x_1$ in such a way that  $x_1^{\SU_M(q)} x_1^{\SU_M(q)}$ includes all unipotent elements and all elements of support $\ge 2$.
We choose $x_2$ so that
\begin{equation}
\label{scalars}
x_2^{\SU_L(q)}x_2^{\SU_L(q)} \supset \ZB(\SU_L(q)),
\end{equation}
which consists of all scalar matrices of the form $\omega I$, where $\omega\in \mu_{q+1}$.
The construction of $x_1$ was accomplished in sections \S4, \S5, and \S6 for many integers $M$.
In this section, we determine the decompositions $N = M + L$ and the elements $x_2$.   We start with the second task.

Let $T$ denote a finite subgroup of $\bar{\F}_q^\times$ which is stable under the $q$-Frobenius $\Frob_q$. There is a natural action of  $T\rtimes (\Z \times \Z/2\Z)$
on $\bar{\F}_q^\times$ such that 
the orbit of $\alpha$ is
$$\cO_T(\alpha) := \{\omega\alpha^{\pm q^i}\mid \omega\in T,\,i\in\N\}.$$
Slightly abusing the language, in what follows, a {\it $T$-orbit} will mean a set of the form $\cO_T(\alpha)$.
We say a subset of $\bar\F_q^\times$ is \emph{$T$-stable} if it is stable under this action, that is, stable under $\Frob_q$, multiplication by any element of $T$,
and inversion.  This is equivalent to being a union of $T$-orbits $\cO_T(\alpha_i)$.
We say $\alpha\in\F_{q^m}$ is \emph{$T$-regular} if $\cO_T(\alpha)$
has $2m|T|$ elements.  
In this case, it consists of $2|T|$ $\Frob_q$-orbits of length $m$, with representatives $\omega \alpha,\omega^{-1}\alpha^{-1}$, $\omega \in T$, which shows
that the product of the elements in $\cO_T(\alpha)$ is $1$.
In particular, we say $\alpha$ is \emph{$\SU$-regular} if it is $\mu_{q+1}$-regular.

\begin{lem}
\label{Singular}
The number of elements of $\F_{q^m}$ which are not $\SU$-regular is $O(m^2 q^{1+m/2})$.
\end{lem}

\begin{proof}
If $\alpha \in \F_{q^m}$ is not $\SU$-regular, it must satisfy an equation of the form
\begin{equation}
\label{SU-irregular}
\omega_1\alpha^{e_1 q^{r_1}} = \omega_2 \alpha^{e_2 q^{r_2}},
\end{equation}
where $e_1,e_2\in\{-1,1\}$, $\omega_1,\omega_2\in \mu_{q+1}$, $r_1,r_2\in \Z$, $|r_1-r_2|<m$, and $(\omega_1,e_1,r_1) \neq (\omega_2,e_2,r_2)$.
Changing $r_1$ and $r_2$ within their residue classes (mod $m$) and exchanging them if necessary, we may assume
$0\le r_1 \le r_2 \le r_1+m/2$.
Replacing $\omega_2$ by $\omega_2/\omega_1$, we may assume $\omega_1=1$.  Therefore, we need only consider $O(m^2q)$ equations,
and each has at most $q^{m/2}+1$ solutions.  The lemma follows.
\end{proof}

\begin{thm}
\label{xm}
Let $q \leq 7$ be a prime power and let $N$ be an integer which is not $0 \pmod {2q+2}$ and is not $6$ or $7 \pmod{12}$ if $q=5$. 
Then there exists an arbitrarily large integer $M\equiv N \pmod{2q+2}$ and an element $x\in H:=\SU_M(q)$ 
satisfying:
\begin{itemize}
\item The element $x$ is regular semisimple,
\item The number of irreducible factors of the characteristic polynomial of $x$ as a polynomial over $\F_{q^2}$ is bounded (in fact by $5$),
\item Every unipotent element in $H$ is the product of two $H$-conjugates of $x$,
\item Every element $g\in H$ with $\supp(g)\ge 2$ is the product of two $H$-conjugates of $x$.
\end{itemize}
\end{thm}

\begin{proof}
Such elements are constructed in the proofs of Theorems~\ref{su2p}, \ref{su2p1}, and \ref{su2pr} whenever $M$ is respectively of the form $2p$, for $p$ a sufficiently large prime; $2p+1$, for $p$ a sufficiently large prime; or $2p+2r+\kappa$, for $p$ a sufficiently large prime, $r$ chosen from a set which represents every residue class (mod $120$) which is prime to $120$, and $\kappa\in \{0,1\}$.  When $q=2$, $2p$ represents $2$ and $4$ (mod $6$), and $2p+2r+1$ represents $1$, $3$, and $5$.  When $q=3$, $2p+2r$ represents $4$ (mod $8$), $2p+2r+1$ represents $1$ and $5$, $2p$ represents $2$ and $6$, and $2p+1$ represents $3$ and $7$.  
When $q=4$, $2p$ represents $2$, $4$, $6$, and $8$ (mod $10$), and $2p+2r+1$ represents $1$, $3$, $5$, $7$, and $9$.
When $q=5$, $2p+2r$ represents  $4$ and $8$ (mod $12$),
$2p+2r+1$ represents $1$, $5$, and $9$, $2p$ represents $2$ and $10$, and $2p+1$ represents $3$ and $11$.
When $q=7$, $2p+2r$ represents $4$, $8$, and $12$ (mod $16)$, $2p+2r+1$ represents $1$, $5$, $9$, and $13$, $2p$ represents $2$, $6$, $10$, and $14$, and $2p+1$ represents $3$, $7$, $11$, and $15$.
Since $p$ can be taken arbitrarily large within each residue class (mod $2q+2$) which is prime to $2q+2$, $M$ can likewise be taken arbitrarily large, and the statement holds.
\end{proof}

We can now prove the unitary part of Theorem~\ref{mainA}.

\begin{proof}
We have already seen that the main theorem of \cite{EG} allows us to assume $q\le 7$.  If $N$ is divisible by $q+1$ and sufficiently large then $\PSU_N(q)$ satisfies Thompson's conjecture by \cite[Theorem 7.7]{LT3}.   This covers all cases when $N\equiv 0\pmod{2q+2}$ as well as the case that $q=5$ and $N\equiv 6\pmod{12}$.
Also, by \cite{EG}, if $N$ is odd and $q\ge 5$, then $\PSU_N(q)$ satisfies Thompson's conjecture, which covers the case that $q=5$ and $N\equiv 7\pmod{12}$.
Therefore, it remains to handle the case $N$ belongs to a residue class (mod $2q+2$) for which Theorem~\ref{xm} applies.  It follows that there exists an arbitrarily large integer $M$ congruent to $N$ (mod $2q+2$) and an element $x_1\in \SU_M(q)$ satisfying all the properties of that theorem.  For each residue class, we fix $M$ and $x_1$.

If $N$ is sufficiently large and belongs to the same residue class as $M$, we may assume that
$N = M+(2q+2)m$, where $m>M$ is large enough that Lemma~\ref{Singular} guarantees  the existence of an $\SU$-regular element $\alpha\in \F_{q^m}$.
We set $L := (2q+2)m$.  If $X\in \SU_m(q)$ has eigenvalue $\alpha$ and $\mu_{q+1} = \langle \gamma \rangle$, then 
$$x_2 := \diag(X,\gamma X, \gamma^2 X,\ldots, \gamma^q X, X^{-1},\gamma X^{-1}, \ldots, \gamma^q X^{-1})$$
is regular semisimple.  Permuting the blocks  preserves the conjugacy class in $\SU_L(q)$, so \eqref{scalars} holds, and for all $\omega\in \mu_{q+1}$, the scalar matrix $\omega I$ is a product of
two conjugates of $x_2$.

The irreducible factors of the characteristic polynomial of $x_1$ have degree $\leq M < m$, so it cannot have an irreducible factor in common with the characteristic polynomial of $x_2$.
Therefore, the element $x$ with blocks  $x_1$ and $x_2$ defined in
\eqref{xdef} is
regular semisimple, and its characteristic polynomial has a bounded number of factors, independent of all choices.  
By \cite[Theorem 1.1]{LTT}, there exists $B>0$, independent of all choices,
such that every element $g\in \SU_N(q)$ with $\supp(g) > B$ lies in $x^{\SU_N(q)} x^{\SU_N(q)}$.  
As $B$ is independent of choices, we may assume $M > B$.

Defining $C$ to be the conjugacy class of the image of $x$ in $S$, we claim that $C^2 = \PSU_N(q)$.
To prove this claim, suppose we have any element of $\PSU_N(q)$.  We lift it to $g\in \SU_{N}(q)$ in such a way that if $g$ is a scalar multiple of a unipotent element, then the scalar is $1$.
We may assume that $\supp(g) \le B < M < L$, so $g$ has a (unique) primary eigenvalue $\lambda \in \mu_{q+1}$ whose eigenspace has dimension $\geq N-B >L$.  
It follows that $g$ is conjugate to a block diagonal matrix $(g_1,g_2)\in \SU_M(q) \times \SU_L(q) < \SU_N(q)$,
where $g_2 = \lambda I_L$ and $\lambda$ is an eigenvalue of $g_1$. Hence, if $g_1$ is a scalar multiple of a unipotent element, the scalar must be $\lambda$, so $g$ itself is $\lambda$ times a unipotent element, and so, by our choice of $g$, 
the scalar must be $1$.

By Theorem~\ref{xm} (and \cite[Lemma 5.1]{GT1} if $g_1$ is semisimple of support $1$), $g_1$ is the product of two 
$\SU_M(q)$-conjugates of $x_1$. By construction, $g_2$ is the product of two $\SU_L(q)$-conjugates of $x_2$, so we are done.
\end{proof}

\section{Unipotent characters: Lusztig symbols, level, and values}\label{sec:unip}

In this section, we recall the parametrization of the unipotent characters of classical groups of type B, C, and D.
We describe certain regular semisimple elements for which the unipotent character values always belong in $\{-1,0,1\}$
and characterize which unipotent characters take non-zero values on such elements.

The length $\len(\lambda)$ of a partition is the number of non-zero terms.
If $\lambda$ and $\mu$ are partitions, we write 
$$\lambda \preceq \mu$$ 
if $\mu_1\ge \lambda_1\ge\mu_2\ge\lambda_2\ge\cdots$.
This implies 
$\len(\mu)-1\le\len(\lambda) \le \len(\mu).$

A \emph{half-symbol} is a set of integers $A = \{a_1,a_2,\ldots,a_n\}$, where $a_1>a_2>\cdots>a_n\ge 0$.  If $a_n>0$, we say $A$ is \emph{reduced}.
We define 
$$\Sigma(A) =  \{a_1+1,a_2+1,\ldots,a_n+1,0\},$$ 
so $A$ is reduced if and only if there is no half-symbol $B$ with $A = \Sigma(B)$.
We define
\begin{equation}\label{Abar}
  \bar A = \{a_2,\ldots,a_n\}.
\end{equation}
Each $A$ determines a partition $\lambda = \pi(A)$, where
$$\lambda_1 = a_1+1-n,\ldots,\lambda_{n-1} = a_{n-1}-1,\lambda_n = a_n, \lambda_{n+1}=\cdots = 0;$$
in particular, $\pi(\bar A)$ is obtained from $\pi(A)$ by removing its largest part.
For every partition $\lambda$, there is a unique reduced half-symbol $A$ with $\pi(A) = \lambda$.  We write $A = \delta(\lambda)$.  Thus,
$$\pi^{-1}(\lambda)= \{\delta(\lambda),\Sigma(\delta(\lambda)),\Sigma^2(\delta(\lambda)),\ldots\}.$$
Note that if $\pi(C)\preceq \pi(A)$, then $\pi(\bar A)\preceq \pi(C)$, so, in particular, the number of parts of $\pi(A)$ minus the number of parts of $\pi(C)$ is $0$ or $1$.

A \emph{symbol} is an ordered pair $(A,B)$ of half-symbols.  For greater readability, 
symbols are sometimes written as fractions
$\Bigl(\frac AB\Bigr)$, with the set brackets of $A$ and $B$ omitted.
The \emph{transpose} of symbol $(A,B)$ is $(B,A)$.  The symbol $(A,B)$ is \emph{non-degenerate} if $A\neq B$.
It is \emph{reduced} if at least one of $A$ and $B$ is so.  The \emph{rank} of $(A,B)$ is 
\begin{equation}\label{rk10}
  \rank(A,B) = \sum_{a\in A} a + \sum_{b\in B} b - \Bigl\lfloor \frac {(|A|+|B|-1)^2}4\Bigr\rfloor
  = |\pi(A)|+|\pi(B)| +  \Bigl\lfloor \frac {(|A|-|B|)^2}4\Bigr\rfloor,
\end{equation}  
and its \emph{defect} is $\defect(A,B)=|A| - |B|$ (which can be negative).  Its \emph{maximum value} is 
$$\max(A,B) = \max(A\cup B).$$  

Let $(A,B)$ be a symbol of even defect and $(C,D)$ a symbol with odd defect.  We say  \emph{$(A,B)$ is linked to $(C,D)$} if one of the two following conditions hold:
\begin{equation}
\label{split conditions}
\pi(B)\preceq \pi(C),\ \pi(D)\preceq \pi(A),\ |A|-|B| = 1+|D|-|C|.
\end{equation}
\begin{equation}
\label{non-split conditions}
\pi(A)\preceq\pi(D),\ \pi(C)\preceq\pi(B),\ |A|-|B| = -1+|D|-|C|.
\end{equation}

From now on we assume  $(A,B)$ is reduced. Then
\begin{equation}
\label{max vs rank}\max(A,B)\le \rank(A,B).
\end{equation}
Indeed, we may assume without loss of generality that $0\not\in A$.  If $\max(A,B)\in A$, 
\begin{align*}
\rank(A,B) \geq 1+2+\cdots+(|A|-1)+&\max(A,B) + 0+1+\cdots+(|B|-1) \\
&-  \Bigl\lfloor \frac {(|A|+|B|-1)^2}4\Bigr\rfloor \\
&\ge \max(A,B) + \frac{(|A|-|B|)^2}4 \ge \max(A,B).
\end{align*}
If $\max(A,B)\in B$, then the rank is at least
\begin{align*}
\rank(A,B) \geq 1+2+\cdots+|A| &+ 0+1+\cdots+(|B|-2)+\max(A,B) \\
&-  \Bigl\lfloor \frac {(|A|+|B|-1)^2}4\Bigr\rfloor \\
&\ge \max(A,B) + \frac{(2+|A|-|B|)^2}4 \ge \max(A,B).
\end{align*}

Setting 
\begin{equation}\label{m10}
  m=m(A,B)=|A|+|B|-1,
\end{equation}
we define the weakly increasing sequence 
$$0\le c_0\le c_1\le c_2 \le \cdots\le c_m$$ 
obtained by merging the sets $A$ and $B$. 
In particular, $\max(A,B) = c_m$.
At least one of every consecutive pair of inequalities is strict.
We define $\epsilon_m$ to $0$ or $1$ depending on whether $m$ (or, equivalently, $\defect(A,B)$) is even or odd
and define the \emph{level} of $(A,B)$ to be
\begin{equation}\label{lev-uni}
  \lev(A,B) = \rank(A,B) - \max(A,B) + \frac{m-\epsilon_m}2 = \rank(A,B) - \max(A,B) + \bigl\lfloor\frac m2\bigr\rfloor;
\end{equation}
note that $\lev(A,B) \in \Z_{\geq 0}$ by \eqref{max vs rank}.

\begin{lem}
\label{no tie}
If $\rank(A,B) > 4\,\lev(A,B)^2 + 2\,\lev(A,B)$, then $\max(A) \neq \max(B)$.
\end{lem}

\begin{proof}
If $\max(A) = \max(B) = c_m$, then
$$\rank(A,B) = \rank(\bar A,\bar B) + 2c_m - \bigl\lfloor \frac{m^2}4\bigr\rfloor + \bigl\lfloor \frac{(m-2)^2}4\bigr\rfloor
\ge 2c_m - m + 1.$$
By \eqref{max vs rank} and \eqref{lev-uni},
$$\lev(A,B) \ge c_m-m+1+ \bigl\lfloor\frac m2\bigr\rfloor \ge c_m - \bigl\lfloor\frac m2\bigr\rfloor \ge c_m - \lev(A,B).$$
and it follows that $c_m \le 2\,\lev(A,B)$.  Thus,
$$\rank(A,B) \le \sum_{a\in A} a + \sum_{b\in B} b \le c_m^2 + c_m \le 4\,\lev(A,B)^2 + 2\,\lev(A,B).$$
\end{proof}

A unipotent character of an orthogonal group is an irreducible character whose restriction to the special orthogonal group
is a sum of unipotent characters.
There are bijective correspondences between symbols of rank $n$ whose defect is $1\pmod4$, unipotent characters of $\SO_{2n+1}(q)$,
and unipotent characters of $\Sp_{2n}(q)$.
There is a bijective correspondence between symbols of rank $n$ whose defect is $0 \pmod4$ and unipotent characters of $\GO^+_{2n}(q)$; 
the only case where this unipotent character is not irreducible over $\SO^+_{2n}(q)$ is when the symbol is degenerate.
There is a bijective correspondence between symbols of rank $n$ whose defect is $2 \pmod4$ and unipotent characters of $\GO^-_{2n}(q)$, all of which
are irreducible over $\SO^-_{2n}(q)$.
This follows from the definition of unipotent characters for the (finite disconnected reductive) group $\GO^\pm_{2n}(q)$, together with the action of 
diagonal automorphisms on unipotent characters of $\SO^\pm_{2n}(q)$, see \cite[Theorem 2.5]{Ma}. The {\it level} $\lev(\chi)$ of the unipotent character $\chi$ is 
defined to be the level of the corresponding symbol. 

If $k$ is a positive integer, then a \emph{$k$-hook} of a symbol $(A,B)$ is either an element $r\ge k$ of $A$ such that 
$r-k\not\in A$ or an element $r\ge k$ of $B$ such that $r-k\not\in B$.   
\emph{Removing} the hook means replacing $(A,B)$ with 
$(B, \{r-k\}\cup A\setminus \{r\})$ or $(\{r-k\}\cup B\setminus \{r\},A)$ respectively.
If $(A,B)$ is the symbol of a unipotent character of $\GO_{2n+1}(q)$ or $\Sp_{2n}(q)$, then removing a $k$-hook leaves
the symbol of a unipotent character of $\GO_{2n-2k+1}(q)$ or $\Sp_{2n-2k}(q)$ respectively.
If $(A,B)$ is the symbol of a unipotent character of $\GO_{2n}^\varepsilon(q)$, then removing a $k$-hook gives
the symbol of a unipotent character of $\GO_{2n-2k}^{-\varepsilon}(q)$.

By a \emph{$k$-cohook}, we mean either an element $r\ge k$ of $A$ such that $r-k\not\in B$ or an element $r\ge k$ in $B$ such that $r-k\not\in A$.
\emph{Removing} this cohook means replacing $(A,B)$ by $(B\cup\{r-k\},A\setminus\{r\})$ in the first case and by
$(B\setminus \{r\},A\cup\{r-k\})$ in the second.  As with $k$-hooks, removing a $k$-cohook reduces the rank of a symbol by $k$
and replaces $\varepsilon$ by $-\varepsilon$ in the case of even-dimensional orthogonal groups.

Let $\WWW_n = (\Z/2\Z)^n\rtimes \SSS_n$ denote the Weyl group of root systems of type $B_n$ and $C_n$ with $n \geq 4$,
which is also the automorphism group of the root system of type $D_n$ if in addition $n > 4$.
We say \emph{$w\in \WWW_n$ is a $k$-cycle} if its image in $\SSS_n$ is so.
The conjugacy classes in $\WWW_n$ are indexed by ordered pairs $(\lambda,\mu)$ of partitions, with $|\lambda|+|\mu| = n$.
In particular, every $n$-cycle is associated to $((n),\emptyset)$ or $(\emptyset,(n))$, while every $n-1$-cycle is associated to 
$((n-1,1),\emptyset)$, $((n-1),(1))$, $((1),(n-1))$, or $(\emptyset,(n-1,1))$.
If $G = \Sp_{2n}(q)$, $G = \SO_{2n+1}(q)$, or $G = \GO_{2n}^\varepsilon(q)$, let $t_n\in G$ denote any regular semisimple  element of $G^\circ=\Sp_{2n}(q)$, 
$G^\circ = \SO_{2n+1}(q)$, or $G^\circ = \SO_{2n}^\varepsilon(q)$ respectively whose centralizer in $G^\circ$ is a maximal torus $T$ associated with
an $n$-cycle in $\WWW_n$.  Thus $|T|=q^n-1$ or $|T|=q^n+1$ depending on whether $w$ is of order $n$ or $2n$.
If $G = \GO_{2n}^\varepsilon(q)$, let $t_{n-1}$ denote a regular semisimple element of $G^\circ$
whose centralizer is a maximal torus $T$ associated with an $n-1$-cycle in $\WWW_n$.
When $\varepsilon=+$, $T$ has order $(q^{n-1}-1)(q-1)$ or
$(q^{n-1}+1)(q+1)$, depending on whether $\mu$ or $\lambda$ is emptyset, while if $\varepsilon=-$ and $T$ has order $(q^{n-1}-1)(q+1)$ or
$(q^{n-1}+1)(q-1)$, depending on whether $\mu$ or $\lambda$ equals $(1)$. The existence of such regular semisimple elements follows from \cite[Lemma 2.1]{LM}. 

\begin{prop}\label{unip-regBC}
If $G$ is $\Sp_{2n}(q)$ or $\GO_{2n+1}(q)$, $n \geq 4$, and $\chi$ is the unipotent character associated with a reduced symbol $(A,B)$,
then $\chi(t_n)\in \{-1,0,1\}$.  Furthermore, $\chi(t_n)\neq 0$ only if $(A,B)$ belongs to
\begin{equation}
\label{four sets}
\begin{split}
&\Bigl\{\Bigl(\frac{1,2,\ldots,l,n}{0,1,\ldots,l-1}\Bigr)\Bigm| 0\le l < n\Bigr\}
\cup \Bigl\{\Bigl(\frac{0,1,\ldots,l+1}{1,2,\ldots,l,n}\Bigr)\Bigm| 0\le l < n\Bigr\}  \\
 &\cup\Bigl\{\Bigl(\frac{1,2,\ldots,l-1}{0,1,\ldots,l,n}\Bigr)\Bigm| 0< l < n\Bigr\}
\cup \Bigl\{\Bigl(\frac{0,1,\ldots,l,n}{1,2,\ldots,l+1}\Bigr)\Bigm| 0\le l < n\Bigr\}.
\end{split}
\end{equation}
%
%
Moreover, there are at most $2$ unipotent characters $\chi$ of any given level such that $\chi(t_n)\neq 0$.
\end{prop}

\begin{proof}
We use the Murnaghan-Nakayama rule \cite[Theorem 3.3]{LM}
for unipotent characters of groups of type $B_n$ and $C_n$ at regular semisimple elements.
By this rule, if $(A,B)$ has neither an $n$-hook nor an $n$-cohook, then $\chi(t_n)=0$.

Assume now that $\chi(t_n) \neq 0$. By the Murnaghan-Nakayama rule, either $(A,B)$ has an $n$-hook and $\lambda=(n)$, or it has an $n$-cohook and $\mu=(n)$.
Since $\max(A,B) \leq \rank(A,B)=n$ by \eqref{max vs rank}, the only possible $n$-hook or $n$-cohook is $n$ itself.

For $(A,B)$ to have an $n$-hook, it is necessary that either $n\in A$ and $0\not\in A$  or
$n\in B$ and $0\not\in B$.
In the first case, we must also have that $(\{0\}\cup \bar A,B)$ is a symbol of rank $0$ and defect $1$ (mod $4$).
The only such symbol which is reduced is $(\{0\},\emptyset)$.  Thus,
$$(\{0\}\cup \bar A,B) = (\{0,1,\ldots,l\},\{0,1,\ldots,l-1\})$$
for some $l\ge 0$, so
$$\Bigl(\frac AB\Bigr) = \Bigl(\frac{1,2,\ldots,l,n}{0,1,\ldots,l-1}\Bigr).$$
In the second case, $(A,\{0\}\cup \bar B)$ has rank $0$ and defect $1$ (mod $4$), so 
$$\Bigl(\frac AB\Bigr) = \Bigl(\frac{0,1,\ldots,l+1}{1,2,\ldots,l,n}\Bigr).$$

For $(A,B)$ to have an $n$-cohook, it is necessary that either $n\in A$ and  $0\not\in B$,
or $n\in B$ and  $0\not\in A$.  In the first case, $(\bar A, B\sqcup\{n\})$ has rank $0$ and defect $3$ (mod $4$),
so it must be equivalent to $(\emptyset,\{0\})$, meaning
$$\Bigl(\frac AB\Bigr) = \Bigl(\frac{0,1,\ldots,l,n}{1,2,\ldots,l+1}\Bigr)$$
for some $l\ge 0$.
In the second case, $(A\sqcup\{n\},\bar B)$ must be equivalent to $(\emptyset,\{0\})$, so
$$\Bigl(\frac AB\Bigr) = \Bigl(\frac{1,2,\ldots,l-1}{0,1,\ldots,l,n}\Bigr)$$
for some $l>0$.

As mentioned above, the conjugacy class of $\alpha_n$ determines whether $(A,B)$ needs an $n$-hook or an $n$-cohook for $\chi(t_n)$ to be non-zero: it needs 
an $n$-hook if $\lambda=(n)$ and it needs an $n$-cohook if $\mu=(n)$.
The first and second sets in \eqref{four sets} contain the symbols with a unique $n$-hook and no $n$-cohook, while the third and fourth contain the symbols with 
no $n$-hook and a unique $n$-cohook. It follows from the Murnaghan-Nakayama rule that $\chi(t_n)=\pm 1$. 
Within any one of these four sets, no two symbols  have the same level; indeed, the listed symbol in these four sets 
with given $l$ has level $l$, $l+1$, $l$, and $l+1$, respectively. 
It follows that the number of unipotent characters $\chi$ of given level with $\chi(t_n)\neq 0$ is at most two.
\end{proof}


\begin{prop}\label{prime-rank-D}
If $G$ is $\GO_{2n}^{\varepsilon}(q)$, $n\ge 5$, $\varepsilon = \pm$, and $\chi$ is the unipotent character associated with a symbol $(A,B)$,
then $\chi(t_n)\in \{-1,0,1\}$ and $j:=\lev(\chi) \leq n-1$.  Moreover, for a fixed $\varepsilon$ the number of unipotent characters $\chi$ of given level $j\in [0,n-1]$ with $\chi(t_n)\neq 0$ is $2$.
\end{prop}

\begin{proof}
We use the Murnaghan-Nakayama rule \cite[Theorem 3.3]{LM}
for unipotent characters of groups of type $D_n$ or $\tw2 D_n$ at regular semisimple elements.  Let $S=(A,B)$ be the reduced symbol associated with $\chi$, 
let $t_n$ be associated to the pair $(\lambda,\mu)$, and 
assume that $$\chi(t_n)\neq 0.$$ 
Then $S$ must have an $n$-hook if $n$ is a part of $\lambda$, or an $n$-cohook if $n$ is a part of $\mu$. 

Since $\max(S) \leq \rank(S)=n$ by 
\eqref{max vs rank}, the hook or cohook can only be $n$; moreover, there can be only one occurrence of $n$, since $\max(A)=\max(B) = \rank(A,B)$ can only happen for the symbol $(0,1,2,\ldots,n|1,2,\ldots,n)$ and its transpose, which are odd.  Since $S$ is reduced, $n$ cannot be both 
a $n$-hook and an $n$-cohook.  Moreover, since $0$ and $n$ each appear at most once, $m\le 2n-1$, so $\lev(\chi) \le n-1$.

Removing the $n$-hook or $n$-cohook must leave a symbol $T$ of even defect and rank $0$, which must be equivalent to the reduced symbol
$(\emptyset |\emptyset)$.  Thus,  $T = (0,1,\ldots,l-1|0,1,\ldots,l-1)$, so the possibilities for $S$ are
$$\Bigl(\frac{1,\ldots,l-1,n}{0,1,\ldots,l-1}\Bigr),\ \Bigl(\frac{0,1,\ldots,l-1}{1,\ldots,l-1,n}\Bigr)$$
if $\varepsilon=1$ and
$$\Bigl(\frac{0,1,\ldots,l-1,n}{1,\ldots,l-1}\Bigr),\ \Bigl(\frac{1,\ldots,l-1}{0,1,\ldots,l-1,n}\Bigr)$$
if $\varepsilon = -1$.
\end{proof}

\begin{prop}\label{unip-regD}
If $G$ is $\GO_{2n}^{\varepsilon}(q)$, $n\ge 5$, $\varepsilon = \pm$, and $\chi$ is the unipotent character associated with a symbol $(A,B)$,
then $\chi(t_{n-1})\in \{-1,0,1\}$ and $j:=\lev(\chi) \leq n-1$.  Moreover, for a fixed $\varepsilon$ the number of unipotent characters $\chi$ of given level $j$ with $\chi(t_{n-1})\neq 0$ is 
$2$ if $j = 0,1$, and at most $4$ if $2 \leq j \leq n-1$.
\end{prop}

\begin{proof}
We again use the Murnaghan-Nakayama rule \cite[Theorem 3.3]{LM}.
Let $S=(A,B)$ be the reduced symbol associated with $\chi$, and 
assume that $$\chi(t_{n-1})\neq 0.$$ 
Then $S$ must have an $(n-1)$-hook if $n-1$ is a part of $\lambda'$, or an $(n-1)$-cohook if $n-1$ is a part of $\mu'$. 
Since $\max(S) \leq \rank(S)=n$ by 
\eqref{max vs rank}, the hook or cohook can only be $n-1$ or $n$.  Removing this hook or cohook must leave a symbol $T$ of even defect and rank $1$.
Thus, $T = (\Sigma^l(C),\Sigma^l(D))$ for some reduced symbol $(C,D)$ of even defect and rank $1$.  
The possibilities for $(C,D)$ are as follows:
%
%
\begin{equation}
\label{rank 1 level 1}
(\{0,1\},\emptyset),\,(\{0\},\{1\}),\,(\{1\},\{0\}),\,(\emptyset,\{0,1\}).
\end{equation}
Each of these symbols contains either a single $1$-hook and no $1$-cohook or a single $1$-cohook and no $1$-hook, and all of them are non-degenerate.

\smallskip
(a) Suppose $l=0$. Then for each $T$ in \eqref{rank 1 level 1} we have $\max(T)=1$ and $m(T)=1$. If $S$ is obtained from $T$ by adding an $(n-1)$-hook
$n$, then $\max(S)=n$, $m(S)=m(T)=1$, and hence $\lev(S) = 0$ by \eqref{lev-uni}; varying $T$ we obtain $2$ symbols $S$ of defect $0$ (mod $4$) and $2$ 
symbols $S$ of defect $2 \pmod4$. The same holds if $S$ is obtained from $T$ by adding an $(n-1)$-cohook
$n$.  If $S$ is obtained from $T$ by adding an $(n-1)$-hook $n-1$, then $\max(S)=n-1$, $m(S)=m(T)=1$, and hence $\lev(S) = 1$ by \eqref{lev-uni}; varying $T$ we obtain $2$ symbols $S$ of defect $0$ and $2$ 
symbols $S$ of defect $2 \pmod4$. The same holds if $S$ is obtained from $T$ by adding an $(n-1)$-cohook
$n-1$. 
So fixing the defect of $S$ modulo $4$ and fixing its level (between $0$ and $1$), we have exactly $2$ possibilities 
for $S$ by hook addition, and $2$ by cohook addition. Furthermore, each of these $S$ either has a single entry $n$ and no occurrence of $n-1$, or  has a single entry $n-1$ and no $n$. Hence $S$ has at most one 
$(n-1)$-hook and at most one $(n-1)$-cohook. It follows from the Murnaghan-Nakayama rule that $\chi(t_{n-1}) = \pm 1$.

\smallskip
(b) From now on we assume $l > 0$.
For $(C,D)$ in \eqref{rank 1 level 1} and each positive integer $m$, $(\Sigma^m(C),\Sigma^m(D))$ is a non-reduced symbol of maximum value $m+1$, level $1$ and rank $1$.  Since $S=(A,B)$ is reduced, 
the hook or cohook removed from $S$ to form $T= (\Sigma^l(C),\Sigma^l(D))$ must be $n-1$. 

Note that $n-1$ must not appear in at least one of the half-symbols 
$\Sigma^l(C)$, $\Sigma^l(D)$.
It follows that $n-1\ge l$ since $0,1,\ldots l-1$ all belong to both $\Sigma^l(C)$ and $\Sigma^l(D)$. If $l \leq n-2$, then $\max(T)=l+1 \leq n-1$ and, since 
$T$ is obtained from $S$ by exchanging $n-1$ to $0$, we have $\max(S)=n-1$. In this case, since $m(S) = m(T)=2l+1$, \eqref{lev-uni} implies 
that $j=\lev(S)=l+1$. If $l=n-1$, then $\max(T)=n=\max(S)$, and so $j=\lev(S)=l=n-1$. In both cases, 
$$2 \leq j=\lev(S) \leq n-1,$$ 
showing in particular that none of these $S$ can appear in (a). 
Moreover, if $S=(A,B)$ has more than one $(n-1)$-hook $n-1$ or more than one $(n-1)$-cohook $n-1$, then $n-1$ must appear in both $A$ and 
$B$, but $0$ cannot appear in any of them, which means that $T$ will be reduced, contrary to $l >0$. It follows from the Murnaghan-Nakayama rule that $\chi(t_{n-1}) = \pm 1$.

Assume now that $2 \leq j  \leq n-2$. The above analysis shows that $j=l+1$, and so $l=j-1$ is completely determined by $j$, and $1 \leq l \leq n-3$.
Note that there are at most two symbols of level $l+1$ and rank $n$ which give rise to
$T=(\Sigma^l(C),\Sigma^l(D))$ by removing an $(n-1)$-hook (which must be $n-1$), and at most two that give rise to $T=(\Sigma^l(C),\Sigma^l(D))$ by removing an $(n-1)$-cohook (again necessarily $n-1$). Fixing the defect of $S$, which implies that we can choose only two possibilities for $T$ from \eqref{rank 1 level 1} for a hook removal,
respectively for a cohook removal, we get at most $4$ possibilities for $S$ for each of these two removal operations.

Next assume that $j=n-1$. Then $l=n-2$ or $l=n-1$. Direct computation shows that, for each value of $l$, each $T=(\Sigma^l(C),\Sigma^l(D))$ gives rise to one $S$ by adding an 
$(n-1)$-hook, respectively by adding an $(n-1)$-cohook, and these two symbols have different defect modulo $4$. So again we get  at most $4$ possibilities for
$S$ of given defect modulo $4$ for each of these two operations.

As mentioned above, given the choice of $(n-1)$-cycle, the condition $\chi(t_{n-1})\neq 0$ determines whether $(A,B)$ is to have an $(n-1)$-hook or an $(n-1)$-cohook (and whether the element of \eqref{rank 1 level 1} equivalent to the symbol obtained by removing this hook or cohook itself has a $1$-hook or a $1$-cohook).  This gives at most four cases 
of $\chi$ with $\chi(t_{n-1})\neq 0$ and $\lev(\chi)=j$.
\end{proof}


\section{Level and degree}

Let $G$ be one of the finite classical groups $\Sp_{2n}(q)$, $\SO_{2n+1}(q)$, or $\GO^\varep_{2n}(q)$, of rank $n$. In this section we define the \emph{level} of
an irreducible character $\chi$ of $G$ in terms of its Lusztig label \cite{LusztigBook}, \cite{DM}, and relate it to  $\chi(1)$ when $\chi(1)$ is not large. For groups of 
type A this has been done in \cite{GLT1}.

\medskip
First we consider \emph{unipotent} characters.
The degree of the unipotent character $\chi$ of $G$ associated to the reduced symbol $(A,B)$ of rank $n$ is given by \cite[p.~359]{LusztigBook}.  Note that the order of the prime-to-$p$ part $|G|_{p'}$ of the relevant classical group $G$ is
within a bounded factor of $q^{2+4+\cdots+2n} = q^{n(n+1)}$ for symplectic and odd-dimensional orthogonal groups and of $q^{2+4+\cdots+2n-n} = q^{n^2}$
for even-dimensional orthogonal groups.
Moreover, the degree is within a factor of $2^{O(m)}$ of 
$$q^{-\sum_{a\in A} a(a+1) - \sum_{b\in B} b(b+1)- \sum_{i\ge 0} \binom{m-1-2i}2}|G|_{p'}\prod_{a_1>a_2 \in A} q^{a_1} \prod_{b_1>b_2\in B} q^{b_1} \prod_{a\in A}\prod_{b\in B} q^{\max(a,b)}.$$
As
$$\sum_{i\ge 0} \binom{m-1-2i}2 = \frac{2m^3-3m^2-2m+3\epsilon_m}{24},$$
we have
\begin{equation}
\label{log}
\begin{split}
\log_q \chi(1) = n^2&+(1-\epsilon_m) n - \sum_{a\in A} a(a+1) - \sum_{b\in B} b(b+1) - \frac{2m^3-3m^2-2m}{24} \\
&+ \sum_{a_1>a_2 \in A} a_1 + \sum_{b_1>b_2\in B} b_1 + \sum_{a\in A}\sum_{b\in B} \max(a,b) + O(|A|+|B|).
\end{split}
\end{equation}
Furthermore, if $|\pi(A)|$ denotes the sum of all parts of the partition $\pi(A)$ as before and $|A|=k$, then  
$\sum_{a \in A}a = |\pi(A)| + k(k-1)/2.$
Hence it follows from \eqref{rk10} that
\begin{equation}\label{rk11}
  \rank(A,B) = |\pi(A)|+|\pi(B)| + \frac {\defect(A,B)^2-(1-\epsilon_m)}4.
\end{equation}  
Next,
$$\sum_{a_1>a_2 \in A} a_1 + \sum_{b_1>b_2 \in B} b_1  + \sum_{a\in A}\sum_{b\in B} \max(a,b) = c_1+2c_2+\cdots+mc_m,$$
and the right hand side of \eqref{log} can be rewritten as
\begin{equation}
\label{c-form}
n^2+(1-\epsilon_m) n - \sum_i (c_i^2+c_i-ic_i) -\frac{2m^3-3m^2-2m}{24} + O(m).
\end{equation}

Setting $d_i = c_i - i/2$, the $d_i$ are non-negative with $d_i \ge 1/2$ for odd $i$, and 
\begin{equation}\label{d10}
  d_{i} \leq d_{i+2},~d_i \leq d_{i+1} + 1/2,~d_0+d_1+\cdots+d_i \ge \frac i4
\end{equation}  
for any $i \geq 0$. By 
\eqref{rk10} we have
\begin{equation}
\label{n}
n = (c_0+c_1+\cdots+c_m) - \frac{m^2}4 +\frac{\epsilon_m}{4}= (d_0+d_1+\cdots+d_m)+\frac m4 + \frac{\epsilon_m}{4}.
\end{equation}
Setting 
$$\Sigma=\bigl(\sum^m_{i=0}d_i\bigr)^2-\sum^m_{i=0}d_i^2 = 2\sum_{0 \leq i < j \leq m}d_id_j,~\Sigma'=\bigl(\sum^{m-1}_{i=0}d_i\bigr)^2-\sum^{m-1}_{i=0}d_i^2,$$
we also have
$$\Sigma' = 2\sum_{0\le i<j\le m-1} d_id_j \ge \frac{m^2-4m}{16}$$
(using only $d_id_j$ with odd $i < j$.) 
This implies
\begin{align*}
n^2 &\ge (d_0+\cdots+d_m)^2 + \frac{m^2}{16} + \frac{m}{2}(d_0+\cdots+d_m)\\
&\ge (d_0^2+\cdots+d_m^2)+ \Sigma'+2d_m(d_0+\cdots+d_{m-1})+\frac{3m^2}{16}\\
&\ge \sum^m_{i=0}d_i^2+ 2d_m(d_0+\cdots+d_{m-1})+\frac{m^2-m}{4}. 
\end{align*}
so
\begin{equation}
\label{m-cubed}
\begin{split}
n^2 - \sum_i (c_i^2&+c_i-ic_i)  - \frac{2m^3-3m^2-2m}{24} \\
&= n^2 - \sum_i c_i - \sum_i d_i^2  - \frac{2m^3-3m^2-2m}{24} + \frac{m(m+1)(2m+1)}{24} \\
&= n^2 - \sum_i c_i - \sum_i d_i^2  + \frac{2m^2+m}{8} \\
&\ge  -\sum_i c_i + 2d_m(d_0+\cdots+d_{m-1})  + \frac{4m^2-m}8 \\
&\ge -n + 2d_m(d_0+\cdots+d_{m-1})  + \frac{4m^2-m}8.
\end{split}
\end{equation}
We also have 
$$n^2 \ge (d_0+\cdots+d_m)^2 + \frac{m^2}{16} + \frac{m}{2}(d_0+\cdots+d_m)
\ge \sum^m_{i=0}d_i^2+\Sigma+\frac{3m^2}{16},$$ 
and so
\begin{equation}
\label{m20}
\begin{split}
n^2 - \sum_i (c_i^2&+c_i-ic_i)  - \frac{2m^3-3m^2-2m}{24} \\
&= n^2 - \sum_i c_i - \sum_i d_i^2  + \frac{2m^2+m}{8} \\
&\ge -n + \Sigma  + \frac{7m^2+2m}{16}.
\end{split}
\end{equation}

In what follows, $\nint(x)$ denotes the integer nearest to $x \in \R$.

\begin{thm}\label{lev-unip}
For any $1 \leq \alpha < 3/2$ and any $c>0$, there is some $N=N(\alpha,c)$ such that the following statement holds whenever
$n \geq N$.
Suppose the unipotent character  $\chi$ of the classical group $G \in \{\Sp_{2n}(q),\SO_{2n+1}(q),\GO^\pm_{2n}(q)\}$ 
is associated to a symbol $(A,B)$ and $\log_q \chi(1) \leq cn^{\alpha}$. 
Then
$$\log_q \chi(1) = (2\,\lev(A,B)+O(n^{2\alpha-3}))n,$$
and hence
$$\lev(A,B) = \nint\bigl( \frac{\log_q \chi(1)}{2n} \bigr).$$
\end{thm}

\begin{proof}
By \eqref{log}, \eqref{c-form}, and \eqref{m-cubed}, we have
\begin{equation}
\label{log-ineq}
\log_q \chi(1)+\epsilon_mn \geq 2d_m(d_0+\cdots+d_{m-1}) + \frac{m^2}2+ O(m).
\end{equation}
Since the left hand side of \eqref{log-ineq}  is $O(n^{\alpha})$, we have $m = O(n^{\alpha/2}) = o(n)$.  

We now prove the stronger bound
\begin{equation}\label{m21}
  m = O(n^{\alpha-1}).
\end{equation}
Certainly, we may assume $m \geq 4$,  
Using \eqref{m20} instead of \eqref{m-cubed}, we have the following variant of \eqref{log-ineq}:
\begin{equation}\label{log10}
  \log_q \chi(1) + \epsilon_mn \geq  \Sigma + \frac{7m^2}{16}+ O(m).
\end{equation}  
Choosing 
$$j =\lfloor m/4 \rfloor \geq 1,$$ 
by \eqref{d10} we have  
$$d_0+d_1 + \cdots + d_{2j-1} \leq d_{2j}+d_{2j+1} + \cdots + d_{4j-1},$$
and hence 
$$d_{2j}+d_{2j+1} + \cdots +d_{m} \geq \sum^m_{i=0}\frac{d_i}2 > \frac{n}2-\frac{m+1}8 > \frac{n}3$$ 
when $n$ is large enough. This implies 
$$\Sigma \geq 2\bigl(d_0+d_1 + \cdots + d_{2j-1}\bigr)\bigl( d_{2j}+d_{2j+1} + \cdots +d_{m}\bigr) \geq \frac{2n}{3}\bigl(d_0+d_1 + \cdots + d_{2j-1}\bigr).$$ 
It follows from \eqref{log10} that 
$$d_0+d_1 + \cdots + d_{2j-1} = O(n^{\alpha-1}),$$
and so $j = O(n^{\alpha-1})$ by \eqref{d10}. The choice of $j$ now implies \eqref{m21}. 

Since $d_i \leq d_m+1/2$ by \eqref{d10}, \eqref{n} shows 
that 
$$d_m+\frac 12\ge \frac{2(n-(m+1)/4)}{m+1}.$$ 
Using \eqref{m21}, we get 
$d_m \gg n^{2-\alpha}$.
Also by \eqref{n} and \eqref{m21},
\begin{equation}
\label{almost n}
d_0+\cdots+d_{m-1}+d_{m} \ge n-O(n^{\alpha-1}).
\end{equation}
However, since $\log_q \chi(1) = O(n^{\alpha})$ and $d_m \gg n^{2-\alpha}$, \eqref{log-ineq} implies
%
$$d_0+\cdots+d_{m-1}= O(n^{2\alpha-2}).$$
By \eqref{almost n} and since $1 \leq \alpha < 3/2$, $d_m = n - O(n^{2\alpha-2}) > n/2$ if $n$ is sufficiently large, which, by \eqref{log-ineq},  implies the improved bound
\begin{equation}\label{d11}
  d_0+\cdots+d_{m-1}= O(n^{\alpha-1}), \mbox{ whence }d^2_0+\cdots+d^2_{m-1}= O(n^{2\alpha-2}).
\end{equation}  
%
Together with \eqref{m21}, this shows that
$$d_m^2-\bigl(n^2-2(n-d_m)n\bigr) = (n-d_m)^2 = O(n^{2\alpha-2}).$$ 
Using the equality in \eqref{m-cubed}, we now obtain
\begin{align*}
\log_q \chi(1) &= n^2 + (1-\epsilon_m) n - \sum_i c_i - \sum_i d_i^2  + \frac{m^2}4 + O(m) \\
&= n^2 + (1-\epsilon_m) n - n - (n^2 - 2(n-d_m)n) + O(n^{2\alpha-2}) \\
&= n(-\epsilon_m+2(n-d_m) + O(n^{2\alpha-3})).
\end{align*}
As 
$$2(n-d_m)  - \epsilon_m =  2\bigl(\rank(A,B)-\max(A,B) + \frac{m-\epsilon_m}2\bigr) = 2\,\lev(A,B),$$
the statement follows.
\end{proof}

In the general case, write $G=\Cl(V)$,
where $V$ is defined over a field $\F_q$ in characteristic $p$ and $\Cl=\SO$ or $\Sp$. Then $G$ is dual to a finite classical group 
$G^*= \Cl^*(V^\sharp)$ and $V^\sharp$ is also defined over $\F_q$. By Lusztig's classification \cite{LusztigBook}, any $\chi \in \Irr(G)$ is labeled by $(s,\psi)$, where 
$s \in G^*$ is semisimple, and $\psi$ is a unipotent character of $\CB_{G^*}(s)$. 
We will identify $G^*$  with $G$ for $G = \Sp_{2n}(q)$ (and view $\SO_{2n+1}(q)$ as $\Sp_{2n}(q)$)
when $2|q$. 

We have the orthogonal decomposition
\begin{equation}
\label{decomp}
V^\sharp = V^\sharp_1 \oplus V^\sharp_{-1} \oplus V^\sharp_0
\end{equation}
into $s$-invariant subspaces, 
where $s$ acts as the scalar $\kappa$ on $V^\sharp_\kappa$ when $\kappa = \pm 1$ (with the convention that $-1=1$ when $2|q$), 
and $s$ has no eigenvalue $1$ or $-1$ on $V^\sharp_0$. 
Define $H^* = \GO(V^\sharp)$ if $G^*$ is of type $B$ or $D$, and $H^*=G^*$ otherwise.
Let $H^*_\kappa \leq H^*$ consist of the elements of $H^*$ that act trivially on the orthogonal complement of $V^\sharp_\kappa$ in $V^\sharp$  
for $\kappa \in \{1,-1,0\}$. Then we have 
\begin{equation}\label{s10}
  \CB_{H^*}(s) = H^*_1 \times H^*_{-1} \times \CB_{H^*_0}(s),
\end{equation}  
and $[\CB_{H^*}(s):\CB_{G^*}(s)] \leq 2$.

We can also consider the corresponding (connected) simple algebraic groups $\cG=\Cl(\tilde V)$ and $\cG^* = \Cl^*(\tilde V^\sharp)$, where 
$\tilde V = V \otimes_{\F_q}\overline{\F_q}$ and $\tilde V^\sharp = V^\sharp \otimes_{\F_q}\overline{\F_q}$, and have a similar to \eqref{s10} 
decomposition
\begin{equation}\label{s11}
  \CB_{\cG^*}(s)^\circ = \cG^*_1 \times \cG^*_{-1} \times \CB_{\cG^*_0}(s).
\end{equation}  

Let $\tilde\psi$ be a unipotent character of $\CB_{H^*}(s)$ lying above $\psi$ (so $\tilde\psi(1)/\psi(1) = 1$ or $2$), and let 
$(A_\kappa,B_\kappa)$ denote the symbol of the $H^*_\kappa$-component $\psi_\kappa$ of $\tilde\psi$ for $\kappa=\pm1$, again 
with the convention that $-1=1$ when $2|q$. Then we define
\begin{equation}\label{lev-gen} 
  \lev(\chi) = \frac{\dim V^\sharp}{2}  - \max_{\kappa = \pm 1}\biggl(\max(A_\kappa,B_\kappa) -\lfloor \frac{m_\kappa}{2}\rfloor +\frac{\delta_\kappa}2 \biggl),
\end{equation}
where $\delta_\kappa \in \{0,1\}$ and $\delta_\kappa=1$ precisely when $2 \nmid q$ and $2\nmid \dim(V^\sharp_\kappa)$. Note that this agrees with \eqref{lev-uni} when $\chi$ is unipotent:
say when $G = \Sp(V) \cong \Sp_{2n}(q)$ with $2 \nmid q$ then $G^* = \SO(V^\sharp) \cong \SO_{2n+1}(q)$, $\delta_\kappa=1$ for $\kappa=1$ and this compensates for
the difference $1/2$ between $(\dim V)/2$ and $(\dim V^\sharp)/2$. Also, \eqref{lev-gen} shows that the irreducible Weil characters of $\Sp_{2n}(q)$
when $q$ is odd have level $1/2$. 

\smallskip
Now we can extend Theorem \ref{lev-unip} to non-unipotent characters of finite classical groups:

\begin{thm}\label{lev-deg}
For any $1 \leq \alpha < 3/2$ and any $c>0$, there is some $N=N(\alpha,c)$ such that the following statement holds whenever
$n \geq N$ and $q$ is any prime power.
For any irreducible character $\chi$ of $G \in \{\Sp_{2n}(q),\SO_{2n+1}(q),\SO^\pm_{2n}(q)\}$ satisfying $\log_q \chi(1) \leq cn^{\alpha}$, we have
$$\log_q \chi(1) = 2\,\lev(\chi)n + O(n^{2\alpha-2}),$$
and hence
$$\lev(\chi) = \nint\bigl( \frac{\log_q \chi(1)}{2n} \bigr).$$
\end{thm}

\begin{proof}
We will give the proof in the case $q$ is odd. The even $q$ case is almost entirely similar, the only differences being that 
we identify $G^*$ with $G$, have $-1=1$, and do not need to consider $G=\SO_{2n+1}(q)$ (in particular, $\dim V^\sharp = \dim V$ is always even
and hence $\delta_\kappa=0$).
 
Let $\chi$ be labeled by $(s,\psi)$ as above and consider the decompositions \eqref{s10} and \eqref{s11}. Then $\CB_{\cG^*_0}(s)$ is a direct product
of $f \geq 0$ factors $\GL_{a_i}$, $1 \leq i \leq f$, and clearly
$$f \leq n.$$
When $\cG^*$ is of type $B$, $\cG^*_1$ is of type $B$ and $\cG^*_{-1}$ is of type $D$.
When $\cG^*$ is of type $C$, both $\cG^*_1$ and $\cG^*_{-1}$ are of type $C$.
When $\cG^*$ is of type $D$, both $\cG^*_1$ and $\cG^*_{-1}$ are of type $D$.
We denote by  $b_1$ the rank of the factor of type $B$, by $c_1$ and $c_2$ the ranks of the  factors of type $C$, and by $d_1$ and $d_2$ the ranks of the factors of type $D$.  Some of these, of course, may be zero, and we also have
\begin{equation}
\label{t def}
  n = \sum_i a_i + \sum_i b_i + \sum_i c_i + \sum_i d_i.
\end{equation}

Up to multiplicative errors of $2^{O(f)}$, the prime-to-$p$ part $|\CB_{G^*}(s)|_{p'}$ is bounded from the above by 
$$q^{\sum_i \frac{a_i^2+a_i}2 + \sum_i (b_i^2+b_i) + \sum_i (c_i^2+c_i) + \sum_i (d_i^2-d_i)}$$
and from the below by 
$$q^{\sum_i a_i + \sum_i (b_i^2+b_i) + \sum_i (c_i^2+c_i) + \sum_i (d_i^2-d_i)}.$$
On the other hand,
the prime-to-$p$ part of $|G|$ is given up to a bounded multiplicative factor by $q^{n^2+n}$ if $\cG$ is of type $B$ or $C$, and $q^{n^2}$ if it is of type $D$.
By \cite{LusztigBook}, 
$$\chi(1)=\frac{\psi(1) |G|_{p'}}{|\CB_{G^*}(s)|_{p'}},$$
so if $\log_q \chi(1) \le cn^\alpha$, then
\begin{equation}\label{s20}
  n^2 - \sum_i \frac{a_i^2}2- \sum_i b_i^2 - \sum_i c_i^2 - \sum_i d_i^2 \le cn^\alpha + O(f).
\end{equation}
By \eqref{t def}, the left-hand side of \eqref{s20} is
$$\sum_i a_i(n-a_i/2) + \sum_i b_i(n-b_i) + \sum_i c_i(n-c_i) + \sum_i d_i(n-d_i).$$
Clearly, $n-a_i/2 \geq n/2 > cn^{\alpha-1}+O(1)$ for all $i$. If all the $n-b_i$, $n-c_i$, and $n-d_i$ are more than $cn^{\alpha-1}+O(1)$, then
the above sum is larger than
$$\bigl(cn^{\alpha-1}+O(1)\bigr)\bigl(\sum_i a_i+\sum_i b_i +\sum_i c_i + \sum_i d_i\bigr) = cn^{\alpha} + O(n),$$
contrary to \eqref{s20}. So at least one of $b_i$, $c_i$, or $d_i$, call it $e$, is at least 
$$n-cn^{\alpha-1}+ O(1),$$ 
and hence 
$$f \le cn^{\alpha-1}+ O(1)=O(n^{\alpha-1})$$ 
by \eqref{t def}. In fact,
$$\sum_i a_i \leq \sum_i (a_i^2+a_i)/2 \leq \bigl(\sum_i a_i\bigr)^2 \leq (n-e)^2 = O(n^{2\alpha-2}).$$
Similarly, if $b$ is any of the $b_j$, $c_j$, or $d_j$ different from $e$, then
$$b^2-b \leq b^2+b \leq 2(n-e)^2 = O(n^{2\alpha-2}).$$
It follows that
$\log_q |\CB_{G^*}(s)|_{p'}$ must be $e^2+e+O(n^{2\alpha-2})$ if $e$ is a $b_i$ or a $c_i$ (in which case $\cG$ is of type $B$ or $C$) and 
$e^2+O(n^{2\alpha-2})$ if $e$ is a $d_i$.
Therefore, 
\begin{equation}
\label{log ratio}
\log_q \frac{|G|_{p'}}{|\CB_{G^*}(s)|_{p'}} = n^2 - e^2 + O(n^{2\alpha-2})
\end{equation}
except in the case that $G^*$ is of type $B$ and $e$ comes from a factor of type $D$, in which case it is $n^2-e^2+n + O(n^{2\alpha-2})$.

If $n$ is sufficiently large, we conclude that $s$ has an eigenvalue $\kappa \in \{\pm 1\}$ with multiplicity $2e+1$ (in which case $G^*$ is of type $B$) or $2e$, 
and the right-hand side of \eqref{log ratio} is 
$$(2(n-e)+O(n^{2\alpha-3}))n,$$ 
unless $G^*$ is of type $B$ and $\kappa = -1$, in which case it is $(2(n-e)+1+O(n^{2\alpha-3}))n$.  In any event,
\begin{equation}
\label{sharp non kappa}
\log_q \frac{|G|_{p'}}{|\CB_{G^*}(s)|_{p'}}  = (\dim V^\sharp-\dim V^\sharp_\kappa)n + O(n^{2\alpha-2}).
\end{equation}
We can express $\tilde\psi$ as a character $\psi_\kappa\boxtimes \psi^\kappa$ of $\CB_{H^*}(s)$ which is a direct product of $H^*_\kappa$ and a complementary factor, and since the codimension of $V^\sharp_\kappa$ in $V^\sharp$ is $O(n^{\alpha-1})$, we have 
$\psi^\kappa(1)=q^{O(n^{2\alpha-2})}$.  Therefore,
$$\log_q \chi(1) = \log_q \psi_\kappa(1) + \log_q \frac{|G|_{p'}}{|\CB_{G^*}(s)|_{p'}}  + O(n^{2\alpha-2}).$$
By Theorem~\ref{lev-unip} and equations \eqref{lev-uni} and \eqref{sharp non kappa},
\begin{equation}\label{s21}
\begin{aligned}
\frac{\log_q \chi(1)}{2n} &= \rank(A_\kappa,B_\kappa) - \max(A_\kappa, B_\kappa) + \bigl\lfloor\frac{m_\kappa}2\bigr\rfloor + \frac{\dim V^\sharp-\dim V^\sharp_\kappa}2  + O(n^{2\alpha-3}) \\
& = \frac{\dim V^\sharp}{2} - \max(A_\kappa,B_\kappa) +\lfloor \frac{m_\kappa}{2}\rfloor -\frac{\delta_\kappa}2 + O(n^{2\alpha-3}).
\end{aligned}
\end{equation}
since $\rank(A_\kappa,B_\kappa)=(\dim V^\sharp_\kappa-\delta_\kappa)/2$. This implies that
$$\max(A_\kappa,B_\kappa) -\lfloor \frac{m_\kappa}{2}\rfloor +\frac{\delta_\kappa}2 = \frac{\dim V^\sharp}{2} - \frac{\log_q \chi(1)}{2n}- O(n^{2\alpha-3}) 
     = n-O(n^{\alpha-1}).$$
On the other hand, by \eqref{max vs rank} we have
$$\max(A_{-\kappa},B_{-\kappa}) -\lfloor \frac{m_{-\kappa}}{2}\rfloor +\frac{\delta_{-\kappa}}2 \leq \frac{\dim V^\sharp_{-\kappa}}{2}=O(n^{\alpha-1}).$$
Hence, when $n$ is sufficiently large we have 
$$\max(A_\kappa,B_\kappa) -\lfloor \frac{m_\kappa}{2}\rfloor +\frac{\delta_\kappa}2 >
\max(A_{-\kappa},B_{-\kappa}) -\lfloor \frac{m_{-\kappa}}{2}\rfloor +\frac{\delta_{-\kappa}}2,$$
and so
$$\lev(\chi) =  \frac{\dim V^\sharp}{2} - \max(A_\kappa,B_\kappa) +\lfloor \frac{m_\kappa}{2}\rfloor -\frac{\delta_\kappa}2$$
by \eqref{lev-gen}. The statement now follows from \eqref{s21}.
\end{proof}

\section{The Howe correspondence}
Throughout this section, we assume that $q$ is an {\it odd} prime power.
Let $G$ be one of $\GO_{2n}^{\varepsilon}(q)$, $\SO_{2n+1}(q)$, or $\Sp_{2n}(q)$, and let $G^\circ$ denote the kernel of the determinant map
$G \to \{\pm 1\}$.
Let $G'$ be one of $\Sp_{2n'}(q)$, $\GO_{2n'}^{\varepsilon}(q)$, or $\SO_{2n'+1}(q)$.  If one of $G$ and $G'$ is an orthogonal group in dimension $m$ and the other is symplectic in dimension $2m'$,
then the central product $G\ast G'$ embeds in the group $\Sp_{2mm'}(q)$.
In this case, we fix a non-trivial additive character of $\F_q$, and this determines the (reducible) Weil representation $\omega$ of $\Sp_{2mm'}(q)$, of degree $q^{mm'}$. A pair $(\chi,\chi')$ of characters $\chi \in \Irr(G)$ and $\chi' \in \Irr(G')$ such that $\chi \boxtimes \chi'$ occurs in the restriction 
$\omega|_{G \ast G'}$ are said to be in the {\it Howe correspondence} (between $G$ and $G'$).


We consider first the case  $G = \GO^\varepsilon_{2n}(q)$.
The dual group $G^*$ is then $\GO^\varepsilon_{2n}(q)$, and the natural representation $V^\sharp$ has dimension $2n$.
Let $s\in (G^\circ)^* = \SO^\varepsilon_{2n}(q)$ be a semisimple element.  We use the notation \eqref{decomp} for the decomposition of $V^\sharp$ under $s$
and write $s=\diag(s_1,s_{-1},s_0)$.  The centralizer $\CB_{G^*}(s)$ can be expressed as $G_1^*\times G_{-1}^*\times G_0^*$,
where $G_1^* = \GO(V_1^\sharp)$, $G_{-1}^* = \GO(V_{-1}^\sharp)$, and $G_0^* = \CB_{\GO(V^\sharp_0)}(s_0)$.
The factor $G_0^*$ is a product of groups of type $\GL$ and $\GU$, which
is naturally embedded in $\GL(V^\sharp_0)$.

If $G = \SO_{2n+1}(q)$, the
dual group $G^*$ is  $\Sp_{2n}(q)$, and the natural representation $V^\sharp$ has dimension $2n$.
Let $s=\diag(s_1,s_{-1},s_0)\in G^* = \Sp_{2n}(q)$ be a semisimple element.  
The centralizer $\CB_{G^*}(s)$ can be expressed as $G_1^*\times G_{-1}^*\times G_0^*$,
where $G_1^* = \Sp(V_1^\sharp)$, $G_{-1}^* = \Sp(V_{-1}^\sharp)$, and $G_0^* = \CB_{\Sp(V^\sharp_0)}(s_0)$.
The factor $G_0^*$ is a product of groups of type $\GL$ and $\GU$, which
is naturally embedded in $\GL(V^\sharp_0)$.

If $G$ is $\Sp_{2n}(q)$, and $s=\diag(s_1,s_{-1},s_0) \in G^* = \SO_{2n+1}(q)$, then $\CB_{G^*}(s)$ is the determinant $1$ subgroup of 
$\GO(V_1^\sharp)\times \GO(V_{-1}^\sharp)\times\CB_{\editnew{\GO}(V^\sharp_0)}(s_0)$.
Since $\dim(V_1^\sharp)$ is odd, $G_1^*\cong \{\pm I\}\times \SO(V_1^\sharp)$; for any $g_1 \in G_1^*$ let $h_1$ denote its 
$\SO(V_1^\sharp)$ component. Then $\CB_{G^*}(s)$ is canonically isomorphic to 
$G_1^*\times G_{-1}^*\times G_0^*$ via the map $\diag(g_1,g_{-1},g_0) \mapsto \diag(h_1,g_{-1},g_0)$, 
where $G_1^* = \SO(V_1^\sharp)$, $G_{-1}^* = \GO(V_{-1}^\sharp)$, and $G_0^* = \CB_{\GO(V^\sharp_0)}(s_0)< \GL(V_0^\sharp)$.

A {\it Lusztig correspondence} associates any $\chi \in \Irr(G)$ to a pair $(s^{G^*},\psi)$ consisting of the $G^*$-conjugacy class of a semisimple element
$s \in G^*$ and a unipotent character $\psi$ of $\CB_{G^*}(s)$. The existence of such a correspondence is the fundamental result of Lusztig \cite{LusztigBook,DM};
however it need not be unique. But, as shown by Pan \cite{Pan-T}, we can, and will, choose one that is compatible with the Howe correspondence for the relevant 
dual pair $(G,G')$, see \cite[p. 426]{Pan-T}.  By means of this choice,
every element $\chi$ in the  Lusztig series $\cE(G,s)$ is identified with a unique unipotent character $\psi = \psi_1\boxtimes \psi_{-1}\boxtimes \psi_0$ of $\CB_{G^*}(s)$.
The character $\psi$ is determined by the unipotent character $\psi_0$ of $\CB_{\GO(V^\sharp_0)}(s_0)$, together with the reduced symbols $(A_1,B_1)$ and $(A_{-1},B_{-1})$ 
of the unipotent characters $\psi_1$ and $\psi_{-1}$, corresponding to $V^\sharp_1$ and $V^\sharp_{-1}$, respectively. Then the {\it level} $\lev(\chi)$ is defined as 
in \eqref{lev-gen}. 
Note that the ambiguity in matching the two symbols $(A_\kappa,B_\kappa)$ and $(B_\kappa,A_\kappa)$ with the two unipotent 
characters of $\GO(V^\sharp_\kappa)$ (which lie above a unique unipotent character of $\SO(V^\sharp_\kappa)$) when this group is even-dimensional (and $A_\kappa \neq B_\kappa$) 
does not affect the $\kappa$-summand in \eqref{lev-gen}, and hence the definition of the level of $\chi$. The corresponding labels of unipotent characters
of $\GO$ and $\SO$ also show that,  in the case of
$G = \GO^\varep_{2n}(q)$, $\chi \in \Irr(G)$ has the same level as the irreducible constituents of its restriction
to $\SO^\varep_{2n}(q)$.

Likewise a semisimple element $s'\in (G')^*$ has centralizer $(G')^*_1\times (G')^*_{-1}\times (G')^*_0$.
Any element $\chi'$ in the Lusztig series $\cE(G',s')$ corresponds to $\psi' = \psi'_1\boxtimes \psi'_{-1}\boxtimes \psi'_0$.
We denote the symbols associated with $\psi'_1$ and $\psi'_{-1}$ by $(A'_1,B'_1)$ and $(A'_{-1},B'_{-1})$ respectively.

If $G$ is an even orthogonal group and $G'$ is symplectic, \cite[Theorem~9.10]{Pan-T} asserts that $(\chi,\chi')$ appears in the Howe correspondence between 
$G$ and $G'$ if and only if all of the following conditions hold:
\begin{equation}
\label{Howe-even}
\begin{split}
&\bullet\text{The symbol $(A_1,B_1)$ is linked to $(A'_1,B'_1)$, in the sense of \eqref{split conditions} and \eqref{non-split conditions}.}\\
&\bullet\text{$(A_{-1},B_{-1})=(A'_{-1},B'_{-1})$.  In particular, $G^*_{-1}\cong (G')^*_{-1}$.}\\
&\bullet\text{There is an isomorphism of $V^\sharp_0$ and $(V')^\sharp_0$ which identifies $(s_0,\psi_0)$ and 
$(s'_0,\psi'_0)$.}
\end{split}
\end{equation}

If $G$ is an odd orthogonal group and $G'$ is symplectic,  \cite[Theorem~9.8]{Pan-T} asserts that $(\chi,\chi')$ appears in the Howe correspondence between 
$G$ and $G'$ if and only if all of the following conditions hold:
\begin{equation}
\label{Howe-odd}
\begin{split}
&\bullet\text{The symbol $(A'_{-1},B'_{-1})$ is linked to $(A_1,B_1)$, in the sense of \eqref{split conditions} and \eqref{non-split conditions}.}\\
&\bullet\text{$(A_{-1},B_{-1})=(A'_1,B'_1)$.}  \\
&\bullet\text{There is an isomorphism of $V^\sharp_0$ and $(V')^\sharp_0$ which identifies $(-s_0,\psi_0)$ and 
$(s'_0,\psi'_0)$.}
\end{split}
\end{equation}

\begin{prop}
\label{main-piece}
Let $S=(A,B)$ be a symbol of even defect such that 
$$\rank(S) > 4\,\lev(S)^2 + 2\,\lev(S).$$ 
Replacing $S$ by its transpose if necessary,
there exists a symbol $S'$ which is linked to $S$ and which has all of the following properties;
\begin{enumerate}[\rm(i)]
\item $\defect(S')\equiv 1\pmod 4$.
\item $\rank(S') = \lev(S)$.
\item If $T$ is a symbol linked to $S'$ such that $\rank(T)=\rank(S)$, $\defect(T)\equiv \defect(S)\pmod 4$, and $\lev(T)\ge \lev(S)$, then $T=S$.
\end{enumerate}
\end{prop}

\begin{proof}
By Lemma~\ref{no tie}, $c_m=\max(A,B)$ occurs in only one of $A$ and $B$, with $m=|A|+|B|-1$ as defined in \eqref{m10}. 
Replacing $S$ by its transpose if necessary, we may assume $c_m\in A$ if $\defect(S)$ is divisible by $4$ and $c_m\in B$ otherwise.  In the former case we define 
$$S' = (B,\bar A),$$ 
with $\bar A$ as defined in \eqref{Abar}, 
so $S$ is linked to $S'$ by \eqref{split conditions}, and in the latter case 
$$S'=(\bar B,A)$$ 
again using \eqref{Abar}, 
so $S$ is linked to $S'$ by \eqref{non-split conditions}. Statement (i) follows immediately.

\smallskip
For (ii), note that $2 \nmid m$, whereas 
the parameter $m$ for $S'$ is $m'=m-1$. So by \eqref{rk10} and \eqref{lev-uni} we have  
\begin{align*}
\rank(S') & =  \sum_{a \in A}a+\sum_{b \in B}b-c_m - \frac{(m-1)^2}{4}\\
& = \rank(S)+ \frac{m^2-1}{4}-c_m - \frac{(m-1)^2}{4}\\
& = n-c_m + \frac{m-1}{2}\\
& = \lev(A,B). 
\end{align*}

\smallskip
To prove (iii), assume $T = (C,D)$ is linked to $S'$.

\smallskip
(a) Suppose first that $\defect(S)$ is divisible by $4$.
We have $\pi(D)\preceq \pi(B)$, and $\pi(\bar A)\preceq \pi(C)$ by \eqref{split conditions}.
By hypothesis, $\rank(T)=\rank(S)$ and $\lev(T) \ge \lev(S)$; furthermore,
$$|C|-|D|=\defect(T) = 1 -\defect(S') = \defect(S) = |A|-|B|.$$ 
Let $m=|A|+|B|-1$ as before, and let $m_T = |C|+|D|-1$. Then 
$$m_T-m= |C|-|A|+|D|-|B| = 2(|C|-|A|),$$
in particular, $\epsilon_m=\epsilon_{m_T}$. 
It follows from \eqref{rk11} that 
\begin{equation}\label{pi10}
  |\pi(A)|+|\pi(B)|=|\pi(C)|+|\pi(D)|.  
\end{equation}  
Using \eqref{lev-uni} we now get
$$\frac{m_T}2 - \max(C,D) \geq \frac m2 - \max(A,B),$$
whence 
$$|C|-|A|=\frac{m_T-m}2 \geq \max(C,D)-\max(A,B).$$
As $\max(A,B)=c_m = \max(A)$, it follows that
$$\max(A)-|A| \geq \max(C,D)-|C| \geq \max(C)-|C|,$$
and so the largest part of $\pi(A)$ is at least the largest part of $\pi(C)$.
As $\pi(\bar A)\preceq \pi(C)$, we also have $\pi(\bar C)\preceq \pi(\bar A)$. Hence we have shown that
$\pi(C) \preceq \pi(A)$, whence
$$|\pi(C)| \leq |\pi(A)|$$
with equality only if $\pi(C)=\pi(A)$. 
Since we also have $|\pi(D)| \leq |\pi(B)|$ with equality only if $\pi(D)=\pi(B)$, \eqref{pi10} now forces equality signs in all the above inequalities, 
i.e.  
$$\pi(C)=\pi(A),~\pi(D)=\pi(B),~|C|=|A|,~|D|=|B|,$$
and hence $A=C$, $B=D$. Thus $S=T$ as stated.

\smallskip
(b) Now we consider the case that $\defect(S)\equiv 2\pmod 4$.  Thus, $\pi(\bar B)\preceq \pi(D)$ and $\pi(C)\preceq \pi(A)$ by
\eqref{non-split conditions}. We have 
$$|C|-|D|=\defect(T) = -1-\defect(S') = \defect(S)= |A|-|B|;$$
furthermore, $\rank(T)=\rank(S)$ and $\lev(T)\ge \lev(S)$ by hypothesis.
Let $m=|A|+|B|-1$ and $m_T = |C|+|D|-1$ as above. Then 
$$m_T-m= |C|-|A|+|D|-|B| = 2(|D|-|B|),$$
in particular, $\epsilon_m=\epsilon_{m_T}$. 
It follows from \eqref{rk11} that \eqref{pi10} holds in this case as well.
Using \eqref{lev-uni} we now get
$$\frac{m_T}2 - \max(C,D) \geq \frac m2 - \max(A,B),$$
whence 
$$|D|-|B|=\frac{m_T-m}2 \geq \max(C,D)-\max(A,B).$$
As $\max(A,B)=c_m = \max(B)$, it follows that
$$\max(B)-|B| \geq \max(C,D)-|D| \geq \max(D)-|D|,$$
and so the largest part of $\pi(B)$ is at least the largest part of $\pi(D)$.
As $\pi(\bar B)\preceq \pi(D)$, we also have $\pi(\bar D)\preceq \pi(\bar B)$. Hence we have shown that
$\pi(D) \preceq \pi(B)$, and so 
$$|\pi(D)| \leq |\pi(B)|$$
with equality only if $\pi(D)=\pi(B)$. 
Since we also have $|\pi(C)| \leq |\pi(A)|$ with equality only if $\pi(C)=\pi(A)$, \eqref{pi10} now forces equality signs in all the above inequalities, 
i.e.  
$$\pi(D)=\pi(B),~\pi(C)=\pi(A),~|D|=|B|,~|C|=|A|,$$
and hence $D=B$, $A=C$. Thus $S=T$ as stated.
\end{proof}

\begin{prop}
\label{main-piece-odd}
Let $S=(A,B)$ be a symbol such that $\defect(S) \equiv 1 \pmod 4$ and
$$\rank(S) > 4\,\lev(S)^2 + 2\,\lev(S).$$ 
Then there exists a symbol $S'$ which is linked to $S$ and such that
\begin{enumerate}[\rm(i)]
\item $\rank(S') = \lev(S)$.
\item If $T$ is a symbol linked to $S'$ such that $\rank(T)= \rank(S)$, $\defect(T)\equiv \defect(S) \pmod4$, and $\lev(T)\ge \lev(S)$,
then $T=S$.
\end{enumerate}
\end{prop}

\begin{proof}
If $c_m\in A$, let $S' = (B,\bar A)$, and if $c_m\in B$, let $S' = (\bar B,A)$.  
For (i), note that $2 \vert m$, whereas 
the parameter $m$ for $S'$ is $m'=m-1$. So by \eqref{rk10} and \eqref{lev-uni} we have  
\begin{align*}
\rank(S') & =  \sum_{a \in A}a+\sum_{b \in B}b-c_m - \frac{m(m-2)}{4}\\
& = \rank(S)+ \frac{m^2}{4}-c_m - \frac{m(m-2)}{4}\\
& = n-c_m + \frac m2\\
& = \lev(A,B). 
\end{align*}

\smallskip
For (ii), let $T=(C,D)$.
We consider first the case that $4$ divides $\defect(S')$, which is equivalent to the condition $c_m\in A$ (using Lemma \ref{no tie}).
As $\defect(T)+\defect(S') \equiv \defect(S)+\defect(S') \equiv 1 \pmod 4$,
it follows that condition \eqref{split conditions} must hold 
for $S'$ and $T$.  
As $S'$ is linked to $T$,
$$\pi(\bar A) \preceq \pi(C),\ \pi(D)\preceq \pi(B).$$
Applying the argument of (a) of Proposition~\ref{main-piece}, we conclude $S=T$.

If $4$ does not divide $\defect(S')$, then $c_m\in B$ and $\defect(S)+\defect(S')\equiv -1\pmod4$. Thus,  
the condition \eqref{non-split conditions} must hold 
for $S'$ and $T$. 
As $S'$ is linked to $T$,
$$\pi(\bar B)\preceq \pi(D),\ \pi(C)\preceq \pi(A).$$
Applying the argument of (b) of Proposition~\ref{main-piece}, we conclude $S=T$.
\end{proof}

\begin{thm}\label{d-head}
Let $q$ be an odd prime power, $G = \GO^\varepsilon_{2n}(q)$, $G' = \Sp_{2n'}(q)$, and $\chi$ be an irreducible character of $G$ of level $n' \geq 1$, with
$n>4(n')^2+3n'$. Then there exists an irreducible character $\chi'$ of $G'$ and a linear character $\sigma$ of $G$ such that $(\chi\otimes \sigma,\chi')$ appears in the
Howe correspondence for $(G,G')$.  Moreover, if $\theta$ is an irreducible character of $G$ such that $(\theta,\chi')$ also appears in the Howe correspondence for 
$(G,G')$, then $\lev(\theta) \le \lev(\chi)$ with equality only if $\theta = \chi\otimes \sigma$.
\end{thm}

\begin{proof}
Using \eqref{max vs rank} and $n \geq 7n'+1$, we see that the maximum in \eqref{lev-gen} is achieved for exactly one value of $\kappa \in \{\pm 1\}$. 
Note that the spinor norm \cite[(22.11)]{Asch} gives rise to a linear character $\varsigma$ of order $2$ of $G$ which is trivial on $\Omega^\varepsilon_{2n}(q)$, and tensoring $\chi$ with $\varsigma$ has the effect of interchanging the unordered pairs of half-symbols $\{A_1,B_1\}$ and $\{A_{-1},B_{-1}\}$, see \cite[Cor. 8.10]{Pan-T}.
Hence, there is a unique $k \in \{0,1\}$ such that, after replacing $\chi$ by $\chi \otimes \varsigma^k$, we have that the maximum in \eqref{lev-gen} is achieved for $\kappa=1$.  
It follows that
\begin{equation}
\label{n'}
\lev(\chi)=n' = \lev(A_1,B_1) + \frac{\dim V^\sharp_0+\dim V^\sharp_{-1}}2.
\end{equation}
In particular, 
\begin{equation}\label{for-s20}
  \dim V^\sharp_1 \ge  \dim V^\sharp - 2n' = 2n-2n', 
\end{equation}   
whence
\begin{equation}
\label{big 1 rank}
\rank(A_1,B_1) = \frac{\dim V^\sharp_1}{2} > 4\,\lev(A_1,B_1)^2 + 2\,\lev(A_1,B_1).
\end{equation}
Tensoring $\chi$ by $\sgn$, the unique linear character of $G$ of order $2$, has the effect of transposing both $(A_1,B_1)$ and $(A_{-1},B_{-1})$ in the chosen Lusztig correspondence, see
\cite[Cor. 8.9]{Pan-T}. So setting $m = m(A_1,B_1)$ as in \eqref{m10} 
and $\varepsilon_1\in \{\pm\}$  the type of $\GO(V^\sharp_1)$ as an even-dimensional orthogonal group, and using Lemma \ref{no tie},
there is a unique $l \in \{0,1\}$ such that, after replacing $\chi$ by $\chi \otimes \sgn^l$ we have 
$c_m\in A_1\setminus B_1$ if $\varepsilon_1=+$ and $c_m\in B_1\setminus A_1$ if $\varepsilon_1 = -$, where 
$c_m=\max(A_1,B_1)$.
(Thus the character $\sigma$ alluded to in the statement's formulation is precisely $\varsigma^k\otimes\sgn^l$.)

We define $s'\in (G')^* = \SO((V')^\sharp) \cong \SO_{2n'+1}(q)$ such that $\diag(s',I_{2n-2n'-1})$ is conjugate to $s$ in $G$; this is possible because of \eqref{for-s20}.
Writing 
$$\CB_{(G')^*}(s') \cong (G')^*_1\times (G')^*_{-1} \times (G')^*_0$$
with  $(G')^*_{-1} = \GO((V')^\sharp_{-1})$,
we have natural isomorphisms $G^*_i\cong (G')^*_i$ for $i=-1$ and $i=0$.
Define $\psi'_i = \psi_i$ for these two values of $i$.  Let $\psi'_1$ be the unipotent character of $(G')^*_1 \cong \SO((V')^\sharp_1)$
associated to the symbol $S':=(B_1,\bar A_1)$ if $\varepsilon_1=+$, and $S':=(\bar B_1,A_1)$  if $\varepsilon_1=-$.
Note that this $S'$ is exactly the symbol constructed in Proposition \ref{main-piece} for $S:=(A_1,B_1)$.
Let $\chi'$ be the irreducible character of $\Sp_{2n'}(q)$ which lies in the series $\cE(G',s')$ and which is associated
to $\psi'_1\boxtimes \psi'_{-1}\boxtimes \psi'_0$.
By the construction and \eqref{Howe-even}, $(\chi,\chi')$ appears in the Howe correspondence between $G$ and $G'$.

Now suppose $(\theta,\chi')$ also appears in the Howe correspondence between $G$ and $G'$, and assume $\lev(\theta) \ge \lev(\chi)$.  
Since both $\theta$ and $\chi$ are in the Howe correspondence with the same character $\chi'$, \eqref{Howe-even} implies that $\theta$ is 
in the same series $\cE(G,s)$ as of $\chi$.
By the Lusztig correspondence, $\chi$ and $\theta$ correspond to unipotent characters $\psi_1\boxtimes \psi_{-1}\boxtimes \psi_0$
and $\varphi_1\boxtimes \varphi_{-1}\boxtimes \varphi_0$ respectively on $G_1^*\times G_{-1}^*\times G_0^*$.  Again because both characters appear in the Howe correspondence with $\chi'$, by \eqref{Howe-even} we have $\psi_{-1} = \varphi_{-1}$, $\psi_0=\varphi_0$, and the symbol $T=(A'_1,B'_1)$ of $\varphi_1$ is linked to $S'$. Now we have $\rank(T) = \rank(S)$ and $\defect(T) \equiv \defect(S) \pmod4$ by the choice of $S'$. Moreover,
since $\lev(\theta) \geq \lev(\chi)$, by \eqref{lev-gen} and \eqref{n'} we get
$$\frac{\dim V^\sharp}{2}-\bigl(\max(T)-\lfloor m_T/2 \rfloor \bigr) \geq \lev(\theta) \geq \lev(\chi) = \lev(S) + \frac{\dim V^\sharp_0+\dim V^\sharp_{-1}}2,$$
which implies that
$$\lev(T) = \frac{\dim V^\sharp_1}{2} - \max(T)+\lfloor m_T/2 \rfloor \geq \lev(S).$$ 
 By \eqref{big 1 rank} and Proposition~\ref{main-piece}, this implies $(A_1,B_1) = (A'_1,B'_1)$, so $\chi=\theta$.
\end{proof}

\begin{thm}\label{b-head}
Let $q$ be an odd prime power, $G = \SO_{2n+1}(q)$, $G' = \Sp_{2n'}(q)$, and $\chi$ be an irreducible character of $G$ of level $n' \geq 1$ with
$n>4(n')^2+3n'$. Then there exists an irreducible character $\chi'$ of $G'$ and a linear character $\sigma$ of $G$ such that $(\chi\otimes \sigma,\chi')$ appears in the
Howe correspondence for $(G,G')$.  Moreover, if $\theta$ is an irreducible character of $G$ such that $(\theta,\chi')$ also appears in the Howe correspondence for
$(G,G')$, then $\lev(\theta) \le \lev(\chi)$ with equality only if $\theta = \chi\otimes \sigma$.
\end{thm}

\begin{proof}
Using \eqref{max vs rank} and $n \geq 7n'+1$, we see that the maximum in \eqref{lev-gen} is achieved for exactly one value of $\kappa \in \{\pm 1\}$. 
Note that the spinor norm \cite[(22.11)]{Asch} gives rise to a linear character $\varsigma$ of order $2$ of $G$ which is trivial on $\Omega_{2n+1}(q)$, and tensoring $\chi$ with $\varsigma$ has the effect of interchanging the unordered pairs of half-symbols $\{A_1,B_1\}$ and $\{A_{-1},B_{-1}\}$, see \cite[Lemma 7.6]{Pan-T}.
Hence, there is a unique $k \in \{0,1\}$ such that, after replacing $\chi$ by $\chi \otimes \varsigma^k$ (thus taking $\sigma=\varsigma^k$), 
we have that the maximum in \eqref{lev-gen} is achieved for $\kappa=1$, and therefore \eqref{n'}--\eqref{big 1 rank} hold.

Let $S:=(A_1,B_1)$ and $m=m_S$ as in \eqref{m10}. Using \eqref{for-s20}  (and recalling Lemma \ref{no tie})
we can find an element 
$s'\in (G')^* = \SO((V')^\sharp) \cong \SO_{2n'+1}(q)$
such that 
\begin{itemize}
\item the $1$-eigenspace $(V')^\sharp_1$ of $s'$ on $(V')^\sharp$ is of (odd) dimension $\dim(V^\sharp_{-1})+1$; 
\item the $(-1)$-eigenspace $(V')^\sharp_{-1}$ of $s'$ is of (even) dimension $2n'-\dim(V^\sharp_{-1})-\dim(V^\sharp_0)$, of type $\varepsilon_1=+$ if $c_m=\max(A_1,B_1) \in A_1$ and $\varepsilon_1=-$ if $c_m\in B_1$; and
\item $s'$ acts as $-s_0$ on $(V')^\sharp_0$ which is of dimension equal to $\dim V^\sharp_0$.  
\end{itemize}
Writing $\CB_{(G')^*}(s') = (G')^*_1\times (G')^*_{-1} \times (G')^*_0$ with $(G')^*_{-1} = \GO((V')^\sharp_{-1})$, we have a natural isomorphism 
$G^*_0\cong (G')^*_0$;  define $\psi'_0 = \psi_0$.  Next, $(G')^*_1 \cong \SO((V')^\sharp_1)$ is dual to 
$G^*_{-1} = \Sp(V^\sharp_{-1})$. Hence we can define the unipotent character  
$\psi'_1$ to have the same symbol as of the unipotent character $\psi_{-1}$.  Let $\psi'_{-1}$ be the unipotent character of $(G')^*_{-1}$
associated to the symbol $S':=(B_1,\bar A_1)$ if $\varepsilon_1=+$ and $S':=(\bar B_1,A_1)$ if $\varepsilon_1=-$.
Note that this $S'$ is precisely the symbol constructed in Proposition \ref{main-piece-odd} for $S:=(A_1,B_1)$ which has defect $\equiv 1 \pmod4$ because 
$G^*_1 = \Sp(V^\sharp_1)$.
Let $\chi'$ be the irreducible character of $G'=\Sp_{2n'}(q)$ which lies in the series $\cE(G',s')$ and which is associated
to $\psi'_1\boxtimes \psi'_{-1}\boxtimes \psi'_0$.
By the construction and \eqref{Howe-odd}, $(\chi,\chi')$ appears in the Howe correspondence between $G$ and $G'$.

Now suppose that $(\theta,\chi')$ also appears in the correspondence between $G$ and $G'$, and assume $\lev(\theta) \ge \lev(\chi)$.  
Since both $\theta$ and $\chi$ are in the Howe correspondence with the same character $\chi'$, \eqref{Howe-odd} implies that $\theta$ is 
belongs to the same series $\cE(G,s)$ as of $\chi$.
By the Lusztig correspondence, $\chi$ and $\theta$ correspond to unipotent characters $\psi_1\boxtimes \psi_{-1}\boxtimes \psi_0$
and $\varphi_1\boxtimes \varphi_{-1}\boxtimes \varphi_0$ respectively on $G_1^*\times G_{-1}^*\times G_0^*$.  Again because both characters appear in the Howe correspondence with $\chi'$, by \eqref{Howe-odd} we have $\psi_{-1} = \varphi_{-1}$, $\psi_0=\varphi_0$, and the symbol $T=(A'_1,B'_1)$ of $\varphi_1$ is linked to $S'$. Now we have $\rank(T) = \rank(S)$ and $\defect(T) \equiv 1 \pmod4$. Moreover,
since $\lev(\theta) \geq \lev(\chi)$, by \eqref{lev-gen} and \eqref{n'} we get
$$\frac{\dim V^\sharp}{2}-\bigl(\max(T)-\lfloor m_T/2 \rfloor \bigr) \geq \lev(\theta) \geq \lev(\chi) = \lev(S) + \frac{\dim V^\sharp_0+\dim V^\sharp_{-1}}2,$$
which implies that
$$\lev(T) = \frac{\dim V^\sharp_1}{2} - \max(T)+\lfloor m_T/2 \rfloor \geq \lev(S).$$ 
By \eqref{big 1 rank} and Proposition~\ref{main-piece-odd}, this implies $(A_1,B_1) = (A'_1,B'_1)$, so $\chi=\theta$.
\end{proof}

\begin{cor}\label{irr-restr}
In the situation of Theorem \ref{d-head}, the character $\chi$ is irreducible over $[G,G] = \Omega^\varepsilon_{2n}(q)$, and hence also over 
$\SO^\varepsilon_{2n}(q)$. Similarly, in the situation of Theorem \ref{b-head}, the character $\chi$ is irreducible over $[G,G] = \Omega_{2n+1}(q)$.
In both cases, the linear character $\sigma$ and the character $\chi'$ are each uniquely determined by $\chi$.
\end{cor}

\begin{proof}
(a) Set $m=2n$, respectively $m=2n+1$. Replacing $\chi$ by $\chi \otimes \sigma$, where $\sigma$ is the character of $G/[G,G]$ in
Theorem \ref{d-head}, respectively Theorem \ref{b-head}, we have that $\chi$ appears in the restriction of a Weil character $\omega$ of 
$\Gamma:=\Sp(V)=\Sp_{2mn'}(q)$ to the image of $G$ under the homomorphism $G \times \Sp_{2n'}(q) \to \Gamma$. As shown in \cite[Proposition 3.1]{GLT2},
$\omega\bar\omega$ is the permutation character $\tau$ of $\Sp(V)$ on the point set of $V = \F_q^{2mn'}$. Hence $(\omega^2)|_G$ is the permutation character of $G$ on 
the set of $2n'$-tuples $\xi:=(v_1, \ldots,v_{2n'})$, $v_i \in A:=\F_q^{m}$ the natural module for $G$. Given any such $2n'$-tuple $\xi$, 
we can always put it in a non-degenerate $4n'$-dimensional subspace $B$ in $A$. Since $\mathrm{codim}_A B \geq 2(n-n') > 4$, $\Stab_G(\xi)$ contains 
elements that act trivially on $B$ and have any prescribed determinant and spinor norm on $B^\perp$. It follows that
$|\xi^G|= [G:\Stab_G(\xi)] = [[G,G]:\Stab_{[G,G]}(\xi)] = |\xi^{[G,G]}|$ and hence $\xi^G=\xi^{[G,G]}$. Thus $G$ and $[G,G]$ have the same orbits on the set of 
$2n'$-tuples, and so 
\begin{equation}\label{eq:sc1}
  [\omega|_G,\omega|_G]_G = [(\omega\bar\omega)|_G,1_G]_G =  [(\omega\bar\omega)|_{[G,G]},1_{[G,G]}]_{[G,G]} = [\omega|_{[G,G]},\omega|_{[G,G]}]_{[G,G]}.
\end{equation}  
This implies that every irreducible constituent of $\omega|_G$, including $\chi$, is irreducible over $[G,G]$. 

\smallskip
(b) Assume now that for a given $\chi$ of level $n'$, in addition to the pair $(\sigma,\chi')$ constructed in the proof of Theorem \ref{d-head},
respectively Theorem \ref{b-head}, there is another pair $(\tilde\sigma,\tilde\chi')$ such that $\tilde\sigma \in \Irr(G/[G,G])$ and 
$\chi \otimes \tilde\sigma$ are in the Howe correspondence with $\tilde\chi' \in \Irr(S)$,
where $S=\Sp_{2n'}(q)$. Then $\chi \otimes \sigma$ and $\chi \otimes \tilde\sigma$ restrict to the same character $\varphi$ of $[G,G]$, and $\varphi$ is 
irreducible by (a). On the other hand,
\eqref{eq:sc1} implies that distinct irreducible constituents of $\omega\vert_G$ must have distinct restrictions to $[G,G]$. Hence 
$\chi \otimes \sigma=\chi \otimes \tilde\sigma$, and hence $\sigma=\tilde\sigma$ by Gallagher's theorem \cite[(6.17)]{Is}.

We will now show that $\chi'=\tilde\chi'$. By the preceding result, we may assume that $\sigma=\tilde\sigma$. Suppose for instance that we are in the case of
Theorem \ref{d-head}, so that the maximum in \eqref{lev-gen} is achieved for $\kappa=1$, and moreover $\varep_1=+$, whence $S' = (B_1,\bar A_1)$. 
Suppose $\tilde\chi'$ lies in the series $\cE(G',\tilde s)$ and  is associated
to $\tilde\psi_1\boxtimes \tilde\psi_{-1}\boxtimes \tilde\psi_0$. Since $\chi$ corresponds to both $\chi'$ and $\tilde\chi'$, \eqref{Howe-even} implies that 
we may assume $\tilde s = s'$ and then have $\tilde\psi_{-1}=\psi_{-1}= \psi'_{-1}$, and $\tilde\psi_0= \psi_{0} = \psi'_0$. Moreover, if $(C,D)$ is the symbol for 
$\tilde\psi_1$, then $(A_1,B_1)$ is linked to $(C,D)$, and hence 
$$\pi(B_1) \preceq \pi(C),~\pi(D) \preceq \pi(A_1),~|C|-|D|= 1-(|A_1|-|B_1|) = |B_1|-|\bar A_1|$$
 by \eqref{split conditions}. In this case we also have $\pi(\bar A_1) \preceq \pi(D)$; and hence $|\pi(B_1)|+|\pi(\bar A_1)| \leq |\pi(C)|+|\pi(D)|$. On the other hand,
 $\rank(S')=n'=\rank(T)$, so \eqref{rk10} implies that equality must be attained in the last inequality. Thus $C=B_1$, $D=\bar A_1$, $T=S'$, 
 $\tilde\psi_1=\psi'_1$, and hence $\tilde\chi'=\chi'$. The same arguments work in all other cases. 
\end{proof}

\begin{cor}\label{so-bijection}
In the situation of Theorem \ref{d-head}, the map $\chi \mapsto \chi'$ gives a four-to-one surjection between the set of characters $\chi$ of 
$G = \GO^\varep_{2n}(q)$ of level $n'$  and $\Irr(\Sp_{2n'}(q))$, where each fiber consists of four characters with  
the same restriction to $[G,G] = \Omega^\varepsilon_{2n}(q)$.
Likewise, in the situation of Theorem \ref{b-head}, the map $\chi \mapsto \chi'$ gives a two-to-one surjection between the set of characters $\chi$ of 
$G = \SO_{2n+1}(q)$ of level $n'$  and $\Irr(\Sp_{2n'}(q))$, where each fiber consists of two characters with  
the same restriction to $[G,G] = \Omega_{2n+1}(q)$.
\end{cor}

\begin{proof}
In view of Corollary \ref{irr-restr}, it suffices to show that the map $\chi \mapsto \chi'$ is onto $\Irr(\Sp_{2n'}(q))$. 
Consider for instance the situation of Theorem \ref{d-head}. Then we need to show that, given any integer 
$0 \leq j \leq n'$, any symbol $S'=(C,D)$ of rank $n'-j$ and defect $\equiv 1 \pmod4$,
and any sign $\varep_1 \in \{\pm\}$, we can find a symbol $S=(A,B)$ of rank $n-j$, level $n'-j$, and such that 
$\defect(S) \equiv 0 \pmod 4$ and $S'=(B,\bar A)$ when $\varep_1=+$, and 
$\defect(S) \equiv 2 \pmod4$ and $S'=(\bar B,A)$ when $\varep_1=-$.  Suppose for instance $\varep_1=+$, and take $m:=|C|+|D|$ which is odd. 
By \eqref{max vs rank}, $\max(C,D) \leq \rank(S') \leq n'$, hence $c_m:= n+(m-1)/2-(n'-j) > \max(C,D)$. Defining $S=(A,B)$ with 
$A := \{c_m\} \sqcup D$ and $B:= C$, we have $S'=(B,\bar A)$, $m=m_S$, $\defect(S)=1-\defect(S') \equiv 0 \pmod4$, $\max(A,B)=c_m$,
$$\rank(S)=\rank(S') +c_m+\frac{m^2-1}{4}-\frac{(m-1)^2}{4}=n$$
by \eqref{rk10}, and 
$$\lev(S) = \rank(S)-c_m+\frac{m-1}{2} = n'-j$$
by \eqref{lev-uni}, 
as desired. The same arguments work in all of the remaining cases.
\end{proof}

For the next result, recall that in the case of symplectic groups $G$ in odd characteristics, the level of $\chi \in \Irr(G)$ is only half-integral (see the comments after 
\eqref{lev-gen}).

\begin{thm}\label{c-head}
Let $q$ be any odd prime power, $G = \Sp_{2n}(q)$ and $\chi$ be an irreducible character of $G$ of level $n' \geq 1$ with
$n>4(n')^2+3n'$. Then there exist a (special if $n' \notin \Z$ and full if $n' \in \Z$)
orthogonal group $G'$ over $\F_q$ of dimension $2n'$,  and an
irreducible character $\chi'$ of $G'$, such that $(\chi,\chi')$ appears in the
Howe correspondence for $(G,G')$.  Moreover, if $\theta$ is an irreducible character of $G$ such that $(\theta,\chi')$ also appears in the Howe correspondence for
$(G,G')$,
then $\lev(\theta) \le \lev(\chi)$ with equality only if $\theta = \chi$.
\end{thm}

\begin{proof}
Let $(A_1,B_1)$ and $(A_{-1},B_{-1})$ be the symbols associated to $\psi_1$ and $\psi_{-1}$ respectively in the Lusztig correspondence for $\chi$, so 
$\defect(A_1,B_1)$ is $1$ (mod $4$)
and $\defect(A_{-1},B_{-1})$ is even.  Define $\varepsilon_{-1}$ to be $+$ if the latter defect is divisible by $4$
and $-$ if not.  Also, let $\varepsilon_0$ denote the type of the even-dimensional orthogonal space $V^\sharp_0$.
Using \eqref{max vs rank} and $n \geq 7n'+1$, we see that the maximum in \eqref{lev-gen} is achieved for exactly one value of $\kappa \in \{\pm 1\}$. 

\smallskip
(a) We consider first the case that the maximum in \eqref{lev-gen} is achieved for $\kappa=1$. In particular, \eqref{n'} holds, $n' \in \Z_{\geq 1}$, and
\begin{equation}\label{for-s21}
  \dim V^\sharp_1 \ge  \dim V^\sharp - 2n' = 2n+1-2n', 
\end{equation}   
whence
\begin{equation}
\label{big 2 rank}
\rank(A_1,B_1) = \frac{\dim(V^\sharp_1)-1}{2} > 4\,\lev(A_1,B_1)^2 + 2\,\lev(A_1,B_1).
\end{equation}
Using Lemma \ref{no tie}, we
define $\varepsilon_1$ to be $+$ if $c_m\in A_1$ and  $-$ if $c_m\in B_1$. Recalling \eqref{for-s21}, we can find some $\varepsilon\in\{\pm\}$ and some element 
$s'\in \SO(V') \cong \SO^\varepsilon_{2n'}(q)$
such that
\begin{itemize}
\item $\diag(s',I_{2n+1-2n'})$ is conjugate to $s$ in $\GL_{2n+1}(q)$;  
\item the restriction of the quadratic form $\sf Q$ on $V'=\F_q^{2n'}$ to the $1$-eigenspace $V'_1$ of 
$s'$ on $V'$ has type $\varepsilon_1$;
\item the restriction of $\sf Q$ to the $(-1)$-eigenspace $V'_{-1}$ has type $\varepsilon_{-1}$; and
\item  the restriction of $\sf Q$ to the subspace $V'_{0}$ has type $\varepsilon_{0}$.
\end{itemize}
Let $G' = \GO(V') \cong \GO^\varepsilon_{2n'}(q) \cong (G')^*$, and define $\psi'_i$, $i=0,\pm 1$ by the following conditions:
\begin{itemize}
\item $\psi'_1$ has symbol $S':=(B_1,\bar A_1)$ if $\varepsilon_1=+$ and symbol $S':=(\bar B_1,A_1)$ if $\varepsilon_1=-$;
\item $\psi'_{-1}$ has symbol $(A_{-1},B_{-1})$; and
\item $\psi'_0 = \psi_0$.
\end{itemize}
Note that $S'$ is precisely the symbol constructed in Proposition \ref{main-piece-odd} for $S:=(A_1,B_1)$; in particular, $S'$ is linked to $S$.
By \eqref{Howe-even}, $(\chi,\chi')$ appears in the Howe correspondence for $(G,G')$.

Now suppose $(\theta,\chi')$ also appears in the Howe correspondence between $G$ and $G'$, and assume $\lev(\theta) \ge \lev(\chi)$.  
Since both $\theta$ and $\chi$ are in the Howe correspondence with the same character $\chi'$, \eqref{Howe-even} implies that $\theta$ is 
in the same series $\cE(G,s)$ as of $\chi$.
By the Lusztig correspondence, $\chi$ and $\theta$ correspond to unipotent characters $\psi_1\boxtimes \psi_{-1}\boxtimes \psi_0$
and $\varphi_1\boxtimes \varphi_{-1}\boxtimes \varphi_0$ respectively on $G_1^*\times G_{-1}^*\times G_0^*$.  Again because both characters appear in the Howe correspondence with $\chi'$, by \eqref{Howe-even} we have $\psi_{-1} = \varphi_{-1}$, $\psi_0=\varphi_0$, and the symbol $T=(A'_1,B'_1)$ of $\varphi_1$ is linked to $S'$. Now we have $\rank(T) = \rank(S)$ and $\defect(T) \equiv 1 \pmod4$. Moreover,
since $\lev(\theta) \geq \lev(\chi)$, by \eqref{lev-gen} and \eqref{n'} we get
$$\frac{\dim(V^\sharp)-1}{2}-\bigl(\max(T)-\lfloor m_T/2 \rfloor \bigr) \geq \lev(\theta) \geq \lev(\chi) = \lev(S) + \frac{\dim V^\sharp_0+\dim V^\sharp_{-1}}2,$$
which implies that
$$\lev(T) = \frac{\dim(V^\sharp_1)-1}{2} - \max(T)+\lfloor m_T/2 \rfloor \geq \lev(S).$$ 
By \eqref{big 2 rank} and Proposition~\ref{main-piece-odd}, this implies $(A_1,B_1) = (A'_1,B'_1)$, so $\chi=\theta$.

\smallskip
(b) Now consider the case that the maximum in \eqref{lev-gen} is achieved for $\kappa=-1$.
It follows that
\begin{equation}
\label{np2}
n'=\lev(\chi)= \lev(A_1,B_1) + \frac{\dim V^\sharp_0+\dim V^\sharp_{1}}2
\end{equation}
is half-integral (but not integral, since $\dim V^\sharp_1$ is odd and $\dim V^\sharp_0$ is even), and
\begin{equation}\label{for-s22}
  \dim V^\sharp_{-1} \ge  \dim V^\sharp - 2n' = 2n+1-2n', 
\end{equation}   
whence
\begin{equation}
\label{big 3 rank}
\rank(A_{-1},B_{-1}) = \frac{\dim V^\sharp_{-1}}{2} > 4\,\lev(A_{-1},B_{-1})^2 + 2\,\lev(A_{-1},B_{-1}).
\end{equation}

Let $G' = \SO(V') \cong \SO_{2n'}(q)$ (recall that $2n' \geq 3$ is an odd integer).
Then using \eqref{for-s22} we can find an element $s'\in (G')^* = \Sp((V')^\sharp) \cong \Sp_{2n'-1}(q)$ such that 
\begin{itemize}
\item $\diag(-s',-I_{2n-2n'+2})$ is conjugate to $s$ in $\GL_{2n+1}(q)$;
\item the $(-1)$-eigenspace $(V')^\sharp_{-1}$ of $s'$ is of (even) dimension $\dim(V^\sharp_{1})-1$; 
\item $s'$ acts as $-s_0$ on $(V')^\sharp_0$ which is of dimension equal to $\dim V^\sharp_0$.  
\end{itemize}
Define $\psi'_i$, $i = 0,\pm 1$, satisfying the following conditions
\begin{itemize}
\item $\psi'_1$ has symbol $S':=(B_{-1},\bar A_{-1})$ if $\varepsilon_{-1}=+$ and $S':=(\bar B_{-1},A_{-1})$ if $\varepsilon_{-1}=-$;
\item $\psi'_{-1}$ has symbol $(A_1,B_1)$; and
\item $\psi'_0=\psi_0$.
\end{itemize}
Note that $S'$ is precisely the symbol constructed in Proposition \ref{main-piece} for $S:=(A_{-1},B_{-1})$; in
particular, $S$ is linked to $S'$.  By \eqref{Howe-odd}, $(\chi,\chi')$ appears in the Howe correspondence for $(G,G')$.

Now suppose $(\theta,\chi')$ also appears in the Howe correspondence between $G$ and $G'$, and assume $\lev(\theta) \ge \lev(\chi)$.  
Since both $\theta$ and $\chi$ are in the Howe correspondence with the same character $\chi'$, \eqref{Howe-odd} implies that $\theta$ is 
in the same series $\cE(G,s)$ as of $\chi$.
By the Lusztig correspondence, $\chi$ and $\theta$ correspond to unipotent characters $\psi_1\boxtimes \psi_{-1}\boxtimes \psi_0$
and $\varphi_1\boxtimes \varphi_{-1}\boxtimes \varphi_0$ respectively on $G_1^*\times G_{-1}^*\times G_0^*$.  Again because both characters appear in the Howe correspondence with $\chi'$, by \eqref{Howe-odd} we have $\psi_{1} = \varphi_{1}$, $\psi_0=\varphi_0$, and the symbol $T=(A'_{-1},B'_{-1})$ of $\varphi_{-1}$ is linked to $S'$. Now we have $\rank(T) = \rank(S)$ and $\defect(T) \equiv \defect(S) \pmod4$ by the choice of $S'$. Moreover,
since $\lev(\theta) \geq \lev(\chi)$, by \eqref{lev-gen} and \eqref{np2} we get
$$\frac{\dim V^\sharp}{2}-\bigl(\max(T)-\lfloor m_T/2 \rfloor \bigr) \geq \lev(\theta) \geq \lev(\chi) = \lev(S) + \frac{\dim V^\sharp_0+\dim V^\sharp_{1}}2,$$
which implies that
$$\lev(T) = \frac{\dim V^\sharp_{-1}}{2} - \max(T)+\lfloor m_T/2 \rfloor \geq \lev(S).$$ 
By \eqref{big 3 rank} and Proposition~\ref{main-piece}, this implies $(A_{-1},B_{-1}) = (A'_{-1},B'_{-1})$, so $\chi=\theta$.

\smallskip
(c) Finally, let us comment on the choice of a Lusztig correspondence. As proved in \cite{Pan-T}, in the case $H = \GO^\pm_{2n}(q)$ or 
$H=\SO_{2n+1}(q)$, there is a unique choice of a Lusztig correspondence for $H$ which satisfies some specific properties including being 
compatible with the Howe correspondence for the pairs $(H,\Sp_{2m}(q))$. In the case of  $G=\Sp_{2n}(q)$, it is proved in \cite{Pan-T} that 
there is a unique choice of a Lusztig correspondence for $G$ which satisfies some specific properties including being 
compatible with the Howe correspondence for the pairs $(G,\GO^\pm_{2m}(q))$ (that occur in part (a) of this proof), and there is a unique choice of a Lusztig correspondence for $G$ which satisfies 
some specific properties including being compatible with the Howe correspondence for the pairs $(G,\SO_{2m+1}(q))$ (that appear only in part (b)).  
This however does not cause any ambiguity for our result. Indeed, as follows from \eqref{for-s21} and \eqref{for-s22}, in the case of (a) we have
$\dim V^\sharp_1 > (\dim V^\sharp)/2$ whereas in (b) we have $\dim V^\sharp_{-1} > (\dim V^\sharp)/2$. These inequalities can be read off from the labeling element 
$s$ of the rational series $\cE(G,s)$ that contains $\chi$, independently of the choice of Lusztig correspondence. Hence for any given $\chi$ whose level $j$
satisfies $n > 4j^2+3j$ one can determine unambiguously whether it belongs to case (a) or (b) of the proof, at which point we will work with only one kind of Howe dual 
pairs.
\end{proof}

Analogues of Corollaries \ref{irr-restr} and \ref{so-bijection} also hold for $\Sp_{2n}(q)$, although we do not need them in the rest of the paper.

\section{A sharp character estimate for odd-characteristic orthogonal groups}
In this section we prove the following result, which implies Theorem B:

\begin{thm}\label{so-bound}
Let $N,m,j \in \Z_{\geq 1}$, $q$ any odd prime power, and let $G=\SO^\varep_N(q)$ with $\varep= \pm$. 
\begin{enumerate}[\rm(i)]
\item Suppose that
$$N \geq 8j^2+4j +\max(4m,2j)+2$$ 
and $\chi$ is an irreducible complex character of $G$ of level $j$.
Then there exists some $\beta=\beta(\chi,\lambda) \in \{1,-1\}$ such that when $g\in G$ has support $m$ and primary eigenvalue $\lambda = \pm 1$, we have
\begin{equation}
\label{SO-bound1}
\Bigm|\frac{q^{mj}\chi(g)}{\chi(1)}-\beta\Bigm| < q^{-N/4}.
\end{equation}
Moreover, if $2|N$ or $\lambda=1$, then $\beta = \chi(\lambda \cdot 1_G)/\chi(1)$. 
\item If $N \geq 8j^2+6j+2$, $\chi$ is an irreducible complex character of $G$ of level $j \geq 1$, and $g\in G$ has support 
$m \geq 2j+7$, then
\begin{equation}
\label{SO-bound2}
  \Bigl|\frac{\chi(g)}{\chi(1)}\Bigr| < \left\{\begin{array}{rl}q^{3-m}, & j=1,\\
  q^{-m}, & j \geq 2.\end{array}\right. 
\end{equation}
\end{enumerate}
\end{thm}

\begin{proof}
(a) We prove the statements by induction on $j\geq 0$, with the induction base $j=0$ being obvious.
The statements are also obvious for $m=0$, so we will assume $m \geq 1$. Write $n = \lfloor N/2 \rfloor$.

Let $\chi \in \Irr(G)$ have level $j \geq 1$. 
Let $S := \Sp_{2j}(q)$, and let $\omega$ denote a (reducible) Weil character of degree $q^{Nj}$ of $\Gamma:=\Sp_{2Nj}(q)$.  
We write $G=\SO(A)$ with $A:=\F_q^N$ and $S=\Sp(B)$ with $B := \F_q^{2j}$. Then $V := A\otimes_{\F_q} B$ is endowed with the tensor product of the 
orthogonal form on $A$ and the symplectic form on $B$, and we can take $\Gamma = \Sp(V)$. 
The restriction of $\omega$ from $\Gamma$ to the image $G \ast S$ of $G \times S$ in $\Gamma$ gives a decomposition
$$\omega|_{G\ast S} = \sum_{\alpha\in \Irr(S)} D_\alpha\boxtimes \alpha,$$
where $D_\alpha$ is given by
\begin{equation}\label{so-dual1}
  D_\al(g) = \frac{1}{|S|}\sum_{s \in S}\omega(g \otimes s)\bar\al(s).
\end{equation}  
In particular, any irreducible constituent of $D_\al$ lies in the Howe correspondence for $(G,S)$ with $\alpha$. 
Since $N \geq 8j^2+6j+2$, $\chi$ extends to an irreducible character of $\GO^\varep_N(q)$ by Corollary \ref{irr-restr}. Multiplying $\chi$ with the spinor
character if necessary, we may assume by Theorems \ref{d-head} and \ref{b-head} that for some $\al \in \Irr(S)$, $D_\al$ has a unique irreducible constituent 
$D^\circ_\al$ of level $j$ which is $\chi$, and all other
irreducible constituents have strictly lower level.

Next we verify \eqref{SO-bound2} directly when $j=1$. Since 
$m = \supp(g)$, we have $\dim\Ker(g-\lambda \cdot \Id_A) \leq N-m$ for all $\lambda \in \overline{\F_q}^\times$. It follows from \cite[Lemma 8.2]{GLT1} 
that $d_V(g \otimes s) = \dim \Ker(g \otimes s - \Id_V) \leq 2(N-m)$, and hence $|\omega(g \otimes s)| \leq q^{N-m}$ for any $s \in S$. Also, according to
Tables I and II in \cite{LOST}, $\chi = D_\al$ or $D_\al-1_G$, and hence
$$\chi(1) \geq (1-q^{2-n})q^N\al(1)/|S|,$$
see the proof of \cite[Proposition 5.12]{LOST}.
It follows from \eqref{so-dual1} that
$$\Bigl| \frac{\chi(g)}{\chi(1)}\Bigr| \leq \frac{q^{N-m}\al(1)|S|}{(1-q^{2-n})q^N\al(1)} = q^{3-m}\frac{1-q^{-2}}{1-q^{2-n}} <q^{3-m}.$$ 
As shown in Example \ref{level1} (below), this upper bound $q^{3-m}$ is optimal (up to a linear factor).

\smallskip
Consider any $g \in G$ of support $m \geq 1$. We will
bound $\DC_\al(1)$ and $|\DC_\al(g)|$ using \eqref{so-dual1}. Embed $G$ as $G \times \{1\}$ in $G \ast S$, we can write 
\begin{equation}\label{so-dual2}
  \omega|_G = \sum^M_{i=1}a_i\theta_i,~~D'_\al := D_\al-\DC_\al = \sum^{M'}_{i=1}b_i\theta_i,
\end{equation}  
where $\theta_i \in \Irr(G)$ are pairwise distinct, $a_i,b_i \in \Z_{\geq 0}$,
$M \geq M'$. Moreover, if $1 \leq i \leq M'$, then $a_i \geq b_i$, and $l:=\lev(\theta_i) \leq j-1$.  By Theorems 
\ref{d-head} and \ref{b-head}, any such $\theta_i$, tensored with a unique linear character, appears as $D^\circ_{\beta_i}$ 
in the Howe correspondence for $G \times \Sp_{2l}(q)$, where $\beta \in \Irr(\Sp_{2l}(q))$. Hence $\theta_i(1) \leq q^{N(j-1)}$, and, as  
$D^\circ_{\beta_i}$ is the unique irreducible constituent of level $l$ of $D_{\beta_i}$, we have  
$$M' \leq \sum^{j-1}_{i=0}k(\Sp_{2i}(q)) \leq 10.8\sum^{j-1}_{i=0}q^i < 5.4q^j.$$
Here, $k(X) = |\Irr(X)|$ denotes the class number of a finite group $X$ as before, and $k(\Sp_{2i}(q)) \leq 10.8q^i$ by \cite[Theorem 3.12]{FG}. 
Also, $\sum^M_{i=1}a_i^2 = [\omega|_G,\omega|_G]_G \leq 12q^{j(2j+1)}$ by \cite[Lemma 2.6]{GLT2} (since 
$\omega\vert_G$ is the $j^{\mathrm{th}}$ power of the permutation character $\tau_A$ of $G$ on $A=\F_q^{N}$, see e.g. \cite[Proposition 3.1(iii)]{GLT2},
and therefore $\omega\overline\omega\vert_G$ is the permutation character of $G$ acting on the ordered $2j$-tuples of vectors from $A$, and we add a factor of $2$ to go from $\GO$-orbits to $\SO$-orbits). It follows that
\begin{equation}\label{for-m30}
  (\sum^{M'}_{i=1}b_i)^2 \leq M'\sum^{M'}_{i=1}b_i^2 < 5.4q^j\sum^M_{i=1}a_i^2 < 65q^{2j(j+1)},
\end{equation}  
and so
\begin{equation}\label{so-dual3}
  D'_\al(1) \leq \sum^{M'}_{i=1}b_iq^{N(j-1)} < \sqrt{65q^{2j(j+1)}}q^{N(j-1)} \leq q^{N(j-1)+j(j+1)+2}.
\end{equation}

\smallskip
(b) We now complete the induction step for (i). Recall that $g$ 
has support $m \leq N/4$, which implies that $g$ has a 
primary eigenvalue $\lambda$ with 
$\lambda^2=1$, and 
$$\delta_A(g) = \dim\Ker(g-\lambda \cdot \Id_A)= N-m > N/2.$$
If $2|N$ or $\lambda=1$, we can multiply $g$ by the central element $\lambda \cdot  I_N$ of $G$, which changes $\chi(g)$ by a factor of 
$\beta=\chi(\lambda \cdot I_N)/\chi(1)$; in such a case we will now assume that $\lambda=1$ and $\beta=1$.
Consider any $s \in S = \Sp(B)$. By \cite[Lemma 8.2]{GLT1} we have 
\begin{equation}\label{so-chi10a}
  d_V(g \otimes s):=\dim \mathrm{Ker}({g\otimes s}-\Id_V) \leq 2j(N-m).
\end{equation}  
In fact, this bound is attained for $s_0=\lambda \cdot I_{2j}$. For all other elements $s\neq s_0$ in $S$ we have the stronger bound
\begin{equation}\label{SO-chi10}
  d_V(g \otimes s) \leq (N-m)(2j-2)+N = 2N^*+N, \mbox{ where }N^*= (N-m)(j-1).
\end{equation}  
Note that $|\omega(g \otimes s)|^2 = q^{d_V(g \otimes s)}$ by \cite[Proposition 3.1(i)]{GLT2}. Next, if $(e_1, \ldots,e_j,f_1, \ldots,f_j)$ is a
symplectic basis of $B$ (so that $B_1:=\langle e_1, \ldots,e_j \rangle_{\F_q}$ and $B_2:=\langle f_1, \ldots,f_j \rangle_{\F_q}$ are two 
complementary maximal totally isotropic subspaces of $B$), then $A \otimes_{\F_q}B_1$ and $A \otimes_{\F_q}B_2$ are two 
complementary maximal totally isotropic subspaces of $V$ which are both fixed by $g \otimes s_0$. Thus the element 
$g \otimes s_0$ is contained in the Levi subgroup $\GL_{Nj}(q)$ in the Siegel parabolic subgroup $\mathrm{Stab}_{\Gamma}(A \otimes_{\F_q}B_1)$. 
It is well known, see e.g. \cite[(13.3)]{Gr}, that the restriction of $\omega$ to the derived subgroup $\SL_{Nj}(q)$ of this Levi subgroup is just its permutation character 
on $\F_q^{Nj}$. If $2|N$ or $\lambda=1$, then $g \otimes s_0$ belongs to $\SL_{Nj}(q)$. It follows that  
$$\frac{\omega(g \otimes s_0)\al(s_0)}{\al(1) }= q^{(N-m)j} = \beta q^{(N-m)j}.$$ 
If $2 \nmid N$ and $\lambda=-1$, then 
$g \otimes s_0$ acts on $A \otimes_{\F_q} B_1$ with determinant $(-1)^{Nj}$; set $\beta_0$ to be the Legendre symbol
$\left(\frac{(-1)^j}{q}\right)$, and let $\beta=\beta_0\alpha(s_0)/\al(1)$. Then using \cite[(13.3)]{Gr} we get
$$\frac{\omega(g \otimes s_0)\al(s_0)}{\al(1)} = \beta q^{(N-m)j}.$$ 
By \eqref{so-dual1} we can now write
\begin{equation}\label{SO-chi11}
  \chi(g)=\DC_\al(g) = \frac{\al(1)}{|S|}\Bigl(  \beta q^{(N-m)j} + X \Bigr),
\end{equation}  
where 
$$X := \sum_{s_0 \neq s \in S}\frac{\omega(g \otimes s)\bar\al(s)}{\al(1)}-\frac{|S|}{\al(1)}D'_\al(g).$$
Recall that for $1 \leq i \leq M'$, $\theta_i$ has level $0 \leq l \leq j-1$ and degree $\leq q^{Nl}$, so by the induction hypothesis we have 
$$|\theta_i(g)| \leq \frac{\theta_i(1)}{q^{ml}}\bigl(1+q^{-N/4}\bigr).$$
Here, as $N \geq 16$ we have $1+q^{-N/4} \leq 1+q^{-4} \leq 1+1/81 < q^{0.1}.$ 
It follows that 
$$|\theta_i(g)| \leq q^{(N-m)l+0.1} \leq q^{N^*+0.1}.$$ 
Together with \eqref{for-m30}, this implies that
$$|D'_\al(g)| \leq \sqrt{65q^{2j(j+1)}}q^{N^*+0.1} < q^{N^*+j(j+1)+2}.$$
Furthermore, 
$$|S| = |\Sp_{2j}(q)| < q^{2j^2+j}.$$ 
Using \eqref{so-dual3} and \eqref{SO-chi10}, we have
$$|X| \leq q^{N^*+N/2+2j^2+j}+q^{N^*+3j^2+2j+2}.$$
Similarly, $\dim \Ker (\Id_A \otimes s-\Id_V) = N\dim\Ker(s-\Id_B) \leq N(2j-1)$ for $s \neq 1_S$, and so using \eqref{so-dual3} we have
\begin{equation}\label{SO-chi12}
  \chi(1)=\DC_\al(1)=\frac{\al(1)}{|S|}\Bigl(  q^{Nj} + Y \Bigr),
\end{equation}  
where
\begin{equation}\label{SO-chi12a} 
   |Y| := \Bigl{|}\sum_{1_S \neq s \in S}\frac{\omega(1_G\otimes s)\bar\al(s)}{\al(1)}-\frac{|S|}{\al(1)}D'_\al(1)\Bigr{|} \leq q^{N(j-1/2)+2j^2+j}+q^{N(j-1)+3j^2+2j+2}.
\end{equation}  
Now, \eqref{SO-chi11} and \eqref{SO-chi12} imply
$$\Bigm| \frac{q^{mj}\chi(g)}{\chi(1)}-\beta\Bigm| = \Bigm| \frac{\beta q^{Nj}+Xq^{mj}}{q^{Nj}+Y}-\beta\Bigm| = \Bigm|\frac{Xq^{mj}-\beta Y}{q^{Nj}+Y}\Bigm|.$$
Setting $R:= q^{N/2+2j^2+j}+q^{3j^2+2j+2}$, 
we have 
$$|Xq^{mj}-\beta Y| \leq R\bigl(q^{N^*+mj}+q^{N(j-1)}\bigr)=q^{N(j-1)}R(q^m+1),~~|q^{Nj}+Y| \geq q^{N(j-1)}(q^N-R).$$
Hence to prove \eqref{bound1}, it suffices to prove
$$R\bigl((q^m+1)q^{N/4}+1\bigr) < q^{N}.$$
Note that $q^m+1 \leq q^m(1+1/3) < q^{m+0.3}$, and $q^{m+N/4+0.3}+1  < q^{m+N/4+0.4}$ since $N \geq 16$ and $q \geq 3$. 
Also, $N/2 \geq 4j^2+3j+1$, whence $R \leq q^{N/2+2j^2+j+0.1}$. It follows that
$$R\bigl((q^m+1)q^{N/4}+1\bigr) < q^{3N/4+m+2j^2+j+0.5} \leq q^N$$
since $N \geq 8j^2+4j+4m+2$, and so we are done.

\smallskip
(c)  Next we complete the induction step for (ii); in particular $j \geq 2$. 
By \eqref{so-dual1} we have 
$$\chi(g)=\DC_\al(g)=  \frac{\al(1)Z}{|S|},$$
where 
$$Z := \sum_{s \in S}\frac{\omega(g \otimes s)\bar\al(s)}{\al(1)}-\frac{|S|}{\al(1)}D'_\al(g).$$
By \eqref{so-chi10a}, the first sum has absolute value at most $|S|q^{(N-m)j}$.
For $1 \leq i \leq M'$, $\theta_i$ has true level $0 \leq l \leq j-1$ and degree $\leq q^{Nl}$. If $l \geq 1$, then the induction hypothesis implies that
$$|\theta_i(g)| \leq q^{3-m}\theta_i(1) \leq q^{3-m+N(j-1)}.$$
Since $j \geq 2$, the same bound also holds for $l=0$.
Together with \eqref{for-m30}, this implies that
$$|D'_\al(g)| \leq \sqrt{65q^{2j(j+1)}}q^{N(j-1)+3-m} < q^{N(j-1)-m+j^2+j+5}.$$
It follows that
\begin{equation}\label{SO-chi12b}
  |Z| \leq q^{(N-m)j+2j^2+j}+q^{N(j-1)-m+3j^2+2j+5}.
\end{equation}  
Since \eqref{SO-chi12} also holds in this case, to prove \eqref{SO-bound2}, it suffices to prove
$$q^m|Z| +|Y| < q^{Nj}.$$
The assumptions on $m,N$ (and $j \geq 2$) imply that $(N-m)j+2j^2+j+m \leq Nj-1$, and 
$$\max(N(j-1)+3j^2+2j+5,N(j-1/2)+2j^2+j) \leq Nj-13.$$
Using \eqref{SO-chi12a} and \eqref{SO-chi12b}, we now have 
$$q^m|Z| +|Y| \leq q^{Nj}(q^{-1}+3q^{-13}) < q^{Nj},$$
and so we are done again.
\end{proof}

\begin{exa}\label{large-level}
{\em Let $G$ be any of the finite classical groups $\GL_N(q)$, $\GU_N(q)$, $\Sp_N(q)$, or $\SO^\varep_N(q)$, of (semisimple) rank $r$. Then the Steinberg 
character $\St$ of $G$ has level $r$ (see e.g. \cite[Example 3.11]{GLT1}), and it vanishes at any non-semisimple element $g$, whose support can 
be taken to be small or close to $N$. This example shows that one cannot remove the low-level assumption in Theorems~\ref{mainB}, \ref{mainC}.

Next, let $G = \GL_N(q)$ with $N \geq 2$, $(N,q) \neq (2,2)$, and let $\chi \in \Irr(G)$ be a unipotent Weil character of degree $(q^N-q)/(q-1)$. Then $\chi$ has level 
$1$ (see again \cite[Example 3.11]{GLT1}), and $\chi(g)=0$ for any regular semisimple element $g$ of order $q^{N-1}-1$ with $1$ as an eigenvalue, which has 
support $m=N-1$. This example shows that one cannot completely remove the constraint on $m$ set in Theorem~\ref{mainB}. Similar examples can 
be constructed for $\GU_N(q)$ and for groups in Theorem~\ref{mainC}.}
\end{exa}

\begin{exa}\label{level1}
{\em Let $q \geq 5$ be any odd prime power and let $1 \leq m \leq N$ be integers such that $N-2 \geq m \geq N(q-4)/(q-3)$ and $2|N$. 
Then we can construct an element 
$g \in \SO(V)\cong \SO^\pm_N(q)$, which has only eigenvalues $\lambda \in \F_q^\times \setminus \{-1\}$ as follows. Each 
$\lambda \in \F_q^\times \setminus \{\pm 1\}$ is an eigenvalue of $g$ with multiplicity $N-m = \dim \Ker(g-\lambda \cdot \Id_V)$. Next, $1$ is an eigenvalue of
$g$ with multiplicity $e:=N-(q-3)(N-m)$, and if $e > 0$, then $g$ has two Jordan blocks with eigenvalue $1$ of size $1$ and $e-1$, so that
$\dim \Ker(g- \Id_V) \leq 2$ and $\supp(g)=m$. In the notation of 
\cite[Table II]{LOST}, $\rho:=1_G+D^\circ_{1_S}+D^\circ_{\St}$ is the permutation character of $G$ acting on the set of singular $1$-spaces of $V$. For any 
$\lambda \in \F_q^\times \setminus \{\pm 1\}$, $\Ker(g-\lambda \cdot \Id_V)$ is totally singular, and each of its $1$-spaces is $g$-invariant, whence
$$\rho(g) \geq (q-3)(q^{N-m}-1)/(q-1).$$
On the other hand, using the notation of the proof of \cite[Propositions 5.7, 5.11]{LOST}, for any $x \in S=\Sp_2(q)$, we note that $|\omega_N(x\otimes g)|$ is at most 
$q^2$ when $x=I_2$ or $x$ is nontrivial unipotent, $q^{N-m}$ when $x$ has eigenvalues $\lambda^{\pm 1}$ with $\pm 1 \neq \lambda \in \F_q^\times$ (in which 
case $\St(x)=1$), and $1$ otherwise. It follows that
$$\begin{aligned} \bigl{|} D^\circ_{1_S}(g)-D^\circ_{\St}(g)\bigr{|}  & = \bigl{|} D_{1_S}(g)-D_{\St}(g)\bigr{|}=  \bigl{|} \frac{1}{|S|}\sum_{x \in S}\omega_N(x\otimes g)(1-\St(x))\bigr{|}\\
&  \leq 
   \frac{q^2(q-1)+q^2(q^2-1)}{q(q^2-1)} < q+1,\end{aligned}$$ 
and therefore that for $\chi:= D^\circ_{1_S}$, $\chi(g) \geq (q-3)(q^{N-m}-1)/2(q-1) - (q+2)/2 \approx q^{N-m}/2$. Since $\chi(1) \approx q^{N-3}$, we see
that $\chi(g)/\chi(1)$ satisfies a lower bound on the order of $q^{3-m}/2$.

A similar example can be constructed for $2 \nmid N$.}
\end{exa}

Recall that the {\it $U$-rank} $\mathsf{r}(\chi)$ of a character $\chi$ of a finite classical group was defined in \cite{GH2} and \cite[Definition 4.1]{GLT2}.  Roughly, the idea is to restrict $\chi$ to a subgroup of $G$ of the form 
$U=\left\{\begin{pmatrix}I_{N/2}&*\\ 0&I_{N/2}\end{pmatrix}\right\}$, identify the irreducible characters of $U$ with $N/2\times N/2$-matrices, and take the maximum rank which appears in the decomposition.

\begin{cor}\label{so-rank}
In the situation of Theorem \ref{d-head}, respectively Theorem \ref{b-head}, the $U$-rank of $\chi$ is $2n'=2\lev(\chi)$.
\end{cor}

\begin{proof}
Note that tensoring $\chi$ with a linear character of $G/[G,G]$ does not change $\mathsf{r}(\chi)$. So we may assume that $(\chi,\chi')$ appears in
the Howe correspondence for $(G,S)$ with $S = \Sp_{2n'}(q)$ and $n'=\lev(\chi)$: $\chi = \DC_\al$ for some $\al \in \Irr(S)$, in the notation of \eqref{so-dual1}. 
By \cite[Corollary 4.7]{GLT2}, some irreducible constituent $\chi_0$ of $D_\al$ has $U$-rank $2n'$. On the other hand, by Theorems \ref{d-head} and \ref{b-head},
any irreducible constituent $\theta$ of $D_\al-\chi$ has level $j \leq n'-1$, and hence, possibly tensored with a linear character of $G/[G,G]$, appears in
the Howe correspondence for $(G,\Sp_{2j}(q))$. The latter implies $\mathsf{r}(\theta) \leq 2j < 2n'$ by the proof of \cite[Corollary 4.4]{GLT2}. It follows
that $\chi_0=\chi$ and thus $\mathsf{r}(\chi)=2n'$.
\end{proof}

Using Corollary \ref{so-rank}  one obtains a {\it lower} bound on $\chi(1)$ in terms of its level $n'=\lev(\chi)$ and its
Howe correspondent $\chi'$. Since we 
do not need this bound in the sequel, we formulate it only in the case of groups of split type D.

\begin{cor}\label{so-bound-split}
In the situation of Theorem \ref{d-head} with $\varep=+$ we have
$$\chi(1) \geq \chi'(1)\frac{|\GL_{n}(q)|}{q^{n(n-2n')}|\Sp_{2n'}(q)|}.$$
\end{cor}

\begin{proof} 
As in the proof of Corollary \ref{so-rank}, we may assume that $(\chi,\chi')$ appears in
the Howe correspondence for $(G,S)$ with $S = \Sp_{2n'}(q)$ and $n'=\lev(\chi)$.
By \cite[Lemma 4.3]{GLT2}, the restriction of $\chi$ to the Siegel parabolic subgroup $P_n= U \rtimes L$ of $G \in \{\GO^+_{2n}(q), \SO^+_{2n}(q)\}$ 
also has $U$-rank $\mathsf{r}(\chi)=2n'$. The idea is to count the total multiplicity in $\chi$ of $U$-characters $\lambda$ of rank $2n'$. For any such $\lambda$, by \cite[Proposition 4.6]{GLT2}, the restriction of the Weil representation with character $\omega$ of 
$\Sp_{4nn'}(q)$ to $U \times S$, with $S:=\Sp_{2n'}(q)$, contains $\lambda \boxtimes \mathsf{reg}_S$, where $\mathsf{reg}_S$ is the regular representation of
$S$. The $\chi'$-isotypic component of the latter has dimension $\chi'(1)$, showing that the multiplicity of $\lambda$ in $\chi\vert_{U}$ is at least $\chi'(1)$. The 
stabilizer in the Levi subgroup $L$ of $\lambda$ has order $q^{n(n-2n')}|\Sp_{2n'}(q)|$, and hence the statement follows.  
\end{proof}

\begin{proof}[Proof of Theorem C] 
We assume $n_0 \geq 3$ to avoid $G \cong \SL_2(q)$. By \cite[Theorem 5.5]{LT3}, there exists a universal constant $\sigma > 0$ such that 
$|\chi(g)|/\chi(1) < \chi(1)^{-\sigma \supp(g)/N}$. Hence, if $\chi(1) > q^{N/\sigma}$ then $|\chi(g)|/\chi(1) < q^{-m}$. Assume now $\chi(1) \leq q^{N/\sigma}$, i.e.
$$\log_q \chi(1) \leq N/\sigma < N^{4/3}$$
(when $N$ is large enough; also if $G = \Omega^\eps_N(q)$ we replace $G$ by $\SO^\eps_N(q)$ and $\chi$ by an 
irreducible character of $\SO^\eps_N(q)$ above it). Applying \cite[Theorem 1.3]{GLT1} when $G = \SL$, $\SU$, and Theorem \ref{lev-deg} in the orthogonal cases,
we deduce that $\chi$ has bounded level $j \leq n_1$ (for some $n_1 \in \Z$ that depends on $\sigma$). 
In particular, when $N \geq 8n_1^2+6n_1+2$, 
$\chi$ extends to a character $\tilde \chi$ of the same (true) level of 
$\tilde G$, where 
$\tilde G = \GL_N(q)$ if $G=\SL_N(q)$ (by \cite[Proposition 5.10]{KT}), $\tilde G = \GU_N(q)$ if 
$G = \SU_N(q)$ (by \cite[Theorem 3.9]{LOST2}), and $\tilde G = \GO^\eps_N(q)$ in the orthogonal cases (by Corollary \ref{irr-restr}). 
If $j = 1$, or if $j \geq 2$ but $m \geq 2n_1+7$, then we are done by 
Theorems \ref{slu-bound-main}(ii) and \ref{so-bound}(ii). In the remaining case where $2 \leq j \leq n_1$ and $m \leq 2n_1+6$, we 
finish by applying Theorems \ref{slu-bound-main}(i) and \ref{so-bound}(i). Finally, when $N$ is large enough, $\lev(\chi)=1$ is equivalent
to $1 < \chi(1) < q^N$ by the main result of \cite{TZ96}.
\end{proof}

\section{Thompson's conjecture for $\Omega_{2p+1}(q)$}

\begin{lem}\label{cent-size}
Let $q$ be an odd prime power bounded from the above by any constant $q_0$. Then there is a constant $C=C(q_0)$ such that the following statement holds 
for any $n \in \Z_{\geq 1}$. If $g \in G=\GO^\pm_n(q)$ has support $m \geq 2$, then 
$$|\CB_G(g)| \leq Cq^{n(n-5)/2}.$$
If $g \in G$ has support $1$, then $g$ is a reflection up to sign.
\end{lem}

\begin{proof}
By \cite[Proposition 6.4(b)]{LT3}, $|\CB_G(g)| \leq q^{n(n-m)/2+n/6} < q^{n(n-5)/2}$ when $m \geq 6$. 

In the remaining cases $1 \leq m \leq 5$, let $g_{\rm ss}$, respectively $g_{\rm u}$ denote the semisimple part of $g$, respectively the unipotent part of
$g$. Since $q \leq q_0$, by enlarging $C$ if necessary, it suffices for us to show that 
$$D(g) \leq n(n-5)/2+O(1),$$ 
where  $D(g)$ denotes the dimension of the centralizer
of $g$ in the orthogonal group $\cG=\GO_n(\overline{\F_q})$. Suppose first that $g_{\rm ss} \in \ZB(G)$. Then we may assume 
that $g$ is unipotent, and use \cite[Theorem 7.1]{LiSe} to bound $D(g)$. Denoting the Jordan block of size $i$ and eigenvalue $1$
by $J_i$, and writing $g = \oplus_i J_I^{r_i}$ we have 
$$m=\sum_i(i-1)r_i,$$
and furthermore $2|r_i$ whenever $2|i$ \cite[Corollary 3.6]{LiSe}. It follows that $2|s$, and hence we have 
the following possibilities for $g$. 
\begin{itemize}
\item $g=J_1^{n-3}\oplus J_3^1$. Then $m=2$ and $D(g) = n(n-5)/2+4$.
\item $g=J_1^{n-4}\oplus J_2^2$. Then $m=2$ and $D(g) = n(n-5)/2+6$.
\item $g=J_1^{n-5}\oplus J_5^1$. Then $m=4$ and $D(g) = n(n-9)/2+12$.
\item $g=J_1^{n-6}\oplus J_3^2$. Then $m=4$ and $D(g) = n(n-9)/2+14$.
\item $g=J_1^{n-7}\oplus J_2^2 \oplus J_3^1$. Then $m=4$ and $D(g) = n(n-9)/2+16$.
\item $g=J_1^{n-8}\oplus J_2^4$. Then $m=4$ and $D(g) = n(n-9)/2+20$.
\end{itemize}

Now we may assume $g_{\rm ss} \notin \ZB(G)$. In particular, if $m=1$ then $g=g_{\rm ss}$ is a reflection up to sign. 
Also, if $\CB_\cG(g)$ preserves a decomposition of the natural $\cG$-module $V$ into an orthogonal sum of non-degenerate 
$k$-dimensional and $n-k$-dimensional subspaces, with $2 \leq k \leq n-2$, then
$$D(g) \leq k(k-1)/2+(n-k)(n-k-1)/2 \leq n(n-5)/2+4.$$
On the other hand, we have the $\CB_{\cG}(g)$-invariant orthogonal decomposition $V = V_1 \oplus V_{-1} \oplus V_0$, where 
$g_{\rm ss}$ acts as the scalar $\alpha \in \{\pm 1\}$ on $V_{\alpha}$ and has no eigenvalue $\pm 1$ on $V_0$. Hence the previous 
remarks allow us to assume that $\dim V_i \in \{0,1,n-1\}$ for $i \in \{0,\pm 1\}$. Also note that $2|(\dim V_0)$. If moreover $\dim V_0 = n-1$ then 
$$D(g) \leq ((n-1)/2)^2+1 < n(n-5)/2$$ 
when $n \geq 9$. Thus we may now assume that $V_0=0$ and $\dim V_1=n-1$. If $g_{\rm u}$ acts trivially on $V_1$ then $m=1$. 
Otherwise we have $m \geq 2$, and $\CB_{\GO(V_1)}(g)$ has codimension at least $n-O(1)$ in $\GO(V_1)$, and hence 
$$D(g) \leq (n-1)(n-2)/2+1-n+O(1) = n(n-5)/2+O(1).$$
\end{proof}

Let $p \geq 7$ be any prime, and let $q\in \{3,5\}$.
In this section we show that Thompson's conjecture
holds for the simple group $\Omega_{2p+1}(q)$ as long as $p$ is sufficiently large. 
Together with the main result of \cite{EG}, this implies Thompson's conjecture for all sufficiently large odd-dimensional orthogonal simple groups 
in the case that the rank is a prime.

We will work inside the special orthogonal group $G = \SO_{2p+1}(q)$ and consider any element $x \in S=\Omega_{2p+1}(q)$ which has some  fixed
$\xi \in \F_{q^{2p}} \setminus (\F_{q^p} \cup \F_{q^2})$ as an eigenvalue. Note that $\xi$ has degree $2p$ over $\F_q$, and the orbit of $\xi$ under the map
$\lambda \mapsto \lambda^{q}$ has length $2p$ and contains $\xi^{-1}$. It follows that $x$ is regular semisimple; furthermore it is real in $S$ by 
\cite[Proposition 3.1]{TZ}. This element is the one denoted by $t_p$ in \S\ref{sec:unip} with $\lambda=\emptyset$.

\begin{thm}\label{so2p1}
There is an explicit absolute constant $A > 0$ such that the following statement holds when the prime $p$ is at least $A$. If $C$ denotes the 
conjugacy class of the image of the element $x$ described above in the simple group $S = \Omega_{2p+1}(q)$, then
$S=C^2$. 
\end{thm}

The rest of the section is devoted to the proof of Theorem \ref{so2p1}.
The chosen element $x$ has 
$$T:=\CB_G(x) \cong C_{q^{p}+1}.$$ 
Note that $T \cap S$ has index $2$ in $T$; indeed, the subgroup $C_{q+1}$ in $T$ can be identified with the diagonal subgroup of
$$\underbrace{\SO^-_2(q) \times \SO^-_2(q) \times \ldots \times \SO^-_2(q)}_{p~{\rm times}}$$
inside $G$. Hence the conjugacy class $x^G$ is a single $S$-conjugacy class. As $x$ is real,
it suffices to show that any nontrivial element $g \in S$ 
is a product of two $G$-conjugates of $x$.
By the extension \cite[Lemma 5.1]{GT1} of Gow's lemma \cite{Gow} (see also \cite{GT3} for further extensions of it), the statement holds in the case
$g$ is semisimple. So we may assume that $g$ is not semisimple, and furthermore  
$$\supp(g) \geq 2$$
(by the proof of Lemma \ref{cent-size}).

As $x$ is real, by Frobenius' formula it suffices to prove 
\begin{equation}\label{SO-sum10}
  \sum_{\chi \in \Irr(G)}\frac{|\chi(x)|^2\bar\chi(g)}{\chi(1)} \neq 0.
\end{equation}  
The summation in \eqref{SO-sum10} includes two linear characters, which both take value $1$ at all elements of 
$S$. Hence it suffices to show
\begin{equation}\label{SO-sum11}
\Bigm|\sum\nolimits^*\frac{|\chi(x)|^2\chi(g)}{\chi(1)}\Bigm| < 2,
\end{equation} 
where $\sum^*$ indicates the sum over $\chi\in\Irr(G)$ such that $\chi(1)>1$, $\chi(x)\neq 0$, and $\chi(g)\not\in [0,\infty)$.
We denote by $\Sigma^*$ the sum in \eqref{SO-sum11}.

The dual group $G^*$ can be identified with $\Sp_{2p}(q)$. Consider any irreducible character $\chi$ in the Lusztig series $\cE(G,s)$, where 
$s \in G^*$ is a semisimple element. The character formula \cite[Proposition 10.1.12]{DM} shows that the generalized Deligne--Lusztig character
$R^G_{T'}(\theta)$ can be non-vanishing on $x$ only when a $G$-conjugate of $x$ is contained in the maximal torus $T'$ of $G$, which means that
$T'$ is conjugate to $T$. It follows that $\chi(x)$ can be nonzero only when (a conjugate of) $s$ is contained in a torus
$$T^* \cong C_{q^p+1}$$
dual to $T$.

Consider the special case where $L^*:=\CB_{G^*}(s)$ is a (proper) Levi subgroup
of $G^*$. Let $L$ be a Levi subgroup of $G$ dual to $L^*$, and let $\hat{s}$ denote the linear character of $L$ in the Lusztig series 
$\cE(L,(s))$. By \cite[Theorem 11.4.3]{DM},  
\begin{equation}\label{SO-sum12}
  \chi = \pm R^G_L(\hat{s}\psi),
\end{equation} 
where $\psi$ is a unipotent character of $L$. Let $\St_G$ and $\St_L$ denote the Steinberg character of $G$, respectively of
$L$. By \cite[Corollary 10.2.10(ii)]{DM},
\begin{equation}\label{SO-sum13}
  \St_G\cdot \chi = \pm \Ind^G_L(\hat{s}\psi\cdot\St_L).
\end{equation}  

Assume now that $\chi(x) \neq 0$. As $x$ is regular semisimple, 
$$\St_G(x) = \pm 1.$$
Now using the primality of $p$, we see that 
the characters $\chi \in \Irr(G)$ that are non-vanishing at $x$ are divided into the following types:

\begin{enumerate}[\rm(I)]
\item Unipotent characters, multiplied possibly by a linear character of $G$ (here $s \in \ZB(G^*)$).
\item Characters $\chi$ in \eqref{SO-sum12} with $s$ belonging to the unique subgroup $C_{q+1}$ of $T^*$, but not in $\ZB(G^*) \cong C_2$.
For any such $s$, $\CB_{G^*}(s) = L^* \cong \GU_p(q)$, a Levi subgroup of $G^*$.
\item Characters $\chi$ in \eqref{SO-sum12} with $s \in T^* \setminus C_{q+1}$.
Any such $s$ is regular semisimple, with $\CB_{G^*}(s) =T^*$.
\end{enumerate} 

Let $m$ denote the support of $g$, which is also $\supp(zg)$ for any $z \in \ZB(G)$. We will use the character bound \cite[Theorem 5.5]{LT3}
\begin{equation}\label{SO-lt-bound}
  \frac{|\chi(g)|}{\chi(1)} \leq \chi(1)^{-\sigma m/N}
\end{equation}
for an explicit absolute constant $\sigma > 0$, and 
$$N:=2p+1.$$

\begin{lem}\label{SO-type1}
There is an absolute constant $N_1>0$ such that when $N \geq N_1$, the total contribution of characters of type {\rm (I)} to $\Sigma^*$ has absolute value 
$< 1$. 
\end{lem}

\begin{proof}
It suffices to prove that the contribution of unipotent characters $\chi$ of type {\rm (I)} to $\Sigma^*$ has absolute value 
$< 1/2$. By Proposition \ref{unip-regBC}, for each level $0 < j < N$ there are at most two unipotent characters $\chi$ of level $j$ that do not vanish at $x$, 
furthermore, 
$$|\chi(x)| = 1.$$
  
First we consider those $\chi$ with $\chi(1) \geq q^{N^{5/4}}$. 
Then by \eqref{SO-lt-bound} the total contribution of 
these unipotent characters to $|\Sigma^*|$ is at most
$$2Nq^{-\sigma N^{1/4}} < 1/6$$
when $N \geq N_1$ and $N_1$ is chosen large enough.

Now we look at the characters $\chi$ with $\chi(1)  < q^{N^{5/4}}$.
Choosing $N_1$ sufficiently large, when $N \geq N_1$ we have $1 \leq j=\lev(\chi) \leq N^{1/4}$ by Theorem \ref{lev-unip}. 
Suppose that 
$$m=\supp(g) \geq N^{1/4}.$$ 
Since $\chi(1) \geq q^{N/2}$, by \eqref{SO-lt-bound}, the total contribution of these unipotent characters to $|\Sigma^*|$ is at most
$$2N^{1/4}q^{-N^{1/4}\sigma/2} < 1/6$$
when $N_1$ is large enough, yielding the claim in this case.

In the remaining case, we have $m \leq N^{1/4}$ and $j=\lev(\chi) \leq N^{1/4}$. Note that any linear character of $G$ takes value $1$ at 
$g \in S$, hence $\chi(g)$ is not changed when we multiply $\chi$ by such a character. Therefore, by Theorem~\ref{mainC},
$$\frac{\chi(g)}{\chi(1)}=\e_\chi + \beta q^{-mj},$$
where $\beta= \beta(\chi,g) = \pm 1$ and $|\e_\chi| < q^{-N/4}$. As $m \geq 2$ and $q \geq 3$, we have
$$\Bigm| \sum_{j \geq 1}q^{-mj}\Bigm| \leq \sum_{j \geq 1}9^{-j} = \frac{1}{8}.$$ 
Choosing $N_1$ large enough, we also have
$$\Bigm| \sum_{1 \leq j \leq N^{1/4}}\eps_\chi\Bigm| \leq 3^{-N/4}N^{1/4} < \frac{1}{24}.$$
It follows that the contribution of these unipotent characters to $|\Sigma^*|$ is less than 
$$2\bigl(\frac 18+\frac 1{24}\bigr) = \frac13,$$ 
and so we are done in this case as well.
\end{proof}

\begin{lem}\label{SO-type2}
There is an absolute constant $N_2>0$ such that when $N \geq N_2$, the total contribution of characters of type {\rm (II)} to $\Sigma^*$ has absolute value 
$< 1/2$. 
\end{lem}

\begin{proof}
Each of these characters can be written in the form \eqref{SO-sum12}, where $L=\GU_p(q)$ is unique up to conjugacy in $G$. We may assume that 
$x \in L$. Note that $L = \CB_G(t)$, where $t$ has order $q+1$. Thus $t$ is an element of order $q+1$ in $T = \CB_G(g) \cong C_{q^p+1}$. It follows 
that $\langle t \rangle = \langle x^A \rangle$, where $A:=(q^p+1)/(q+1)$, and without loss we may assume $t=x^A$. 
We claim that there are at most $q|L|$ elements $h \in G$ such 
that $hxh^{-1} \in L$. Indeed, for any such $h$, $\langle t \rangle$ is the unique $C_{q+1}$ subgroup in $\CB_G(hxh^{-1})$, which is 
$\langle hx^Ah^{-1} \rangle$. It follows that there is some $1 \leq i \leq q$ coprime to $q+1$ such that 
$$t^i=x^{Ai} = hx^Ah^{-1}=hth^{-1}.$$
For any fixed $i$, any two choices for such $h$ will differ by some element in $\CB_G(t)=L$, yielding the claim.

Next, for any such $h$, $hxh^{-1}$ is a regular semisimple element in $L$ labeled by the $p$-cycle element in the Weyl group of type $A_{p-1}$.
Hence, by \cite[Corollary 3.1.2]{LST}, there are exactly $p=(N-1)/2$ unipotent characters $\psi$ of $L$ that are nonzero at $hxh^{-1}$ (namely the ones labeled 
by hook partitions of $p$), and then 
$\psi(hxh^{-1}) = \pm 1$. It follows from \eqref{SO-sum12} that 
$$|\chi(x)| \leq q \leq 5.$$

There are at most $(q+1)-2 \leq 4$ choices for the element $s$ (up to conjugacy). 
Thus there are fewer than $2N$ characters of type (II) with $\chi(x) \neq 0$. Each of them has degree
$$\chi(1) \geq [G:L]_{q'} = (q-1)(q^2+1)(q^3-1) \ldots (q^{p}-1) > q^{(N^2-9)/8}.$$
Hence, by \eqref{SO-lt-bound}, the total contribution of 
these characters to $|\Sigma^*|$ is at most
$$50Nq^{-\sigma N/9} < \frac 12$$
when $N \geq N_2$ and $N_2$ is large enough.
\end{proof}

To complete the proof of \eqref{SO-sum11} and of Theorem \ref{so2p1}, it remains to prove:

\begin{lem}\label{type3}
There is an absolute constant $N_3>0$ such that when $N \geq N_3$, the total contribution of characters of type {\rm (III)} to $\Sigma^*$ has absolute value $< 1/2$. 
\end{lem}

\begin{proof}
Recall that $\supp(g) \geq 2$, so by Lemma \ref{cent-size} (with $q_0=5$) there exists an absolute constant $C > 0$ such that
$$|\chi(g)| \leq \sqrt{|\CB_G(g)|} \leq Cq^{p^2-1.5p}$$
for any $\chi \in \Irr(G)$. 

Each $\chi$ of type (III) can be written in the form $\pm R^G_T(\theta)$ for some $\theta \in \Irr(T)$. Hence, by \cite[Proposition 7.1]{LTT}
we have the bound
$$|\chi(x)| \leq 2p,$$
and furthermore
$$\chi(1) = [G:T]_{q'} = (q^2-1)(q^4-1) \ldots (q^{2p}-1)/(q^{p}+1) > q^{p^2-1}.$$
There are fewer than $q^p$ choices for $\theta$. Hence the total contribution of these characters to $|\Sigma^*|$ is at most
$$4p^2Cq^pq^{-1.5p+1}  \leq 4p^2Cq^{1-0.5p}< \frac12$$
when $N \geq N_3$ and $N_3$ is large enough.
\end{proof}

\section{Thompson's conjecture for $P\Omega^\pm_{2p+2}(q)$}

Let $p \geq 7$ be any prime, and let $q\in \{3,5\}$.
In this section we show that Thompson's conjecture
holds for the simple groups $P\Omega^\varep_{2p+2}(q)$, $\varep=\pm$,  as long as $p$ is sufficiently large. 
Together with the main result of \cite{EG} and \cite[Theorem 7.7]{LT3}, this implies Thompson's conjecture for all sufficiently large even-dimensional orthogonal simple groups 
in the case that the rank is a prime plus one.

We begin with some preliminary facts.

\begin{lem}\label{rational}
Let $\cG$ be a connected reductive algebraic group over a field of positive characteristic
with a Steinberg endomorphism $F$, and let
$(\cG^*,F^*)$ be dual to $(\cG, F)$. Consider the finite group $G = \cG^F$ and its dual group $G^* = (\cG^*)^{F^*}$. Suppose that 
$\chi$ belongs to the rational Lusztig series $\cE(G,(s))$ labeled by a semisimple element $s \in G^*$. If $g \in G$ is a semisimple 
element of order coprime to $|s|$, then $\chi(g)$ is rational.
\end{lem}

\begin{proof}
By \cite[Lemma 4.1]{TZ2}, on restriction to the set of semisimple elements of $G$, $\chi$ is a $\Q$-linear combination of 
the Deligne--Lusztig characters $R_{\cT,\theta}$ that belong to $\cE(G,(s))$. The latter condition implies by \cite[Lemma 2.1]{H}
that the character $\theta$ of the subgroup $\cT^F$ has order $m:=|s|$. Since the order $|g|$ is coprime to $m$, the character formula
for $R_{\cT,\theta}$ shows that it takes a rational value at $g$ (see e.g. part 1) of the proof of \cite[Theorem 4.2]{TZ2}).
\end{proof}

\begin{lem}\label{congr}
Let $\ell$ be a prime, $G$ a finite group, and $g \in G$ any element of order $\ell$. Suppose that $\chi$ is a complex (not necessarily irreducible) character
of $G$ and $\chi(g) \in \Z$. Then $\chi(g) \equiv \chi(1) \pmod{\ell}$. If in addition $|\chi(g)| < \ell$, then $\chi(g)=0$ if and only if $\ell|\chi(1)$.
\end{lem}

\begin{proof}
Let $\zeta \in \C^\times$ denote a primitive $\ell^{\mathrm {th}}$ root of unity. Then $\chi(g) \equiv \chi(1) \pmod{(1-\zeta)}$ in $\Z[\zeta]$, whence 
the first statement follows. In particular, if $\chi(g)=0$, then $\ell|\chi(1)$. Conversely, if $\ell|\chi(1)$, then $\ell|\chi(g)$, and hence the second statement follows. 
\end{proof}

%

For a given $\varep=\pm$, we will work inside the special orthogonal group $G = \SO^\varep_{2p+2}(q)$. Again according to \cite{Zs}, $q^{2p+2}-1$ admits a
primitive prime divisor $\ell$ (i.e. one that does not divide $\prod^{2p+1}_{i=1}(q^i-1)$).
Consider an element 
$$x \in \Omega^-_{2p}(q) \hookrightarrow S=\Omega^\varep_{2p+2}(q)$$ 
of order $\ell$, which has some fixed
$\xi \in \F_{q^{2p}} \setminus (\F_{q^p} \cup \F_{q^2})$ of order $\ell$ as an eigenvalue. 
Note that $\xi$ has degree $2p$ over $\F_q$, and the orbit of $\xi$ under the map
$\lambda \mapsto \lambda^{q}$ has length $2p$ and contains $\xi^{-1}$. 
It follows that $x$ is regular semisimple in $S$; furthermore it is real in $S$ by 
\cite[Proposition 3.1]{TZ}. This element is the one denoted by $t_{p}$ in \S\ref{sec:unip} (with $\lambda=\emptyset$ or $(1)$ depending on whether 
$\varep=+$ or $\varep=-$). Moreover, in the terminology of \cite[\S4]{LTT}, $x$ has the {\it signed cycle type} $(2,-2p)$; in particular, it has (two)
pairwise distinct cycles.

\begin{thm}\label{so2p2}
There is an explicit absolute constant $A > 0$ such that the following statement holds when the prime $p$ is at least $A$ and $\varep=\pm$. If $C$ denotes the 
conjugacy class of the image of the element $x$ described above in the quasisimple group $S = \Omega^\varep_{2p+2}(q)$, then
$C^2=S \smallsetminus \{-\mathrm{Id}\}$ if $\varep=+$ and $C^2=S$ if $\varep=-$. 
\end{thm}

The rest of the section is devoted to the proof of Theorem \ref{so2p2}.
The chosen element $x$ has 
$$T:=\CB_G(x) = T_1 \times T_2 \cong C_{q^{p}+1} \times C_{q+\varep}.$$ 
Note that $G = S \times \langle -\mathrm{Id} \rangle$ and $\ZB(S)=1$ if $\varep=-$. If $\varep=+$, then 
$\ZB(G)=\ZB(S) = \langle -\mathrm{Id} \rangle$. Next, 
$T \cap S$ has index $2$ in $T$. More concretely, the first direct factor $T_1 \cong C_{q^p+1}$ in $T$ has a unique cyclic subgroup $C_{q+1}$ which 
can be identified with the diagonal subgroup $T_{11}$ of
$$\underbrace{\SO^-_2(q) \times \SO^-_2(q) \times \ldots \times \SO^-_2(q)}_{p~{\rm times}}$$
inside $\SO^-_{2p}(q)$, and the second direct factor $T_2 \cong C_{q+\varep}$ can be identified with $\SO^{-\varep}_2(q)$, where we embed
naturally $\SO^-_{2p}(q) \times \SO^{-\varep}_2(q)$ in $G$.
Hence the conjugacy class $x^G$ is a single $S$-conjugacy class. As $x$ is real,
it suffices to show that any element $g \in S \smallsetminus \ZB(G)$ 
is a product of two $G$-conjugates of $x$.
By the extension \cite[Lemma 5.1]{GT1} of Gow's lemma \cite{Gow} (see also \cite{GT3} for further extensions of it), the statement holds in the case
$g$ is semisimple. So we may assume that $g$ is not semisimple, and furthermore  
$$\supp(g) \geq 2$$
(by the proof of Lemma \ref{cent-size}).

As $x$ is real, by Frobenius' formula it suffices to prove 
\begin{equation}\label{sum20}
  \sum_{\chi \in \Irr(G)}\frac{|\chi(x)|^2\bar\chi(g)}{\chi(1)} \neq 0.
\end{equation}  
The summation in \eqref{sum20} includes two linear characters, which both take value $1$ at all elements of 
$S$, and a priori the sum is a rational number. Hence it suffices to show
\begin{equation}\label{SO-sum21}
 \Re\Bigm(\sum\nolimits^*\frac{|\chi(x)|^2\chi(g)}{\chi(1)}\Bigm) > -2,
\end{equation} 
where $\sum^*$ indicates the sum over $\chi\in\Irr(G)$ such that $\chi(1)>1$ and $\chi(x)\neq 0$.
We denote by $\Sigma^*$ the sum in \eqref{SO-sum21}.

Note that the dual group $G^*$ can be identified with $G$, and we will do so for the rest of the proof. 
Consider any irreducible character $\chi$ in the Lusztig series $\cE(G,s)$, where 
$s \in G$ is a semisimple element. The character formula \cite[Proposition 10.1.12]{DM} shows that the generalized Deligne--Lusztig character
$R^G_{T'}(\theta)$ can be non-vanishing on $x$ only when a $G$-conjugate of $x$ is contained in the maximal torus $T'$ of $G$, which means that
$T'$ is conjugate to $T$. It follows that $\chi(x)$ can be nonzero only when (a conjugate of) $s$ is contained in a torus
$$T^* \cong C_{q^p+1} \times C_{q+\varep}$$
dual to $T$. Without loss of generality, and using the identification of $G^*$ with $G$, we may assume that $s \in T$.
Using the aforementioned decomposition $T = T_1 \times T_2$, we will write 
$$s = \diag(s_1,s_2)$$ 
with $s_i \in T_i$ for $i=1,2$.

Consider the special case where $L:=\CB_{G}(s)$ is a (proper) Levi subgroup
of $G$, which happens precisely when $|s| > 2$,
and let $\hat{s}$ denote the linear character of $L$ in the Lusztig series 
$\cE(L,(s))$. By \cite[Theorem 11.4.3]{DM},  is 
\begin{equation}\label{sum22}
  \chi = \pm R^G_L(\hat{s}\psi),
\end{equation} 
where $\psi$ is a unipotent character of $L$. Let $\St_G$ and $\St_L$ denote the Steinberg character of $G$, respectively of
$L$. By \cite[Corollary 10.2.10(ii)]{DM},
\begin{equation}\label{sum23}
  \St_G\cdot \chi = \pm \Ind^G_L(\hat{s}\psi\cdot\St_L).
\end{equation}  

Assume now that $\chi(x) \neq 0$. As $x$ is regular semisimple, 
$$\St_G(x) = \pm 1.$$
Now using the primality of $p$, we see that 
the characters $\chi \in \Irr(G)$ that are non-vanishing at $x$ are divided into the following types:

\begin{enumerate}[\rm(I)]
\item Unipotent and quadratic-unipotent characters, multiplied possibly by a linear character of $G$. Here we have $|s| \leq 2$.
\item Characters $\chi$ in \eqref{sum22}, where $s_1 \in T_{11}$ has order at most $2$ but $s_2 \in T_2$ has order larger than $2$.
For any such $s$, $L=\CB_{G}(s) \cong \SO^-_{2p}(q) \times T_2$, a Levi subgroup of $G$.
\item Characters $\chi$ in \eqref{sum22} with $s_1$ belonging to the unique subgroup $T_{11} \cong C_{q+1}$ of $T_1$, but $|s_1| >2$.
For any such $s$, $s_1$ has spectrum $(\alpha, \ldots,\alpha,\alpha^{-1}, \ldots, \alpha^{-1},1,1)$ on the natural module 
$V=\F_q^{2p+2}$ for $G$, where $\alpha \in \C^\times$ and $\alpha^{q+1}=1 \neq \alpha^2$. Furthermore, the spectrum of $s_2$ on 
$V$ is $(1, \ldots,1,\beta,\beta^{-1})$, where $\beta \in \C^\times$ and $\beta^{q+\varep}=1$. If $\beta \notin \{\alpha,\alpha^{-1}\}$ then 
$L=\CB_{G}(s) \cong \GU_p(q) \times T_2$. If $\beta \in \{\alpha,\alpha^{-1}\}$ (which can happen only when $\varep=-$), then 
$L=\CB_G(s) \cong \GU_{p+1}(q)$. In either case, $L$ is a Levi subgroup of $G$.
\item Characters $\chi$ in \eqref{sum22} with $s_1 \in T_1 \setminus T_{11}$.
Any such $s$ is regular semisimple, with $\CB_{G}(s) =T$.
\end{enumerate} 

Let $m$ denote the support of $g$, which is also $\supp(zg)$ for any $z \in \ZB(G)$. We will use the character bound \cite[Theorem 5.5]{LT3}
\begin{equation}\label{SO-lt-bound2}
  \frac{|\chi(g)|}{\chi(1)} \leq \chi(1)^{-\sigma m/N}
\end{equation}
for an explicit absolute constant $\sigma > 0$, and 
$$N:=2p+2.$$

\begin{prop}\label{d-type1}
There is an absolute constant $N_1>0$ such that when $N \geq N_1$, the total contribution of characters of types {\rm (I)} and 
{\rm (II)} to $\Sigma^*$ has real part greater than $-1$. 
\end{prop}

\begin{proof}
(a) Recall that replacing $s$ by $-s$ is equivalent to tensoring $\chi$ with the unique linear character of $G$ of order $2$. 
Hence it suffices to prove that the contribution to $\Sigma^*$ of characters $\chi$ that correspond to $s$ with $s_1=1$ has real part
greater than $-1/2$. 

Here we have three cases: $s_2=1$, $|s_2|=2$, and $|s_2| > 2$. In all cases, $|s|$ divides $q^2-1$ and so is coprime to $\ell$. Hence by Lemma \ref{rational} we have
$$\chi(x) \in \Z.$$

\smallskip
(a1) In the first case, $\chi$ is unipotent, and, by Proposition \ref{unip-regD}, for each level $0 < j \leq p$ there are $c_j \leq 2$ two unipotent characters $\chi$ of level $j$ that do not vanish at $x$ ($c_j=1$ if $j=1$ and $c_j \leq 2$ if $j>1$);
furthermore, 
\begin{equation}\label{for-x10}
  |\chi(x)| = 1.
\end{equation}  

\smallskip
(a2) In the second case, $\chi$ is quadratic-unipotent, and $C=C_G(s)$ contains 
$$D=\SO^-_{2p}(q) \times \SO^{-\varep}_2(q)$$ 
as a subgroup of index $2$. As noted above, $x$ has $2$ distinct cycles. Since $\chi$ is quadratic-unipotent, it follows from \cite[Corollary 8.2]{LTT} that 
 $|\chi(x)| \leq 2^{11}$.
Recall that $\ell$ is a primitive prime divisor of $q^{2p+2}-1$, which implies that $\ell \geq 2p+3=N+1$. Taking $N> 2^{11}$, we see
that $\ell > 2^{11}+1$, and hence 
\begin{equation}\label{for-x14}  
  |\chi(x)| < \ell-1.
\end{equation}  
It follows from Lemma \ref{congr} that $\chi(x) \neq 0$ if and only if $\ell \nmid \chi(1)$.

By Lusztig's classification, $\chi$ is labeled by a unipotent character $\tilde\psi$ of $C$, and 
\begin{equation}\label{for-x11}
  \chi(1)=[G:C]_{q'}\tilde\psi(1).
\end{equation}
As $\ell$ divides $q^p+1$ we have 
$$q^{p+1}-\varep \equiv -(q+\varep) \pmod{\ell},~~q^p-1 \equiv -2 \pmod{\ell}.$$
But $\ell \nmid 2(q+\varep)$, so we see that 
\begin{equation}\label{for-x12}
  [G:C]_{q'} = \frac{(q^{p+1}-\eps)(q^p-1)}{2(q+\varep)} \equiv 1 \pmod{\ell}.
\end{equation}
It follows that $\chi(x) \neq 0$ exactly when $\ell \nmid \tilde\psi(1)$.  Any unipotent
character $\tilde\psi$ of $C$ is an irreducible constituent of $\Ind^C_D(\psi)$, where $\psi$ is a unipotent character of $D$, so in fact its kernel contains 
$\SO^{-\varep}_2(q)$, and we can write
\begin{equation}\label{for-x12a}
  \psi = \psi_1 \boxtimes 1_{\SO^{-\varep}_2(q)},
\end{equation}  
where $\psi_1$ is a unipotent character of $\SO^-_{2p}(q)$.  By \cite[Theorem 2.5]{Ma}, $\psi_1$ is $C$-invariant, and hence extends to $C$. Thus each $\psi_1$ gives rise to two unipotent characters of $\GO^-_{2p}(q)$, and also two choices of $\tilde\psi$, all of the same degree 
$\tilde\psi(1)=\psi(1)=\psi_1(1)$. Viewing $x$ as an element of order $\ell$ in $\SO^-_{2p}(q)$, we have $\psi_1(x) \in \{0,\pm 1\}$ by Proposition \ref{prime-rank-D}. So $\chi(x) \neq 0$ exactly when 
$\psi_1(x) \neq 0$. Now, if $\psi_1 \in \Irr(\SO^-_{2p}(q))$ has level $j_1 \leq p-2$ and $\chi$ has level $j$, then \eqref{lev-uni} and \eqref{lev-gen} imply that 
$j_1=j-1.$

By Proposition \ref{prime-rank-D}, for each level $0 < j_1 \leq p-1$ there is
a unique unipotent character $\psi_1$ of level $j$ with $\psi_1(x) \neq 0$, and 
furthermore, 
$\psi_1(x) = \pm 1$.
Now applying Lemma \ref{congr}, we see that
$$\tilde\psi(1)=\psi_1(1) \equiv \pm 1 \pmod{\ell}.$$ 
Combining this with \eqref{for-x11} and \eqref{for-x12}, we obtain that
$\chi(1) \equiv \pm 1 \pmod{\ell}.$
Again using Lemma \ref{congr}, we deduce that
$$\chi(x) \equiv \pm 1 \pmod {\ell}.$$
In conjunction with \eqref{for-x14}, this implies that \eqref{for-x10} also holds in this case as well. 
We have also shown that for any fixed level $1 \leq j \leq p-1$, there are two characters $\chi$ of type (I) and level $j$ with
$|s_2|=2$ such that $\chi(x) \neq 0$.
If $p \leq j \leq p+1$, then there are at most four such characters (combining two different values for $j_1$).

\smallskip
(a3) In the third case, $\chi$ is of type (II) and we have
$$C_G(s)=L=D=\SO^-_{2p}(q) \times \SO^{-\varep}_2(q).$$ 
So $\chi$ is written in the form \eqref{sum22}, where $\psi$ is a unipotent character of $D$. In particular, 
\begin{equation}\label{for-x15}
  \chi(1)=[G:D]_{q'}\psi(1) = \frac{(q^{p+1}-\eps)(q^p-1)}{q+\varep}\psi(1).
\end{equation}  
Again write $\psi$ in the form \eqref{for-x12a}.
Using \eqref{sum23} and arguing as in the proof of Lemma \ref{SO-type2}, as well as $q \leq 5$, we see that there is some 
absolute constant $C_1$ such that
$|\chi(x)| \leq C_1$.
Recall that $\ell \geq 2p+3=N+1$. Taking $N \geq C_1+2$, we see
that $\ell \geq C_1+3$ and hence 
\begin{equation}\label{for-x17}  
  |\chi(x)| \leq \ell-3.
\end{equation}  
Arguing as in (a2) or using \eqref{sum23}, we see that the unipotent character
$\psi_1$ of $\SO^-_{2p}(q)$ is non-vanishing at $x$ (viewed as an element in $\SO^-_{2p}(q)$), and moreover, if $\psi_1$ has level $j_1 \leq N-3$ and $\chi$ has level $j$, then 
$j_1=j-1.$

By Proposition \ref{prime-rank-D}, for each level $0 < j_1 \leq p-1$ there is
a unique unipotent character $\psi_1$ of level $j_1$ with $\psi_1(x) \neq 0$.
Furthermore, $\psi_1(x) = \pm 1$ for such a $\psi_1$, and so
$$\psi(1)=\psi_1(1) \equiv \pm 1 \pmod{\ell}$$ 
by Lemma \ref{congr}.
Combining this with \eqref{for-x15} and \eqref{for-x12}, we obtain
$$\chi(1) \equiv \pm 2 \pmod{\ell}.$$
Again using Lemma \ref{congr}, we deduce that
$\chi(x) \equiv \pm 2 \pmod {\ell}$.
In conjunction with \eqref{for-x17}, this implies that 
\begin{equation}\label{for-x18}
  \chi(x) = \pm 2
\end{equation}  
in this case.

Note that the set of $q+\varep-2$ elements of order $>2$ of $T_2$ breaks into $(q+\varep-2)/2$ conjugacy classes in $S$.
We have therefore shown that for any fixed level $1 \leq j \leq p-1$, there are $(q+\varep-2)/2 \leq 2$ characters $\chi$ of type (II) and level $j$ 
such that $\chi(x) \neq 0$ (and $s_1=1$).
If $p \leq j \leq p+1$, then there are at most $q+\varep-2 \leq 4$ such characters (again combining two different values for $j_1$).

In particular, for each level $1 \leq j \leq N$, the number of characters $\chi$ of types (I) and (II) with $\chi(x) \neq 0$ and $s_1=1$ is 
at most $10$, and we always have $|\chi(x)| \leq 2$.

\smallskip  
(b) First we consider those $\chi$ with $\chi(1) \geq q^{N^{5/4}}$. 
Then by \eqref{SO-lt-bound} and \eqref{for-x10}, the total contribution of 
these characters (subject to $s_1=1$) to $|\Sigma^*|$ is at most
$$40Nq^{-\sigma N^{1/4}} < \frac 14$$
when $N \geq N_1$ and $N_1$ is chosen large enough.

Now we look at the characters $\chi$ with $\chi(1)  < q^{N^{5/4}}$.
Choosing $N_1$ sufficiently large, when $N \geq N_1$ we have $1 \leq j=\lev(\chi) \leq N^{1/4} \leq N-2$ by Theorem \ref{lev-unip}. 
Suppose that 
$$m=\supp(g) \geq N^{1/4}.$$ 
Since $\chi(1) \geq q^{N/2}$, by \eqref{SO-lt-bound2}, the total contribution of these characters to $|\Sigma^*|$ is at most
$$40N^{1/4}q^{-N^{1/4}\sigma/2} < \frac14$$
when $N_1$ is large enough, yielding the claim in this case.

\smallskip
(c) In the remaining case, we have $m \leq N^{1/4}$ and $j =\lev(\chi) \leq N^{1/4}$. In particular, 
$g$ has primary eigenvalue $\lambda=\pm 1$. Note that any linear character of $G$ takes value $1$ at 
$g \in S$, hence $\chi(g)$ is not changed when we multiply $\chi$ by such a character. Therefore, by 
Theorem~\ref{mainC},
$$\frac{\chi(g)}{\chi(1)}=\e_\chi + \frac{\chi(\lambda \cdot I)}{\chi(1)}q^{-mj},$$
where $|\e_\chi| < q^{-N/4}$.
Recall that $m \geq 2$ and $q \geq 3$.
Choosing $N_1$ large enough, we have
$$\Bigm| \sum_{1 \leq j \leq N^{1/4}}\eps_\chi|\chi(x)|^2\Bigm| \leq 40\cdot 3^{-N/4}N^{1/4} < \frac{1}{8}.$$
Now, if $\lambda=1$, then the ``main terms'' $\dfrac{\chi(\lambda \cdot I)}{\chi(1)}q^{-mj}$ in $\dfrac{\chi(g)}{\chi(1)}$ are all positive,
and so we are done in this case.

Assume now that $\lambda=-1$.
To determine the sign of $\chi(-I)/\chi(1) = \pm 1$, we use \cite[Proposition 4.5]{NT}, which implies that 
$\chi(-I)/\chi(1)=1$ if and only if $s_2$ lies in 
$$T_2 \cap [G,G] = C_{(q+\varep)/2}.$$

Fix any $1 \leq j \leq N^{1/4}$. First suppose that $4|(q+\varep)$, which implies that the element of order $2$ in $T_2$ lies in $[G,G]$.
Thus $T_2 \cap [G,G]$ gives rise to $c_j \in \{1,2\}$ unipotent characters, two with $|s_2|=2$, and $(q+\varep-4)/4$ characters with
$|s_2| > 2$, all with $\chi(-I)=\chi(1)$. On the other hand, $T_2 \setminus [G,G]$ gives rise to  $(q+\varep)/4$ characters with
$|s_2| > 2$ and $\chi(-I)=-\chi(1)$. Their total contribution of main terms to $\Re(\Sigma^*)$ is 
$$q^{-mj}\bigl( c_j+2+4 \frac{q+\varep-4}4-4  \frac{q+\varep}4\bigr) =(c_j-2)q^{-mj}.$$
Suppose now that $4 \nmid (q+\varep)$, which implies that the element of order $2$ in $T_2$ does not lie in $[G,G]$.
Thus $T_2 \cap [G,G]$ gives rise to $c_j \in \{1,2\}$ unipotent characters and $(q+\varep-2)/4$ characters with
$|s_2| > 2$, all with $\chi(-I)=\chi(1)$. On the other hand, $T_2 \setminus [G,G]$ gives rise to two characters with $|s_2|=2$ and $(q+\varep-2)/4$ characters with
$|s_2| > 2$, all with $\chi(-I)=-\chi(1)$. Their total contribution of main terms to $\Re(\Sigma^*)$ is 
$$q^{-mj}\bigl( c_j+4 \frac{q+\varep-2)}4-2-4  \frac{q+\varep-2}4\bigr) =(c_j-2)q^{-mj}.$$
We have shown that the total contribution of main terms to $\Re(\Sigma^*)$ is $-q^{-m} \geq -1/9$,
and so we are done in this case as well.
\end{proof}

\begin{lem}\label{d-type2}
There is an absolute constant $N_2>0$ such that when $N \geq N_2$, the total contribution of characters of type {\rm (III)} to $\Sigma^*$ has absolute value 
$< 1/2$. 
\end{lem}

\begin{proof}
Each of these characters can be written in the form \eqref{sum22}, where $L=\GU_p(q) \times T_2$ or $\GU_{p+1}(q)$ is unique up to conjugacy in $G$. We may assume that $x \in L$. Arguing as in the proof of Lemma \ref{SO-type2}, we see that there is some absolute constant $C_2>0$ such that 
the number of $h \in G$ such that $hxh^{-1} \in L$ is at most $C_2|L|$.
For any such $h$, $hxh^{-1}$ is a regular semisimple element, either in $\GU_p(q)$ labeled by the $p$-cycle element in the Weyl group of type $A_{p-1}$, or 
in $\GU_{p+1}(q)$ labeled by the $p$-cycle element in the Weyl group of type $A_{p}$.
Hence, by \cite[Corollary 3.1.2]{LST}, there are at most $p=(N-2)/2$ unipotent characters $\psi$ of $L$ that are nonzero at $hxh^{-1}$, and then 
$\psi(hxh^{-1}) = \pm 1$. It follows from \eqref{sum22} that 
$$|\chi(x)| \leq C_2.$$

There are at most $q^2-1 \leq 24$ choices for the element $s$ (up to conjugacy). 
Thus there are fewer than $12N$ characters of type (III) with $\chi(x) \neq 0$. Each of them has degree
$$\chi(1) \geq [G:L]_{q'} \geq (q-1)(q^2+1)(q^3-1) \ldots (q^{p}-1)(q^{p+1}+1) > q^{(N^2-2N-8)/8}.$$
Hence, by \eqref{SO-lt-bound2}, the total contribution of 
these characters to $|\Sigma^*|$ is at most
$$12NC_2^2q^{-\sigma N/9} < \frac12$$
when $N \geq N_2$ and $N_2$ is large enough.
\end{proof}

To complete the proof of \eqref{SO-sum21} and of Theorem \ref{so2p2}, it suffices to prove:

\begin{lem}\label{d-type3}
There is an absolute constant $N_3>0$ such that when $N \geq N_3$, the total contribution of characters of type {\rm (IV)} to $\Sigma^*$ has absolute value $< 1/2$. 
\end{lem}

\begin{proof}
Recall that $\supp(g) \geq 2$, so by Lemma \ref{cent-size} (with $q_0=5$) there exists an absolute constant $C_3 > 0$ such that
$$|\chi(g)| \leq \sqrt{|\CB_G(g)|} \leq C_3q^{N(N-5)/4}$$
for any $\chi \in \Irr(G)$. 

Each $\chi$ of type (III) can be written in the form $\pm R^G_T(\theta)$ for some $\theta \in \Irr(T)$. Hence, by \cite[Proposition 7.1]{LTT}
we have the bound
$$|\chi(x)| \leq 4p=2N-4,$$
and furthermore
$$\chi(1) = [G:T]_{q'} \geq (q^2-1)(q^4-1) \ldots (q^{2p}-1)(q^{p+1}-1)/(q^{p}+1)(q+1) > q^{N(N-2)/4-2}.$$
There are fewer than $q^p(q+1) < q^{N/2+1}$ choices for $\theta$. Hence the total contribution of these characters to $|\Sigma^*|$ is at most
$$(2N-4)^2C_3q^{3-N/4} < 4N^2C_33^{3-N/4}< \frac12$$
when $N \geq N_3$ and $N_3$ is large enough.
\end{proof}

\section{Thompson's conjecture for odd-characteristic orthogonal groups}
In this section we show that Thompson's conjecture
holds for the simple group $\Omega_N(q)$ with $q=3,5$, as long as $N$ is sufficiently large. 
Together with the main result of \cite{EG}, this implies Thompson's conjecture for all sufficiently large orthogonal simple groups in odd characteristic.
Combined with \cite[Theorem 7.7]{LT3} this completes the proof of the orthogonal part of Theorem~\ref{mainA}.
Combined with the unitary part of Theorem~\ref{mainA} in \S7, this finishes the proof of the asymptotic Thompson Conjecture.

\begin{thm}\label{main-B}
Let $q\in\{3,5\}$.
There is an explicit absolute constant $A>0$ such that if $N>A$ and $\varep =\pm$, then there exists a regular semisimple conjugacy class $C$ in 
$S=\Omega^\varep_N(q)$ such that 
$C^2$ contains all the elements in $S$, except the central involution $-\Id$ of $\SO^\varep_N(q)$ if it belongs to $S$.
\end{thm}

\begin{proof}
(a) If $2 \nmid N$, it is convenient for us to write 
$$N=4n+3+2e,$$ 
where we take $e=0$ if $N \equiv 3 \pmod 4$ and $e=3$ if $N \equiv 1 \pmod 4$ (and we can take $\varep$ to be void).
Similarly, if $2 \mid N$ we write 
$$N=4n+4+2e,$$ 
where we take $e=0$ if $N \equiv 0 \pmod 4$ and $e=3$ if $N \equiv 2 \pmod 4$.

By the prime number theorem, if $n$ is sufficiently large we can write 
$$2n+1 = p + 4k$$ 
for some prime $p$ such that 
$$n < p < 3n/2.$$
Choose $\kappa \in \{1,2\}$ such that $2\mid (N-\kappa)$. Then we have 
$$N = (2p+\kappa)+8k+2e.$$

\smallskip
Let $V = \F_q^N$ denote the natural module for $S$, so that $S = \Omega(V)$.
We fix a non-degenerate subspace $W$ of $V$ of codimension $2p+\kappa$
(whence $\dim(W) = 8k+2e$ is even) and of type $\delta$ chosen as follows.
If $e=0$, then $\delta=+$.
If $e=3$, then $\delta=-$ if $q=3$ and $\delta=+$ if $q=5$. 
We will decompose 
$$W = V_2 \oplus V_3,$$
as an orthogonal sum, where $V_2$ is a non-degenerate subspace of dimension $8k$ and of type $+$. If $e=0$, then $V_3=0$;
otherwise $V_3$ is non-degenerate of dimension $2e$ and of type $\delta$.
Let $V_1$ denote the orthogonal complement to $W$ in $V$. If $2|N$, then $V_1$ has type $\gamma =\varep\delta$; if 
$2 \nmid N$ we take $\gamma$ to be void. Then let $S_1$ denote the group 
$$\Omega(V_1)=\Omega^\gamma_{2p+\kappa}(q),$$ 
embedded in $S$ as the subgroup that acts trivially on $W$.

\smallskip
Let $A_1>0$ and $C_1\subset S_1$ be respectively the constant denoted $A$ and the conjugacy class denoted $C$ in Theorem~\ref{so2p1}
when $\kappa=1$, and in Theorem~\ref{so2p2} when $\kappa=2$. 
If $A$ is sufficiently large, $p>A_1$, so these theorems assert the existence of such a $C_1$.
Let $x_1$ be an element of the class $C_1$.

\smallskip
Let $\ell$ be a primitive prime divisor of $q^{4k-2}-1$ \cite{Zs}, and let $y$ be a regular semisimple element of 
$$R=\Omega^+_{4k}(q)$$ 
of order $\ell$ (with eigenvalues of order $\ell$ or $1$ and with centralizer of order $(q^{2k-1}+1)(q+1)$ in $\SO^+_{4k}(q)$, see e.g. \cite[\S2.1]{GT2}), which is then real
by \cite[Proposition 3.1]{TZ}. Note that $-I_{4k} \in R$, so $-y\in R$, and all eigenvalues of $-y$ have order $2\ell$ or $2$. 
We can embed $R \times R$ in 
$$S_2 = \Omega(V_2) \cong \Omega^+_{8k}(q),$$ 
and see that $\diag(y^{-1},-y^{-1})$ is conjugate to $\diag(-y^{-1},y^{-1})$ in $S_2$ by \cite[Lemma~7.3]{LT3}. Hence, for any scalar $\alpha = \pm 1 \in \F_q^\times$, 
we can express 
$\diag(\al I_{4k},\al I_{4k})$ as a product of two conjugates of $\diag(y,-y)$ in $\Omega^+_{8k}(q)$:
$$\diag(I_{4k},I_{4k}) = \diag(y,-y)\cdot\diag(y^{-1},-y^{-1}),~~\diag(-I_{4k},-I_{4k}) = \diag(y,-y)\cdot\diag(-y^{-1},y^{-1}).$$

\smallskip
If $e=0$, we define 
$$x = \diag(x_1,y,-y),$$ 
which is an element of $\Omega^\varep_N(q)=S$, and we define $C$ to be the conjugacy class of $x$ in $S$. 
Note that $x$ is not regular but has two repeated eigenvalues, $1$ with multiplicity $3$ if $2 \nmid N$ and 
$4$ if $2|N$, and $-1$ with multiplicity $2$
(since $4k<p$, and $y$ and $-y$ have no common eigenvalue).  Still, in either case we have 
$$|\CB_S(x)| \leq q^{2n-1}|\GO^-_4(q)| \cdot |\GO^-_2(q)| < q^{2n+8}.$$

Suppose $e=3$, and let $i$ denote a fixed element of $\F_{q^2}^\times$ of order $4$. If $q=3$, we take
$$z=\diag(i,-i,-1,-1,1,1) \in \SO^-_2(3) \times \SO^+_2(3) \times \Omega^+_2(3) < \SO^\delta_6(3).$$
If $q=5$, we take
$$z=\diag(i,-i,-1,-1,1,1) \in \SO^+_2(5) \times \SO^-_2(5) \times \Omega^-_2(5) < \SO^\delta_6(5).$$
Note that, in both cases, $z$ is a regular semisimple and real element in 
$$S_3=\Omega^\delta_6(q)$$ 
which contains $-I$, and furthermore, $z$ is conjugate to $-z$ in $S_3$.
(These statements are obvious if we consider $z$ as an element of $H:=\GO^\delta_6(q)$. Note that $\CB_H(z)$ contains 
a natural subgroup $\GO^{-\delta}_2(q)$ which has index $4$ over $S_3 \cap \GO^{-\delta}_2(q) = \Omega^{-\delta}_2(q)$, and so 
$\CB_H(z)$ has index $4=[H:S_3]$ over $\CB_{S_3}(z)$, whence the statements follow.) It follows that for any scalar $\al = \pm 1 \in \F_q^\times$, 
we can write 
$\al I$ as a product of two $S_3$-conjugates of $z$:
$$I = z \cdot z^{-1},~~-I = -z \cdot z^{-1}.$$
Now we define 
$$x = \diag(x_1,y,-y,z),$$ 
which is an element of $\Omega^\varep_N(q)=S$, and we define $C$ to be the conjugacy class of $x$ in $S$. Again, $x$ is not regular, with two 
repeated eigenvalues: $1$ with multiplicity $5$ if $2\nmid N$ and $6$ if $2|N$, and $-1$ with multiplicity $4$. Hence we have 
$$|\CB_S(x)| \leq q^{2n} \cdot |\GO^-_6(q)| \cdot |\GO^-_4(q)|  < q^{2n+23}.$$

\smallskip
(b) We now show that if $A$ is sufficiently large, $C^2$ contains all elements $g \neq -I$ of $S$.

Let $g$ denote any element of $S$.  We consider first the case that 
$$\supp(g) \le n+1.$$  
Then the primary eigenvalue $\lambda$ of $g$ lies in $\{\pm 1\}$, and, by \cite[Lemma 6.3.4]{LST}, the $\lambda$-eigenspace of $g$ on $V$ contains 
a non-degenerate subspace of codimension $\leq 2n+2 \leq 2p+\kappa-1$. It follows that some $S$-conjugate of $g$ acts as the scalar $\lambda$ on $V_1^\perp$,
i.e. 
$g$ is conjugate in $S$ to an element of the form 
$$\diag(g_1,\lambda I_{8k},\lambda I_{2e}).$$  
Moreover, $\lambda I_{8k} \in S_2$, $\lambda I_{2e} \in S_3$, so $g_1 \in S_1$. As mentioned above, we can express $\lambda I_{8k}$ as a product of two conjugates of $\diag(y,-y)$ in $S_2=\Omega^+_{8k}(q)$, 
and we can express $\lambda I_{2e}$ as a product of two conjugates of $z$ in $S_3=\Omega^\delta_{2e}(q)$.
If $2 \nmid N$, then by Theorem~\ref{so2p1}, $g_1$ is a product of two $S_1$-conjugates of $x_1$, and so $g$ itself lies in $C^2$.
By Theorem~\ref{so2p2}, the same holds if $2|N$ and $\lambda=1$. Suppose $2|N$ and $\lambda=-1$. By assumption, $g \neq -I$, so $g_1 \neq -I$, and 
hence we are again done by Theorem~\ref{so2p2}.

\smallskip
We may therefore assume 
$$\supp(g)\ge n+1>\frac N5.$$  
By \cite[Theorem~5.5]{LT3}, for every irreducible character $\chi$ of $S$,
$$\frac{|\chi(g)|}{\chi(1)} \le \chi(1)^{-\sigma/5},$$
where $\sigma>0$ is an absolute constant.  By \cite[Theorem~6.2]{GLT2}, there exists $\nu > 0$ such that 
if the centralizer of $t$ in $S$ is smaller than $|S|^\nu$, then $|\chi(t)| \le \chi(1)^{\sigma/11}$. 
Taking $A$ large enough, we may assume that $|\CB_S(x)| < |S|^\nu$, and hence 
$$|\chi(x)| \le \chi(1)^{\sigma/11}.$$ 
Therefore,
$$\frac{|\chi(x)^2\chi(g)|}{\chi(1)} \leq \chi(1)^{-\sigma/55}$$
for all $\chi \in \Irr(S)$.
By \cite[Theorem~1.2]{LiSh}, if $A$ is sufficiently large, 
$$\sum_{1_S \neq \chi \in \Irr(S)} \frac{|\chi(x)^2\chi(g)|}{\chi(1)} < 1.$$
%
The Frobenius formula implies that $g \in C^2$ in this case as well.
\end{proof}

\end{document}